\documentclass[reqno]{amsart}
\usepackage[all]{xy}

\usepackage{amsmath}

\usepackage{amscd}

\usepackage{amsthm}

\usepackage{amssymb}

\usepackage{amsfonts}

\usepackage{epsfig}

\usepackage{enumerate}

\usepackage{graphicx}

\usepackage{graphicx, color}

\usepackage[mathcal]{euscript}

\usepackage{mathtools}

\usepackage{mathrsfs}

\usepackage{tikz}
\usetikzlibrary{arrows,decorations.markings}

\usepackage{subfig}

\usepackage{indentfirst} 

\usepackage{xcolor}

\usetikzlibrary{calc, positioning} 

\usepackage[colorlinks=true, pdfstartview=FitV, linkcolor=black,
			   citecolor=black, urlcolor=black]{hyperref}
\usepackage{bookmark}
\usepackage[titletoc]{appendix} 
\hypersetup{hypertexnames=false} 

          \theoremstyle{definition}
          \newtheorem{theorem}{Theorem}[section]
          \newtheorem{prop}[theorem]{Proposition}
          \newtheorem{lemma}[theorem]{Lemma}

          \newtheorem{cor}[theorem]{Corollary}

          \newtheorem{defn}[theorem]{Definition}
          \newtheorem{notation}[theorem]{Notation}
          \newtheorem{example}[theorem]{Example}

          \newtheorem{rmk}[theorem]{Remark}

\numberwithin{equation}{section}

          \newcommand{\nc}{\newcommand}
          \nc{\DMO}{\DeclareMathOperator}	
          \nc{\commentout}[1]{}

          \nc{\newnotation}{\nomenclature}
          \nc{\wrap}{\cW}
          \nc{\Tw}{\mathsf{Tw}}
          \nc{\loc}{\mathsf{Loc}}
          \nc{\Top}{Top}
          \nc{\emb}{\mathsf{emb}}
          \nc{\ind}{\mathsf{Ind}}
          \nc{\Ind}{\mathsf{Ind}}
          \nc{\Loc}{\mathsf{Loc}}
          \nc{\Cob}{\mathsf{Cob}}
          \nc{\mul}{\mathsf{Mul}}
          \nc{\fat}{\mathsf{fat}}
          \nc{\cob}{\mathsf{Cob}}
          \nc{\coh}{\mathsf{Coh}}
          \nc{\Liouaut}{\Aut_{\mathsf{Liou}}}
          \nc{\idem}{\mathsf{Idem}}
          \nc{\sets}{\mathsf{Sets}}
          \nc{\near}{\mathsf{near}}
          \nc{\sing}{\mathsf{Sing}}
          \nc{\Sing}{\mathsf{Sing}}
          \nc{\perf}{\mathsf{Perf}}
          \nc{\block}{\mathsf{block}}
          \nc{\ssets}{\mathsf{sSets}}
          \nc{\cmpct}{\mathsf{cmpct}}
          \nc{\compact}{\mathsf{cmpct}}
          \nc{\pwrap}{\mathsf{PWrap}}
          \nc{\coder}{\mathsf{Coder}}
          \nc{\bimod}{\mathsf{Bimod}}
          \nc{\grmod}{\mathsf{GrMod}}
          \nc{\Morita}{\mathsf{Morita}}
          \nc{\morita}{\mathsf{Morita}}
          \nc{\spaces}{\mathsf{Spaces}}
          \nc{\pwrms}{\mathsf{PWrFuk}_{M,S}}
          \nc{\pwrmf}{\mathsf{PWrFuk}_{M,F}}
          \nc{\pwrapmf}{\mathsf{PWrFuk}_{M,F}}
          \nc{\fuk}{\mathsf{Fukaya}}
          \nc{\infwr}{\mathsf{InfWr}}
          \nc{\fukaya}{\mathsf{Fukaya}}
          \nc{\autml}{\mathsf{Aut}_{M,\Lambda}}
          \nc{\fukml}{\mathsf{Fukaya}_{M,\Lambda}}
          \nc{\fukmle}{\mathsf{Fukaya}_{M,\Lambda,\epsilon}}
          \nc{\fukmod}{\wrfukcompact(M)\modules}
          \nc{\lag}{\mathsf{Lag}}
          \nc{\lagm}{\lag_M}
          \nc{\lago}{\lag^o}
          \nc{\lagml}{\lag_{M,\Lambda}} 
          \nc{\lagmle}{\lag_{M,\Lambda,\epsilon}}
          \nc{\Fun}{\mathsf{Fun}}
          \nc{\fun}{\mathsf{Fun}}
          \nc{\vect}{\mathsf{Vect}}
          \nc{\chain}{\mathsf{Chain}}
          \nc{\chainn}{Chain}
          \nc{\wrfuk}{\mathsf{WrFukaya}}
          \nc{\wrfukcompact}{\mathsf{WrFukaya}_{\mathsf{cmpct}}}
          \nc{\pwrfuk}{\mathsf{PWrFukaya}}
          \nc{\inffuk}{\mathsf{InfFuk}}
          \nc{\pwrfukml}{\mathsf{PWrFukaya}_{M,\Lambda}}
          \nc{\inffukml}{\mathsf{InfFuk}_{M,\Lambda}}
          \nc{\nattrans}{\mathsf{NatTrans}}
          \nc{\corres}{\mathsf{Corres}}
          \nc{\fukep}{\fukaya_\Lambda(M,\epsilon)}
          \nc{\fukepop}{\fukaya_\Lambda(M,\epsilon)^{\op}}
          \nc{\lagep}{\lag_\Lambda(M,\epsilon)}
          \DMO{\cyl}{cyl} 
          \nc{\dbcoh}{D^b\mathsf{Coh}}
          \nc{\corr}{\mathsf{Corr}}
          \nc{\Liouauto}{{\Aut^o}}
          \nc{\Liouautb}{\Aut^{b}}
          \nc{\Liouautgr}{{\Aut^{gr}}}
          \nc{\Liouautgrb}{\Aut^{gr,b}}
          \nc{\Fuk}{\mathsf{Fuk}}

          \DMO{\im}{im}
          \DMO{\ev}{ev}
          \DMO{\stable}{Ex}
          \DMO{\inj}{inj}
          \DMO{\fib}{fib}
          \DMO{\conf}{Conf}
          \DMO{\chains}{Chains}
          \DMO{\cochains}{Cochains}
          \DMO{\cone}{Cone}
          \DMO{\Map}{Map}
          \DMO{\ran}{Ran}
          \DMO{\rot}{Rot}
          \DMO{\leg}{Leg}
          \DMO{\imm}{imm}
          \DMO{\adj}{adj}
          \DMO{\symp}{Symp}
          \DMO{\tree}{Tree}
          \DMO{\cube}{Cube}
          \DMO{\deep}{deep}
          \DMO{\back}{back}
          \DMO{\Hoch}{Hoch}
          \DMO{\front}{front}
          \DMO{\flow}{Flow}
          \DMO{\floer}{Floer}
          \DMO{\Maps}{Maps}
          \DMO{\exact}{exact}
          \DMO{\excess}{Excess}
          \DMO{\Decomp}{Decomp}
          \DMO{\decomp}{Decomp}
          \DMO{\collar}{collar}
          \DMO{\yoneda}{Yoneda}
          \DMO{\hamspace}{Ham}
          \DMO{\sympspace}{Symp}
          \DMO{\holomaps}{Holomaps}
          \DMO{\comp}{Comp}
          \DMO{\crit}{Crit}
          \DMO{\test}{{test}}
          \DMO{\sign}{sign}
          \DMO{\topp}{top}
          \DMO{\indx}{Index}
          \DMO{\Break}{Break} 
          \DMO{\zero}{zero} 
          \DMO{\ob}{Ob}
          \DMO{\gr}{Gr} 
          \DMO{\Gr}{Gr} 
          \DMO{\cl}{Cl} 
          \DMO{\grlag}{GrLag}
          \DMO{\Pin}{Pin}
          \DMO{\Graph}{Graph}
          \DMO{\pin}{Pin}
          \DMO{\gap}{Gap}
          \DMO{\Ex}{Ex}
          \DMO{\id}{id}
          \DMO{\End}{End}
          \DMO{\sym}{Sym}
          \DMO{\aut}{Aut}
          \DMO{\Aut}{Aut}
          \DMO{\haut}{hAut}
          \DMO{\hAut}{hAut}
          \DMO{\DK}{DK} 
          \DMO{\poly}{poly} 
          \DMO{\diff}{Diff}
          \DMO{\coll}{coll}
          \DMO{\dist}{dist} 
          \DMO{\coker}{coker} 
          \nc{\kernel}{\ker} 
          \DMO{\sspan}{span}
          \DMO{\hocolim}{hocolim}	
          \DMO{\holim}{holim}
          \DMO{\sk}{sk}
          \DMO{\Symp}{Symp}

          \DMO{\ho}{ho}
          \DMO{\fin}{fin}
          \DMO{\tor}{Tor}
          \DMO{\ext}{Ext}
          \DMO{\ret}{Ret}
          \DMO{\ham}{Ham}
          \DMO{\con}{con}
          \DMO{\leaf}{leaf}
          \DMO{\supp}{supp}
          \DMO{\edge}{edge}
          \DMO{\colim}{colim}
          \DMO{\edges}{edges}
          \DMO{\Image}{image}
          \DMO{\roots}{roots}
          \DMO{\height}{height}
          \DMO{\finmod}{FinMod}
          \DMO{\leaves}{leaves}
          \DMO{\planar}{planar}
          \DMO{\vertices}{vertices}

\nc{\norm}[2]{{ \ensuremath{\|} #1 \ensuremath{\|}}_{#2}}
\nc{\Dbar}[1]{\ensuremath{{\bar{\partial}}_{#1}}}
\nc{\Ce}{\ensuremath{\mathbb{C}}}
\nc{\B}{\ensuremath{\mathbb{B}}}
\nc{\osc}{\operatorname{osc}}
\nc{\leng}{\operatorname{leng}}

          \nc{\lagg}{\lag^{\cG}}
          \nc{\iso}{\mathsf{Iso}}
          \nc{\Set}{\mathsf{Set}}
          \nc{\Ass}{\mathsf{ \bf Ass}}
          \nc{\Mod}{\mathsf{Mod}}
          \nc{\modules}{\mathsf{Mod}}
          \nc{\sset}{\mathsf{sSet}}
          \nc{\liou}{\mathsf{Liou}}
          \nc{\poset}{\mathsf{Poset}}
          \nc{\trno}{T^*\RR^n_{\geq 0}}
          \nc{\spectra}{\mathsf{Spectra}}
          \nc{\tensorfin}{\tensor^{\fin}}
          \nc{\lagptg}{\lag_{pt,pt}^{\cG}}
          \nc{\Fin}{\mathcal{F}\mathsf{in}}
          \nc{\lagnl}{\lag_{N,\Lambda}}
          \nc{\lagmlg}{\lag_{M,\Lambda}^{\cG}}
          \nc{\lagsplit}{\lag^{\mathsf{split}}}
          \nc{\lagktimes}{(\lag^{\dd k})^\times}
          \nc{\lagplanar}{\lag^{\times,\planar}}
          \nc{\Cont}{\text{\rm Cont}}
          \nc{\Ham}{\text{\rm Ham}}
          \nc{\Dev}{\text{\rm Dev}}
          \nc{\Lin}{\text{\rm Lin}}
          \nc{\Int}{\text{\rm Int}}
          \nc{\Hom}{\text{\rm Hom}}
          \nc{\Chord}{\text{\rm Chord}}
          \nc{\nbhd}{\mathcal{N}\text{\rm{bhd}}}

          \nc{\onef}{1_{\fukaya}}
          \nc{\smsh}{\wedge}
          \nc{\un}{\underline}
          \nc{\xto}{\xrightarrow}
          \nc{\xra}{\xto}
          \nc{\tensor}{\otimes}
          \nc{\del}{\partial}
          \nc{\dd}{\diamond}
          \nc{\tri}{\triangle}
          \nc{\bb}{\Box}
          \nc{\into}{\hookrightarrow}
          \nc{\onto}{\twoheadrightarrow}
          \nc{\contains}{\supset}
          \nc{\transverse}{\pitchfork}
          \nc{\uncirc}{\underline{\circ}}
          \nc{\thetacontact}{\theta} 

          \nc{\Jbar}{\overline{J}}
          \nc{\Fbar}{\overline{F}}
          \nc{\delbar}{\overline{\del}}
          \nc{\thetabar}{\overline{\theta}}
          \nc{\omegabar}{\overline{\omega}}
          \nc{\Liou}{\text{\rm Liou}}

          \nc{\Yhat}{\widehat{Y}}

          \nc{\Mliou}{M}

          \nc{\vece}{ {\vec \epsilon}}	
          \nc{\vecd}{ {\vec \delta}}
          \nc{\ov}{\overline}
          \DMO{\op}{op}
          \nc{\opp}{ ^{\op}}

          \nc{\hiro}{\textcolor{blue}}
          \nc{\YG}{\textcolor{orange}}
		  \nc{\eqn}{\begin{equation}}
          \nc{\eqnn}{\begin{equation}\nonumber}
          \nc{\eqnd}{\end{equation}}
          \nc{\enum}{\begin{enumerate}}
          \nc{\enumd}{\end{enumerate}}
          \nc{\beastar}{\begin{eqnarray*}}
          \nc{\eeastar}{\end{eqnarray*}}

\numberwithin{equation}{section}

\def\R{{\mathbb R}}
\def\osc{{\hbox{\rm osc }}}
\def\Crit{{\hbox{Crit}}}

\def\E{{\mathbb E}}
\def\Z{{\mathbb Z}}
\def\C{{\mathbb C}}
\def\R{{\mathbb R}}

\def\N{{\mathbb N}}

\def\11{{\mathbb I}}

\def\Jbar{{\widetilde J}}

\def\delbar{{\overline \partial}}

          \def\cC{\mathcal C}\def\cD{\mathcal D}
          \def\cG{\mathcal G}
          \def\cI{\mathcal I}\def\cJ{\mathcal J}
          
          \def\cS{\mathcal S}
          \def\cW{\mathcal W}\def\c\Mliou{\mathcal \Mliou}

          \def\CC{\mathbb C}

          \def\RR{\mathbb R}
          \def\\Mliou\Mliou{\mathbb \Mliou}
          \def\ZZ{\mathbb Z}

          \def\s\Mliou{\mathsf \Mliou}

          \def\b\Mliou{\mathbf \Mliou}

          \def\f\Mliou{\mathfrak \Mliou}

\def\Jbar{{\widetilde J}}

\def\delbar{{\overline \partial}}

\def\b{\beta}
\def\c{\chi}

\def\f{\phi}

\def\s{\sigma}

\def\CC{{\mathcal C}}

\def\CE{{\mathcal E}}

\def\CI{{\mathcal I}}
\def\CJ{{\mathcal J}}

\def\CL{{\mathcal L}}
\def\CM{{\mathcal M}}

\def\CS{{\mathcal S}}
\def\CT{{\mathcal T}}

\def\grad#1{\,\nabla\!_{{#1}}\,}

\def\darr#1{\raise1.5ex\hbox{$\leftrightarrow$}
\mkern-16.5mu #1}

\def\roughly#1{\raise.3ex\hbox{$#1$\kern-.75em
\lower1ex\hbox{$\sim$}}}

\def\opname#1{\mathop{\kern0pt{\rm #1}}\nolimits}

\def\End{\opname{End}}
\def\dim{\opname{dim}}

\def\dist{\opname{dist}}

\def\grad{\opname{grad}}

\def\supp{\operatorname{supp}}

\def\Dev{\operatorname{Dev}}

\def\leng{\operatorname{leng}}

\def\End{\operatorname{End}}

\def\Aut{\operatorname{Aut}}
\def\coker{\operatorname{Coker}}

\def\Cont{\operatorname{Cont}}
\def\Crit{\operatorname{Crit}}

\def\Sing{\operatorname{Sing}}

\def\Image{\operatorname{Image}}
\def\ev{\operatorname{ev}}
\def\Int{\operatorname{Int}}

\def\ben{\begin{enumerate}}
\def\een{\end{enumerate}}
\def\be{\begin{equation}}
\def\ee{\end{equation}}
\def\bea{\begin{eqnarray}}
\def\eea{\end{eqnarray}}
\def\beastar{\begin{eqnarray*}}
\def\eeastar{\end{eqnarray*}}
\def\bc{\begin{center}}
\def\ec{\end{center}}

\renewcommand{\b}{\beta}

\def\Hoch{{\tt Hoch}}

\def\Cont{\operatorname{Cont}}
\def\Crit{\operatorname{Crit}}

\def\Sing{\operatorname{Sing}}

\def\Ham{\operatorname{Ham}}

\def\Graph{\operatorname{Graph}}
\def\id{\text\rm{id}}

\def\E{\ifmmode{\mathbb E}\else{$\mathbb E$}\fi} 
\def\N{\ifmmode{\mathbb N}\else{$\mathbb N$}\fi} 
\def\R{\ifmmode{\mathbb R}\else{$\mathbb R$}\fi} 
\def\Q{\ifmmode{\mathbb Q}\else{$\mathbb Q$}\fi} 
\def\C{\ifmmode{\mathbb C}\else{$\mathbb C$}\fi} 
\def\H{\ifmmode{\mathbb H}\else{$\mathbb H$}\fi} 
\def\Z{\ifmmode{\mathbb Z}\else{$\mathbb Z$}\fi} 

\def\Hoch{{\tt Hoch}}

\def\Cont{\operatorname{Cont}}
\def\Crit{\operatorname{Crit}}

\def\Sing{\operatorname{Sing}}

\def\Ham{\operatorname{Ham}}

\def\Graph{\operatorname{Graph}}

\def\grad#1{\,\nabla\!_{{#1}}\,}

\def\darr#1{\raise1.5ex\hbox{$\leftrightarrow$}
\mkern-16.5mu #1}

\def\roughly#1{\raise.3ex\hbox{$#1$\kern-.75em
\lower1ex\hbox{$\sim$}}}

\def\opname#1{\mathop{\kern0pt{\rm #1}}\nolimits}

\def\End{\opname{End}}
\def\dim{\opname{dim}}

\def\dist{\opname{dist}}

\def\grad{\opname{grad}}

\def\supp{\operatorname{supp}}

\def\Dev{\operatorname{Dev}}

\def\leng{\operatorname{leng}}

\def\End{\operatorname{End}}

\def\Aut{\operatorname{Aut}}
\def\coker{\operatorname{Coker}}

\def\Cont{\operatorname{Cont}}
\def\Crit{\operatorname{Crit}}

\def\Sing{\operatorname{Sing}}

\def\Image{\operatorname{Image}}
\def\ev{\operatorname{ev}}
\def\Int{\operatorname{Int}}

\begin{document}

\title{Fukaya Category of Infinite-type Surfaces}
\author{Jaeyoung Choi and Yong-Geun Oh}
\date{}

\begin{abstract} In this paper,
we construct a Fukaya category of any infinite type surface whose objects are \emph{gradient sectorial Lagrangians}. This class of
Lagrangian submanifolds is introduced by one of the authors in \cite{oh:intrinsic} which can serve
as an object of a Fukaya category of any Liouville manifold that admits an exhausting proper Morse function, in particular on the Riemann surface of infinite type. We describe a generating set
of the Fukaya category in terms of the end structure of the surface when the surface has countably many limit points in its ideal boundary, the latter of
which can be described in terms of a subset of the Cantor set.
We also show that our Fukaya category is not quasi-equivalent to
the limit of the Fukaya category of surfaces of finite type appearing
in the literature.
\end{abstract}


\thanks{This work is supported by  the IBS project \# IBS-R003-D1.}
\date{February, 2023}

\maketitle

\tableofcontents

\section{Introduction}

Let $M$ be a symplectic surface, i.e., a surface equipped with an area form $\omega$.
If $M$ is noncompact, $M$ is always exact and becomes
a Liouville manifold: If we let $\omega = d\alpha$, we can find a vector field $X$ satisfying
\be\label{eq:liouville-X}
X \rfloor d\alpha = \alpha \quad \text{\rm or equivalently } \, \CL_X\alpha = \alpha.
\ee
If $M$ is of finite type, its Fukaya category is well-understood by now.
(See \cite{heather} for example.)
However, not much has been said about the Fukaya category of infinite type surfaces, except that
Auroux and Smith \cite{auroux-smith} take the definition of a Fukaya category of infinite type surface
to be a colimit of that of finite type surfaces utilizing the operations of
pair of pants decomposition and gluing constructions. Since then, this definition
has prevailed in the literature. (See \cite{pascaleff-sibilla} for example.)

Recently infinite type Riemann surfaces have attracted much interest of researchers
in relation to the study of \emph{big mapping class groups}. Especially in the blogpost \cite{calegari:blogpost},
D. Calarari proposed the study of the mapping class group $\text{\rm Map}(\R^2 \setminus C)$, where
$C$ denotes a Cantor set and posed the question of whether this group has an infinite-dimensional
space of quasimorphisms, as is the case with the mapping class group of a surface of finite
topological type. See \cite{bavard:collar}, \cite{bavard:big-mapping} for some
relevant developments related to the study of big mapping class groups.

In this paper, partially motivated by the study of big mapping
class groups and Calagari's proposal of studying quasimorphisms on the infinite
type surfaces, we will provide a geometric construction of a Fukaya category of infinite type surface
by a direct construction without taking the colimit. We hope that this geometric construction
combined with the study of Lagrangian spectral invariants can be utilized in
some new dynamical approach to the Calagari's question on the space of quasimorphisms
on the big mapping class groups. In relation to this, we refer readers to
\cite{azam-blanchet} for the mapping class group action on the Fukaya category  of surfaces (of finite type).

Another source of possible applications comes from the homological mirror symmetry.
There is an interesting construction of modular forms via the study of
Fukaya category of the \emph{divisor complements} and the mirror symmetry of elliptic curves by
Ueda and his collaborators. During their construction, study of the Fukaya category of
surfaces of infinite type naturally arise as some universal covering of punctured
elliptic curves on which the mapping class group of the latter surface acts.
(See \cite{nagano-ueda}, \cite{hashimoto-ueda} for example.)

\subsection{Hyperbolic Riemann surface structure}

One of difficulties in an attempt to directly construct a Fukaya category
of an infinite type surface lies in the question on how one should handle
the end structure of the surface.
If a surface is of finite type, it automatically has finitely many ends and all ends  are eventually cylindrical, i.e.,
each end is diffeomorphic to $[0, \infty) \times Q$ for some compact manifold.
However, this is not necessarily the case when the surface is of infinite type.
Unlike the case of finite type for which we can \emph{easily prescribe}
the Liouville vector field $X$ to be cylindrical, i.e., to satisfy
\eqref{eq:liouville-X}, there are several things to be made clear before
attempting to construct a Fukaya category on the infinite type surfaces
(or more generally infinite type non-compact symplectic manifolds). Here are
three points of immediate concern:
\begin{enumerate}
\item Since construction of Fukaya category involves study of pseudoholomorphic curves, one needs to make ensure the bulk
admits reasonable geometric analysis of pseudoholomorphic curves.
This means that one should ask some boundedness of relevant Riemannian
metric.
\item The above means that the resulting Fukaya category will depend not
only on the symplectic structure $\omega$ but also on the quasi-conformal class of the  K\"ahler metric associated to $g = \omega(\cdot, J \cdot)$.
Because of this \emph{the set of tame almost complex structures
may not be contractible in $C^\infty$-topology \cite{liu-papado}.
The associated tame metrics may not be quasi-conformally equivalent
which would imply that the set of almost complex structures tame
to symplectic surface $(M,\omega)$ in the standard sense
may not even be path-connected.} This is the reason why
 the invariant arising from the study of pseudoholomorphic
curves on an infinite type surface
 is an invariant of $(M,\CT, \omega)$ instead of $(M,\omega)$.
\item It is not obvious what kind of asymptotic condition on the background symplectic  manifold and the associated Lagrangian submanifolds to put besides its tameness.
\end{enumerate}
 In particular, the point (3)   makes identifying the relevant condition
 for a noncompact  Lagrangian brane as a legitimate object of
 the relevant Fukaya category a nontrivial question.

To resolve these points, we first mention that it is well-known
to the experts that  every noncompact surface admits a
\emph{tame hyperbolic structure}. which induces a Riemann surface
structure $(M,J_0)$ as well as the topology of the surface.
Equipping a $J$-tame almost symplectic form of infinite volume, we
obtain  a K\"ahler structure $(M,\omega, J_0)$.
We will fix  the reference hyperbolic structure on $M$
in the sense of \cite{liu-papado},
which will be used the underlying topology and quasi-conformal structure of
the surface which will be used later in our discussion.

\begin{defn}[Hyperbolic Riemann surface]\label{defn:hyperbolic-structure-intro}
 A hyperbolic Riemann surface
is a triple $(M,J_0, g_0)$ whose universal cover is isometric to the unit disk.
We call it \emph{tame} if it has bounded curvature and its injectivity radius
is positive.
\begin{enumerate}
\item A \emph{hyperbolic structure}, denoted by $\CT = \CT_M$,
of a surface $M$ is a choice of  $(M,J_0, g_0): = (M,J_\CT,g_\CT)$
that is tame.
\item
$\CT$ also determines a symplectic form
$$
\omega_\CT = g_0(J_0 \cdot, \cdot)
$$
which we call $\CT$-symplectic form.
\item We then denote by $\CJ_{\CT}$ the $\omega_\CT$-tame almost
complex structures.
\end{enumerate}
\end{defn}
Now it is easy to check that $\CJ_\CT$ is contractible with respect to
the $C^\infty$ topology of $\End(TM)$ induced from
that of $(M,\CT)$. We then use the associated (almost) K\"ahler metric
$$
g_\CT = \omega_\CT(\cdot, J_\CT \cdot)
$$
for all the relevant geometric estimates appearing in the study of pseudoholomorphic curves
needed in the Floer theory throughout the paper.

\begin{defn}\label{defn:quasiisometry-intro} Two hyperbolic Riemann surfaces $(M,\CT)$ and $(M', \CT')$
are equivalent, if there is a diffeomorphism $\phi: M \to M'$
such that the two hyperbolic structures $\CT$ and $\phi^*\CT'$ on
$M$ are quasi-isometric.  We denote by
$$
\text{\rm QC}(M,\CT)
$$
the automorphism group of $\CT$.
\end{defn}

Thanks to the infinite volume hypothesis, a version of Greene and Shihohama's theorem \cite{Greene-Shiohama}
proves that the two symplectic forms
$\omega_{\CT}$ and $\phi^*\omega_{\CT'}$ are also
symplectomorphic by a quasi-isometric symplectic isotopy.
This will implicates that the Fukaya category we construct in the present
paper is an invariant of quasi-isometric symplectic isotopy.

\begin{defn} Let $\omega_\CT$ be the $\CT$-symplectic form
of a hyperbolic Riemann surface $(M,\CT)$. The automorphism group
of $\omega_\CT$ is the intersection
$$
\Symp(M,\omega_\CT) \cap \text{\rm QC}(M,\CT)
=: \Symp_{\text{\rm QC}}(M,\omega_\CT).
$$
\end{defn}
This takes care of the points (1) and (2) above.

\subsection{Tame Liouville manifolds and gradient-sectorial Lagrangian submanifolds}

To answer to the point (3) above, we recall that since any noncompact
surface is homotopy equivalent to one-dimensional CW-complex, the
symplectic form $\omega_\CT$ is exact, i.e, it can be written as
$$
\omega_\CT = d\alpha_\CT
$$
for some one form $\alpha_\CT$ which is uniquely defined up to the
transformation
$$
\alpha_\CT \mapsto \alpha + dg
$$
for some smooth function $g: M \to \R$. By a suitable such transformation,
we can make the one-from tame in that
$$
\|\nabla^k \alpha\|_{C^0} \leq C(k).
$$
We denote by $X_{\alpha_\CT}$ the Liouville vector field associated to
$\alpha_\CT$, i.e., the unique vector field $X$ determined by
\be\label{eq:XCT}
X \rfloor d\alpha_\CT = \alpha_\CT.
\ee
Furthermore any noncompact Riemann surface admits
a plurisubhamonic exhaustion function $\psi$ which satisfies the inequality
$$
-d(d\psi \circ J) \geq 0
$$
as a $(1,1)$-current, and its sub-level sets compact. Following the term of \cite{oh:intrinsic},
we call the pair $(\psi,J)$ a pseudoconvex pair.
This enables us to take the class of
\emph{gradient sectorial Lagrangians} with respect to a given
\emph{pseudoconvex pair} $(\psi,J)$ (at infinity) as the objects of our Fukaya category. The notion of
\emph{gradient sectorial Lagrangians} in general dimensions is introduced by the second-named author in \cite{oh:intrinsic}
whose definition we recall now.

Now we are ready to introduce the
following notion of $\CT$-tame Liouville manifolds in general.

\begin{defn}[$\CT$-tame Weinstein manifolds]
We call the triple $(M,\CT,\alpha_\CT)$ a
\emph{$\CT$-tame Liouville manifold} if the following hold:
 \begin{enumerate}
 \item There exists a pair $(\psi, J)$ of a Morse function $\psi$ and a $\omega_\CT$-tame
 almost complex structure $J$ such that $(\psi,J)$ is a
  pseudoconvex pair on $M \setminus K$,
i.e., $-d(d\psi \circ J) = g d\alpha$ for some function $g: M \setminus K \to \R$ with
$g \geq 0$.
\item There exists a constant $C = C(k)$ depending on $k$
such that
$$
\|\nabla^k \psi \| _{C^0} \leq C(k)
$$
for all $k\geq 1$.
\item On each cylindrical end, if any, the associated Liouville vector field
$X_{\alpha_\CT}$ is gradient-like
for the function $\psi$, i.e., $\frac{X[\psi]}{|\grad \psi|^2} =: h > 0$ for some
smooth positive function $h > 0$.
\end{enumerate}
\end{defn}
Now we recall the class of \emph{$\psi$-gradient-sectorial Lagrangian branes} with respect to
the given pseudoconvex pair $(\psi,J)$ from \cite{oh:intrinsic}
and define the notion of \emph{$\psi$-wrapped Fukaya category} whose objects are
$\psi$-gradient-sectorial Lagrangian submanifolds:
Consider the normalized gradient vector field of $\psi$ given by
\eqn\label{eq:Z-fraks}
Z_{\psi}: = \frac{\grad \psi}{|\grad \psi|^2}
\eqnd
with respect to the usual metric
$$
g_J(v, w): = \frac{d\alpha(v, J w) + d\alpha(w,Jv)}{2}.
$$
The following definition is a special case of the definition of gradient-sectorial
Lagrangian branes from \cite{oh:intrinsic} restricted to the Liouville manifolds $(M,\alpha)$.

\begin{defn}[Gradient-sectorial Lagrangian branes]
\label{defn:gradient-sectorial-Lagrangian-intro}
Let $(M,\alpha)$ be a Liouville-tame symplectic manifold equipped with a
pseudoconvex pair $(\psi,J)$.
We say that a proper exact Lagrangian submanifold $L$ of $(M,\alpha)$ is
\emph{$\psi$-gradient-sectorial} if there exists a sufficiently large $r_0> 0$ such that
$L \cap {\psi}^{-1}([r_0,\infty))$ is $Z_{\psi}$-invariant, i.e., $Z_{\psi}$ is tangent to
$L \cap {\psi}^{-1}([r_0,\infty))$.
\end{defn}
When $\psi$ is fixed, we will just call it \emph{gradient sectorial} dropping $\psi$.
Also note that once $\psi$ is defined on $M$ and $n\in \ZZ_{>0}$ is not a critical value of $\psi$,
$\psi^{-1}([0,n])$ is a Liouville subdomain of $M$ and denote it as $M^{\leq n}$.
Then $M^{\leq 1}\subset M^{\leq 2}\subset \cdots$ gives us a compact exhaustion of $M$ by Liouville
subdomains.

\subsection{Statements of main theorems}

Once an orientable separable surface without boundary
$M$ is given, we will construct a Liouville-tame symplectic surface $M'$ with a pseudoconvex pair $(\psi,J)$ according to its end structure and topological structure of some compact subset.
Such $M'$ is homeomorphic to $M$ which provides a normal form
of the given hyperbolic structure $\CT$. We will call such surface as a
\emph{standard surface}. (See Subsection \ref{subsec:standard-surface}
for the precise definition.)

\begin{theorem}
Let $(M,\CT)$ be a noncompact hyperbolic Riemann surface without boundary.  Then there is a unique representative modulo quasi-isometry,
which we call a \emph{standard surface} equipped with
a pseudoconvex pair $(J,\psi)$, i.e., with a Weinstein structure.
\end{theorem}

Now we define a Fukaya category on such a tame Weinstein
triple $(M,J,\psi)$.

\begin{theorem}
Let $M'$ be a standard surface with a pseudoconvex pair $(J,\psi)$.
Then we can define a Fukaya category $Fuk(M')$ whose objects are
gradient sectorial Lagrangians.
\end{theorem}

Moreover, such a Fukaya category is well-defined up to quasi-equivalence
as an invariant of the quasi-isometry class of $(M,\omega_\CT)$.  We
 call it a \emph{Fukaya category} of $M$ associated to the hyperbolic
structure $\CT$ and denote by $Fuk(M,\CT)$.

\begin{theorem}
Let $(M,\CT)$ be a separable surface without boundary equipped with
hyperbolic structure $\CT$. If $(M,\CT) \sim (M', \CT')$ in the sense of
\ref{defn:quasiisometry-intro}, then their
Fukaya categories $Fuk(M,\CT)$, $Fuk(M',\CT')$ are quasi-equivalent.
\end{theorem}

We have made clear the condition for the bulk to be able to
define Fukaya category for which we need to equip the infinite-type
surface with the structure of hyperbolic Riemann surface, which
we denote by $\CT$. Since this structure will be fixed from now on,
we drop it from the notations for the various geometric structures
associated to it.

\subsection{Some computation of morphisms of $Fuk(M,\CT)$}

Let us consider the simplest example of infinite type surface $M$:
a surface $M$ with a cylindrical end and an end with infinitely many genus,
i.e. a cylinder $M$ with infinitely many genus on one end.
We can describe such a surface $M$ as a Lefschetz fibration $\pi: M \to C$
over the cylinder $C = \R \times S^1$. Put
the cylindrical coordinate $(s,t)$ on $C$. We ask the fibration to satisfy the following
on each \emph{non-cylindrical end}:
\begin{enumerate}
\item $\pi: M \to C$ is eventually $1$-periodic on each end.
\item The function
$$
s \circ \pi \to \R
$$
restricts to a Morse function on $[k,k+1]$ that carries a unique critical point of index 1
with the value $\frac14, \, \frac34$ respectively.
\item On each cylinder $[k,k+1]$ for $k \in \Z$ sufficiently large
$$
\pi^{-1}(s) \cong \begin{cases}
S^1 \quad & \text{for } \, s \in [k,k+\frac14) \cup (k+\frac34, k+1]\\
S^1\vee S^1  \quad & \text{for } \, s=k+\frac14 \text{ or }  s=k+\frac34\\
S^1 \sqcup S^1 \quad & \text{for } \, s \in (k+\frac14,k+\frac34)
\end{cases}
$$
\item $s\circ \pi$ defines a globally defined Morse function on $M$.
\end{enumerate}
We call an open subset of $M$ a \emph{cylindrical region with genus} if it admits
the aforementioned Lefschetz fibration structure over the cylinder
$$
[k_1, k_2], \quad k_1 < k_2, \quad k_i \in \Z \cup \{\pm \infty\}.
$$

\begin{figure}
\centering

\begin{tikzpicture}

\draw (-5,1) -- (5,1);
\draw (-5,-1) -- (5,-1);

\draw[] (-5,-1) arc (270:90:.2 and 1);
\draw[] (-5,-1) arc (-90:90:.2 and 1);

\draw[densely dashed] (-3,-1) arc (270:90:.2 and 1);
\draw[] (-3,-1) arc (-90:90:.2 and 1);

\begin{scope}[scale=0.8]
\path[rounded corners=24pt] (-.9,0)-- +(0.9,.6)-- +(1.8,0) +(0,0)-- +(0,0.34)--    +(1.8,0);
\draw[rounded corners=28pt] (-1.1,.1)-- +(1.1,-0.7)-- +(2.2,0);
\draw[rounded corners=24pt] (-.9,0)-- +(0.9,.6)-- +(1.8,0);
\end{scope}

\begin{scope}[scale=0.8]
\path[rounded corners=24pt] (2.5,0)-- +(0.9,.6)-- +(1.8,0) +(0,0)-- +(0,0.34)--    +(1.8,0);
\draw[rounded corners=28pt] (2.3,.1)-- +(1.1,-0.7)-- +(2.2,0);
\draw[rounded corners=24pt] (2.5,0)-- +(0.9,.6)-- +(1.8,0);
\end{scope}

\draw[fill=black] (4.5,0) circle (1pt);
\draw[fill=black] (4.7,0) circle (1pt);
\draw[fill=black] (4.9,0) circle (1pt);

\end{tikzpicture}

\caption{Cylinder with infinitely many genus on one end}
\label{fig:cyl_inf_genus}
\end{figure}
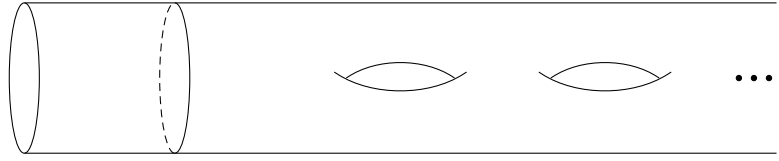

\begin{figure} \centering

\begin{tikzpicture}

\draw[] (-5,-1) arc (270:90:.2 and 1);
\draw[] (-5,-1) arc (-90:90:.2 and 1);

\begin{scope}
\draw[rounded corners=24pt] (-5,1) -- (1,1) --(3,3) -- (5,3);
\end{scope}

\begin{scope}
\draw[rounded corners=24pt] (-5,-1) -- (1,-1) --(3,-3) -- (5,-3);
\end{scope}

\begin{scope}
\draw[rounded corners=28pt] (5,1) -- (3,1) --(3,-1) -- (5,-1);
\end{scope}

\draw[densely dashed] (5,1) arc (270:90:.2 and 1);
\draw[] (5,1) arc (-90:90:.2 and 1);
\draw[densely dashed] (5,-3) arc (270:90:.2 and 1);
\draw[] (5,-3) arc (-90:90:.2 and 1);

\begin{scope}
\draw[blue,rounded corners=24pt] (-4.8,-0.5) -- (1.2,-0.5) --(3.3,-2.5) -- (5.16,-2.5);
\end{scope}

\begin{scope}
\draw[red,rounded corners=4pt] (-4.9,-0.9) -- (-4.5,-0.9) --(-4,1) -- (-3.8,1);
\draw[red,densely dashed, rounded corners=4pt] (-4,1) -- (-3.8,1) --(-3.3,-1) -- (-3.1,-1);
\draw[red,rounded corners=4pt] (-3.3,-1) -- (-3.1,-1)--(-2.6,1) -- (-2.4,1) ;
\draw[red,densely dashed, rounded corners=4pt] (-2.6,1) -- (-2.4,1) --(-1.9,-1) -- (-1.7,-1);
\draw[red,rounded corners=4pt] (-1.9,-1) -- (-1.7,-1) --(-1.5,-0.1)--(-0.2,-0.1)--(-0,-1)--(0.2,-1);
\draw[red,densely dashed, rounded corners=4pt] (-0,-1)--(0.2,-1) --(0.7,1)--(0.8,1.1);
\draw[red,rounded corners=4pt] (0.8,1.1) --(0.9,1.2)-- (1.4,-0.5) --(1.8,-0.7) --(3.1,-0.43)--(3.2,-0.6) ;
\draw[red,densely dashed, rounded corners=4pt] (3.1,-0.43)--(3.2,-0.6) --(2.75,-2.65)--(2.85,-2.75);
\draw[red,rounded corners=4pt] (2.75,-2.7)--(2.9,-2.75)-- (4,-2.1) --(5.22,-2.1);
\end{scope}

\draw[fill=black] (-4.4,-0.5) circle (1pt) node[below right]{$z$};
\draw[fill=black] (-2.98,-0.5) circle (1pt) node[below right]{$z^0$};
\draw[fill=black] (-1.6,-0.5) circle (1pt) node[below right]{$z^{-1}=y_1$};
\draw[fill=black] (-0.1,-0.5) circle (1pt) node[above right]{$x_1$};
\draw[fill=black] (3.53,-2.37) circle (1pt) node[below]{$y_2$};
\end{tikzpicture}

\caption{Hamiltonian flow of a Lagrangian near joining parts}
\label{fig:ham_flow_near_gluing_parts}
\end{figure}
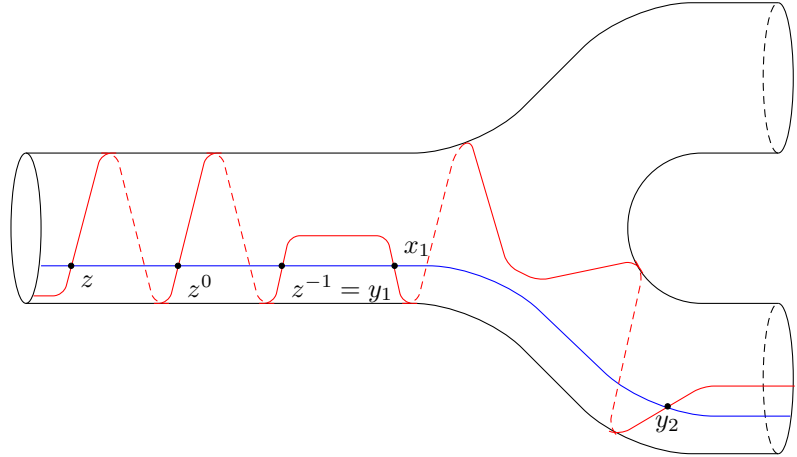

In Section \ref{subsec:standard-surface}, we will define a triple
$$
(\omega,J,\psi)
$$
of an area form $\omega$, a complex structure $J$ and $\psi$ on any given noncompact surface
for which $(J,\psi)$ is a pseudoconvex pair, i.e.,
$$
-d(d\psi \circ J) \geq 0
$$
as a $(1,1)$-current. (See \cite{oh:intrinsic}.)
Then the aforementioned Morse function $\psi$ above will be plurisubharmonic function
with respect to such a complex structure. In particular $(M,\omega)$ becomes a Weinstein manifold
(See \cite{cieliebak-eliashberg} for the definition.)

\begin{defn}[Weinstein triple] Let $M$ be a noncompact manifold. We call a triple $(\omega=d\alpha,J,\psi)$
a  \emph{Weinstein triple} if $(J,\psi)$ is a pseudoconvex pair and the Liouville vector field
$X$ is $\psi$-gradientlike.
\end{defn}

Recall that any Riemann surface can be regarded as a union of disjoint pairs of pants,
joining finite cylinders between them, and cylindrical ends.
On each cylindrical part, we put the area form of the form $\omega= \pi^*(ds \wedge dt)$
where $(s,t)$ is the relevant cylindrical coordinates.

We will then consider Hamiltonians $H$ of the type $H = \kappa \circ \psi$ for a one-variable function
function $\kappa: \R \to \R$. Since we set our Hamiltonian to be only determined by $\psi$,
the Hamiltonian flow of a point $p$ on $M$ should
remain in the connected component of level set $\psi^{-1}(\psi(p))$.

Consider $M$ on the cylindrical region with genus, say, $U \subset M$ with $\pi: U \to C$.
Note that each connected component of level set $\psi^{-1}(x)$ of $\psi$ in $U$ is homotopic to
$$
S^1\vee S^1
$$
 if $x=k+\frac14$ or $x= k+\frac34$ and otherwise homotopic to $S^1$.

We note that since each wedge sum point of $\psi^{-1}(x)$ for $x=k+\frac14$ or $x= k+\frac34$ is
a critical point, and the Hamiltonian flow cannot pass through those points.
This implies that in each level set of critical values of $\psi$, Hamiltonian flow should remain in
each connected open arc which corresponds to a connected components of
$S^1\vee S^1 \setminus \{\text{wedge sum point}\}$ up to homotopy.

Now comes the description of Lagrangian branes in our framework.
In addition to compact Lagrangian branes,
we consider $\psi$-gradient sectorial Lagrangian $L$. We require
$L$ on the cylindrical region with genus $U$ to transversally intersect each regular level set of $\psi$ eventually on
every cylindrical ends with genus. We say such $L$ is in general position, if it
does not pass through critical points outside a compact subset of $M$. We may and will assume that
$L$ is invariant eventually under the gradient flow of $\psi$.

This is illustrated in Figure \ref{fig:ham_flow_near_gluing_parts}.
In each critical level set, the Hamiltonian trajectory of a point cannot wrap around the wedge
$S^1\vee S^1$ and return to the initial position.
If we remove the critical value level sets from $U$, each connected component is homotopic to
a finite cylinder and the intersection
$$
\phi_H^\delta(L) \cap L
$$
with the region always occurs in pair contained in the same cylinder except for the one point in
the critical level $\psi^{-1}(0)$ which is also the critical level of $H$.
Therefore, every Hamiltonian chord contained in the non-cylindrical end comes in pairs
$$
(x^i_j, y^i_j)
$$
each of which forms an \emph{acyclic subcomplex} satisfying the relation
\be\label{eq:acyclic}
\delta x^i_j = 0, \quad \delta y^i_j = x^i_j.
\ee
for the Floer differential $\delta$. This example indicates to us the intuition that the Fukaya algebra
of $L$ is quasi-isomorphic to the subalgebra generated by the chords contained in the cylindrical end, which can be written as the following.

\begin{prop}
Let $L_1,L_2$ be gradient sectorial Lagrangian submanifold of $M$.
Then there exists a Liouville subdomain $M^{\leq n}$ of $M$ such that
$Mor(L_1,L_2)$ is quasi-isomorphic to subcomplex generated by generators contained in $M^{\leq n}$
or isomorphic to $Mor_{Fuk(\widehat{M^{\leq n}})}(\rho_n(L_1),\rho_n(L_2))$ where $\rho_n$ is
a Viterbo restriction functor from $Fuk(M)$ to $Fuk(\widehat{M^{\leq n}})$.
\end{prop}

Since $M^{\leq n}$ gives us a compact exhaustion of $M$ by Liouville subdomains and
the pair $(Fuk(\widehat{M}^i)_{i\geq 1},(\rho_{i,j})_{1\leq i\leq j})$ is an inverse diagram
where $\rho_{i,j}:Fuk(\widehat{M^{\leq j}})\to Fuk(\widehat{M^{\leq i}})$ is a Viterbo restriction functor, we may expect Fukaya category on $M$ to be (quasi-)equivalent to the inverse limit
$\mathcal{X}$ of this diagram. However, that is the case only when $M$ is of finite type.

\begin{theorem}
The $A_\infty$-functor $\nu:Fuk(M) \to \mathcal{X}$ is a quasi-equivalence if and only if $M$ is of finite type.
\end{theorem}

Organization of the paper is now in order.
In Part 1, we summarize Richard's classification result and explain the scheme of
a construction of surfaces. We also provide theoretical foundation to define our Fukaya category using the gradient sectorial Lagrangian branes of \cite{oh:intrinsic}.
In Part 2, we will give a construction of the aforementioned Fukaya category.
In Part 3, we will describe generators and algebraic structure of Fukaya category.

\part{Preliminaries}

In this paper, we will provide a good representative in each symplectomorphism class of the surface with which
we can classify and deal with the infinite type ends more easily.
We first construct a family of such surfaces, which we call \emph{standard surfaces},
by inductively attaching a few building blocks to the previously given compact surface on its boundary, and show that
every noncompact surface can be constructed up to homeomorphism in this way.

\section{Review of homeomorphism classification of noncompact surfaces}
\label{sec:surfaces}

In this section, we briefly recall the main ideas of Richards' homeomorphism classification
of noncompact surfaces in \cite{richards} and explain how we can promote his classification result
one up to symplectomorphisms so that we can describe a generating set of
Lagrangian branes of the Fukaya category we will construct.
Leaving the full proofs of these results to Richards' paper \cite{richards},
we will explain main ideas of his proofs which will enter in our description of
the generating set.

\subsection{Homeomorphism classification and the end structure of surfaces}\label{subsec:end_structure}

We start with recalling some definitions from \cite{richards}.

We have the following standard definition of the \emph{ideal boundary}
from \cite{richards}, in which the definition is applied to a surface but can be equally
applied to general topological spaces.

\begin{defn}\label{defn:ends} Let $M$ be a noncompact topological space.
An \emph{end} of $M$ is an equivalence class of nested sequences $p=\{P_1\supset P_2 \supset \dots\}$ of
connected unbounded regions in $M$ such that
\begin{enumerate}
\item The boundary of $P_i$ in $M$ is compact for every $i$.
\item For any bounded subset $A$ of $M$, $P_i \cap A=\emptyset$ for sufficiently large $i$
\end{enumerate}
Two ends $p=\{P_1\supset P_2 \supset \dots\}$ and $q=\{Q_1\supset Q_2 \supset \dots\}$ are equivalent
if for any $n$ there is a corresponding integer $N$ such that $P_n \subset Q_N$ holds and vice versa.
We say an equivalence class
$$
[\{P_1\supset P_2 \supset \dots\}]
$$
an \emph{end} of $M$.
\end{defn}

\begin{defn}[Ideal boundary] \label{defn:ideal_bdy}
 The ideal boundary $B(M)$ of a surface $M$ is a topological space having ends as its elements and
equipped with the topology given as the following:
for any subset $U$ of $M$ whose boundary is compact in $M$, we define $U^*$ to be the set of all ends $p=\{P_1\supset P_2 \supset \dots\}$ such that $P_n\subset U$ for sufficiently large $n$.
All of such $U^*$ forms a basis of this topology.
\end{defn}

\begin{defn}
Let $p=\{P_1\supset P_2 \supset \dots\}$ be an end.
We say $p$ is planar and/or orientable if $P_n$ are planar and/or orientable for all sufficiently large $n$.
\end{defn}

\begin{defn}
 Let $p=\{P_1\supset P_2 \supset \dots\}$ be an end. We say $p$ is cylindrical
if $P_n$ are cylindrical for all sufficiently large $n$, i.e.,
$P_n \cong Q_n \times [0,\infty)$ for some compact manifold $Q_n$.
\end{defn}

\begin{rmk}
\begin{enumerate}
\item These definitions do not depend on the representative of an end.
\item A cylindrical end is planar and forms an isolated point in the ideal boundary.
If $M$ is a finite type surface, it has finitely many ends and every end is cylindrical.
\end{enumerate}
\end{rmk}

\begin{defn}[Ideal boundary triple]\label{defn:ideal-bdy-triple}
Consider the nested triples consisting of 3 sets
$$
B(M)\supset B'(M) \supset B''(M)
$$
where $B'(M)$ is the set of nonplanar ends, and $B''(M)$ is the set of nonorientable ends.
We call the triple the \emph{ideal boundary triple} of $M$.
\end{defn}

These are closed subsets of $B(M)$ by definition.

The following proposition is proved by Ahlfors and Sario in \cite{Ahlfors-Sario}.

\begin{prop} The ideal boundary $B(M)$ of a separable surface $M$ is totally disconnected, separable, and compact.
\end{prop}

Richards improved Ker\'ekj\'art\'o's earlier results and proved the following in \cite{richards}.

\begin{theorem}[{Theorem 1 of \cite{richards}}]\label{thm:Kerek}
Let $M$ and $M'$ be two separable surfaces of the same genus and orientability class.
Then $M$ and $M'$ are homeomorphic to each other if and only if
their ideal boundaries considered as triples of spaces are topologically equivalent.
\end{theorem}

Combining this with the following well-known result in general topology
\begin{prop}\label{prop:Subset_Cantor}
Any compact, separable, totally disconnected space $X$ is homeomorphic to a subset of the Cantor set.
\end{prop}
Richards obtained the following complete classification result of noncompact surfaces \emph{up to homeomorphism}.

\begin{theorem}[{Theorem 2 of \cite{richards}}] \label{thm:Richards}
Let $(X,Y,Z)$ be any triple of compact, separable, totally disconnected spaces with $X\supset Y\supset Z$.
Then there is a surface $M$ whose ideal boundary triple $(B(M),B'(M),B''(M))$ is
topologically equivalent to the triple $(X,Y,Z)$.
\end{theorem}

\begin{theorem}[{Theorem 3 of \cite{richards}}] \label{thm:Repn_surface}
Every surface is homeomorphic to a surface formed from a sphere $S^2$ by first removing a closed totally disconnected set $X$ from $S^2$, then removing the interiors of a finite or infinite sequence $D_1,D_2,\dots$ of nonoverlapping closed disks in $S^2 \setminus X$, and finally suitably identifying the boundaries of these discs in pairs. (It may be necessary to identify the boundary of one disk with itself to produce an odd ``cross-cap.") The sequence $D_1,D_2,\dots$ ``approaches $X$" in the sense that, for any open set $U$ in $S^2$ containing $X$, all but finitely number of the $D_i$ are contained in $U$.
\end{theorem}

\begin{figure}
    \centering
    \subfloat{{
    \resizebox{6cm}{4cm}{
\begin{tikzpicture}

\draw (-3,0) -- (3,0);
\draw [densely dashed](0,0) ellipse [x radius=3.8cm ,y radius=2.4cm];

\draw (-2,0) circle (40pt);
\draw (2,0) circle (40pt);

\draw (-8/3,0) circle (18pt);
\draw (-4/3,0) circle (18pt);
\draw (4/3,0) circle (18pt);
\draw (8/3,0) circle (18pt);

\draw[fill=black] (-3,0) circle (1pt) node[below right]{$0$};
\draw[fill=black] (-1,0) circle (1pt) node[below left]{$\frac{1}{3}$};
\draw[fill=black] (1,0) circle (1pt) node[below right]{$\frac{2}{3}$};
\draw[fill=black] (3,0) circle (1pt) node[below left]{$1$};

\end{tikzpicture}
}
}}
    \qquad
    \subfloat{{
    \resizebox{4cm}{4cm}{
\begin{tikzpicture}[scale=0.7]

\draw[] (-5,1) arc (270:90:1 and 0.2);
\draw[] (-5,1) arc (-90:90:1 and 0.2);

\draw[rounded corners=10pt] (-4,1.2)-- +(0,-1)--+(1,-1)-- +(1,0);

\draw[] (-2,1) arc (270:90:1 and 0.2);
\draw[] (-2,1) arc (-90:90:1 and 0.2);

\draw[rounded corners=13pt] (-6,1.2)-- ++(0,-1)--++(1.5,-1.5)--++(0,-0.8);
\draw[rounded corners=13pt] (-1,1.2)-- ++(0,-1)--++(-1.5,-1.5)--++(0,-0.8);

\draw[] (-3.5,-2) arc (270:90:1 and 0.2);
\draw[] (-3.5,-2) arc (-90:90:1 and 0.2);

\draw[] (1,1) arc (270:90:1 and 0.2);
\draw[] (1,1) arc (-90:90:1 and 0.2);

\draw[rounded corners=10pt] (2,1.2)-- +(0,-1)--+(1,-1)-- +(1,0);

\draw[] (4,1) arc (270:90:1 and 0.2);
\draw[] (4,1) arc (-90:90:1 and 0.2);

\draw[rounded corners=13pt] (0,1.2)-- ++(0,-1)--++(1.5,-1.5)--++(0,-0.8);
\draw[rounded corners=13pt] (5,1.2)-- ++(0,-1)--++(-1.5,-1.5)--++(0,-0.8);

\draw[] (2.5,-2) arc (270:90:1 and 0.2);
\draw[] (2.5,-2) arc (-90:90:1 and 0.2);

\draw[rounded corners=15pt] (-2.5,-2)-- +(0,-1)--+(2,-2)--+(4,-1)-- +(4,0);
\draw[rounded corners=13pt] (-4.5,-2)-- ++(0,-1.8)--++(3,-1.5)--++(0,-0.5);
\draw[rounded corners=13pt] (3.5,-2)-- ++(0,-1.8)--++(-3,-1.5)--++(0,-0.5);

\draw[densely dashed] (-0.5,-6.1) arc (270:90:1 and 0.2);
\draw[densely dashed] (-0.5,-6.1) arc (-90:90:1 and 0.2);
\end{tikzpicture}
}
    }}
    \caption{a surface with 4 cylindrical ends}
    \label{fig:example}
\end{figure}
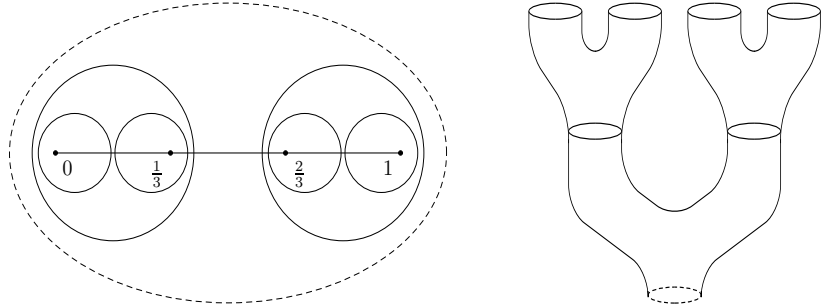

Next, leaving full proofs of these results to \cite{richards}, we will explain ideas of
Richards proof which we will suitably adapt to describe a generating set of Lagrangian branes
for our Fukaya category.

\subsection{Summary of Richards' classification theorems}

Although for our purpose of studying symplectic (and so orientable) surface cross-cap is
irrelevant, we include them in our summary of Richards' construction below for completeness' sake
and also to avoid too much deviation from the original proof of Richards' \cite{richards}.

\subsubsection{Theorem \ref{thm:Kerek}}

In the proof of Theorem \ref{thm:Kerek}, Richards decomposed $M$ and $M'$ into compact subsurfaces bordered by Jordan curves and used the fact that a connected compact bordered surface is topologically determined by
\begin{itemize}
\item
its orientability,
\item its genus, and
\item the number of its boundary curves.
\end{itemize}

More specifically,
Richards decomposed $M$ and $M'$ into $A_1\subset A_2\subset \dots$ and $A'_1\subset A'_2\subset \dots$ respectively such that
 $A_n$ and $A'_n$ are compact subsets and contained in the interior of $A_{n+1}$ and $A'_{n+1}$ respectively. After that, he used the induction to construct homeomorphism $f_n$ of $A_n$ onto $A'_n$ by extending $f_{n-1}$ from the boundary of $A_{n-1}$ beginning with $A_0=A'_0=\emptyset$.

\subsubsection{Theorem \ref{thm:Richards}}

In the proof of Theorem \ref{thm:Richards}, Richards regarded the given triple $(X,Y,Z)$
as a triple of subsets of the Cantor set using Proposition \ref{prop:Subset_Cantor} and
constructed a separable surface whose ideal boundary is identified with the given triple.
The main method of his construction is to embed $X$ into the 2-sphere $S^2$ so that
its image points have the form $(x,0) \in [0,1] \times \{0\} \subset \R^2$ with $x\in [0,1]$ by
regarding $S^2$ as the one point compactification of $\mathbb{R}^2$.
Under this construction, the triadic expansion of the real number $x$ does not involve digit 1.

\begin{defn}[$\cD'$ and $\cD$]
\begin{enumerate}[(1)]
\item Consider the collection $\cD'$ consisting of all closed disks $D$ in the $xy$ plane whose diameters are
given by the intervals contained in the $x$ axis
\be\label{eq:interval-nm}
\left[(n-\frac{1}{3})/3^m,(n+\frac{4}{3})/3^m \right], \quad \text{for }\, 0 \leq n < 3^m
\ee
where $n$ is an integer which admits a triadic expansion free from $1$'s.
\item
Let $\cD$ be the sub-collection defined by
\be\label{eq:CD}
\cD: = \{ D \in \cD' \mid D \cap X \neq \emptyset\},
\ee
i.e., those consisting of all disks in $\cD'$ containing at least one point of $X$.
\end{enumerate}
\end{defn}

Then $\cD$ determines a basis of the topology of $X$. The lattice, under the inclusion,
of sets in the collection $\cD$ has the following properties, which we shall use below:
\begin{enumerate}
\item $\cD$ is nested, i.e., any two disks in $\cD$ are either disjoint or one contains the other.
\item The intersection of the disks in any infinite linear chain of discs in $D$ consists of exactly one point of $X$.
\end{enumerate}
The latter holds because of the following reasons:
\begin{enumerate}[{(a)}]
\item Any infinite set of nested discs containing a point of $X$ contains a strictly monotone sequence $D_\ell$
under the inclusion order.
\item The diameters of the intersections $D_\ell \cap X$ are intervals of the form \eqref{eq:interval-nm} for each $\ell$.
We also observe that the diameter of $D_\ell$ converges to 0 as $\ell \to \infty$ for any infinite sequence $D_\ell$.
\item $X$ is compact.
\end{enumerate}
Combining the above, we derive that $D_\ell \cap X \neq \emptyset$ and
$D_\ell \cap X \supset D_{\ell +1} \cap X$ and $\text{\rm diam}\, D_\ell \to 0$ as $\ell \to \infty$.
This implies that
$$
\bigcap_{\ell = 1}^\infty D_\ell \cap X
$$
is a nonempty subset of $X$ whose diameter is zero. This concludes Statement (2) above.

Let $\H^+$ and $\H^-$ be the half planes $y>0$ and $y<0$ respectively.

\begin{defn}[$D'$ and $D''$]\label{defn:D'D''}
For each disc $D$ in $\cD$,
we define $D'$ and $D''$ to be the two disjoint largest discs in $\cD'$ properly contained in $D$.
Note that at least one of $D'$ or $D''$ is in $\cD$.
\end{defn}

For every disc $D$ in $\cD$, we choose two circles $C^+(D)$ and $C^-(D)$, each contained in the interior of $D$, such that:
\begin{itemize}
\item $C^+(D)\subset \H^+$ and $C^-(D)\subset \H^-$.
\item $C^+(D)$ and $C^-(D)$ intersect neither $D'$ nor $D''$.
\item $C^+(D)$ and $C^-(D)$ are symmetric with respect to the $x$ axis.
\end{itemize}

Then no two distinct circles $C^\pm(D)$ intersect.

We now construct $M$ as the ``double" of a compact surface with boundary, which is
$S^2$ with the points in $X$ and the interiors of some of the circles $C^\pm (D)$ removed.
First we fill in the circles $C^\pm (D)$
for all $D\in \cD$ satisfying $D\cap Y= \emptyset$, and remove the interiors of $C^\pm (D)$ for all $D\in \cD$ for which $D \cap Y \neq \emptyset$. If $D \cap Y \neq \emptyset$ but $D\cap Z = \emptyset$, then we identify the boundaries of $C^+(D)$ and $C^-(D)$ by reflecting $C^+(D)$ in the $x$ axis(preserving orientation in $M$). If $D\cap Z \neq \emptyset$, then we fill in $C^-(D)$ and glue boundary of $C^+(D)$ with itself and make a cross-cap.

We shall show that the ideal boundary of $M$ is equivalent to the triple $(X,Y,Z)$. It follows from Properties (1) and (2) above
that any point in $X$ can be uniquely represented as the intersection of the sets in a maximal ordered chain in the lattice of
sets in $\cD$. From the definition of an end, this gives us an end of $M$. This defines a mapping from $X$ into
the ideal boundary $B(M)$ (Definition \ref{defn:ideal_bdy}).

Since $X \setminus Y$ and $X \setminus Z$ are open subsets of $X$, every point $p$ in $X \setminus Y$ is
contained in some disk $D\in \cD$ such that $D \cap Y =\emptyset$, and similarly for every $p\in X \setminus Z$.
Hence the subsets $Y$ and $Z$ of $X$ correspond precisely to the maximal chains which represent
nonplanar and nonorientable boundary components of $M$.

After that, we check that this map is a surjective homeomorphism from $X$ into $B(M)$ and this concludes Theorem \ref{thm:Richards}.

\subsubsection{Theorem \ref{thm:Repn_surface}}

Now, given a surface $M$, we can construct a surface $M'$ whose ideal boundary triple is $(B(M),B'(M),B''(M))$
by taking a ``double" of a modified sphere. According to Theorem \ref{thm:Kerek}, it is sufficient to
consider possible variations in the genus and orientability class of $M$ and $M'$.
If either of these surfaces has infinite genus or infinite degree of nonorientability, then the ideal boundaries
contain nonplanar or nonorientable end and the assumption that these invariants be the same is redundant.
Since it is possible to vary the genus or degree of nonorientability in the finite case
by adding or subtracting a finite number of ``handles" or ``cross caps", we have Theorem \ref{thm:Repn_surface}.
When adding such ``handles" or ``cross caps" to $M'$, choose disjoint disks located out of the largest disk,
say $D_0$, in $\cD'$ whose diameter is $[-\frac{1}{3},\frac{4}{3}]$ and identify them as needed.
This gives rise to Theorem \ref{thm:Repn_surface}.

\begin{defn}Let $M'$ be a surface constructed as above.
$M' \setminus \text{\rm Int}(D_0)$ in $M'$ is compact which we call a \emph{compact part} of $M'$.
\end{defn}

\begin{rmk} Note that a cylindrical end corresponds to an isolated point in the ideal boundary.
Therefore if a Riemann surface has an ideal boundary which is a perfect set,
it cannot have any cylindrical end.
Furthermore every surface can have at most countably many cylindrical ends.
\end{rmk}

\subsection{Equivalence classes of end structures and its building
blocks}\label{subsec:eq_class_end_structure}

In the last subsection, we introduced an ideal boundary of a separable surface.
By Theorem \ref{thm:Kerek}, two separable surfaces of the same genus and
orientability class are homeomorphic to each other if and only if
their ideal boundaries considered as triples of spaces are topologically equivalent.
Therefore, understanding topological structure of the ideal boundary will help us
to deal with the end structures and the homeomorphism classes of separable surfaces.
In this subsection, we will provide the description of  a basis of the topology of
the ideal boundary whose description is now in order.

Let $\mathcal{C}$ be the Cantor set equipped with a subspace topology of $\mathbb{R}$.
Recall that the Cantor set is a compact totally disconnected Hausdorff space.
Its topology can be described by the following basis.
\begin{defn}[The standard basis of the Cantor set] We equip $\mathcal{C}$ with
a basis of a topology given by the set of intervals
$$
\mathcal{I}= \{[0,1],[0,1/3],[2/3,1],\dots \}.
$$
\end{defn}
The $\mathcal{I}$ corresponds to $\cD'$ in the last subsection such that a disk $D$ whose diameter is $[(n-\frac{1}{3})/3^m,(n+\frac{4}{3})/3^m]$ corresponds to $[n/3^m,(n+1)/3^m]$.

The following is easy to verify whose proof is omitted.

\begin{lemma} Then $\mathcal I$ has the following properties:
\begin{enumerate}
\item A pair of elements of $\CI$  is either disjoint or nested.
\item For every point $p\in \mathcal{C}$, there exists
a decreasing sequence $I_1,\,I_2,\,\dots$ of elements of $\mathcal{I}$ containing $p$ such that
$I_i\supset I_{i+1}$ for all $i$ and$\{p\}=\bigcap_{i=1}^{\infty} I_i$.
\item $\mathcal{I}$ is a POSET under the inclusion relation by setting
$I < I'$ to be $I \subset I'$.
\item
 For each $m = 0,\,1,\, \ldots $, the sub-collection $\mathcal{I}^{(m)} \subset \mathcal{I}$ defined by
\be\label{eq:collection-k}
\left\{\left[\frac{n}{3^m}, \frac{(n+1)}{3^m}\right] \, \Big| \, 0 \leq n < 3^m, \, \left\lfloor \frac{n}{3^k} \right\rfloor
 \not\equiv 1\mod 3\, \forall 0\leq k<m  \right\}
\ee
is also an open cover of $\mathcal{C}$ whose elements are disjoint from one another.
This sub-collection will be used to construct a rooted binary graph in Section \ref{sec:symplecto-classification},
for which the index $m$ will be the distance from the root.
\end{enumerate}
\end{lemma}

Given a pair of intervals $[a,b], \, [c,d]\in  \mathcal I$, their relationship is one of the following:
\begin{itemize}
\item They are disjoint.
\item They are the same.
\item One is a proper subset of the other.
\end{itemize}

Moreover, in the last case, we have additional information.

\begin{defn}
Let $i_\mathcal{C}:\{([a,b],[c,d])\in \mathcal{I}\times \mathcal{I}|[a,b]\subsetneq [c,d]\}\rightarrow \{0,2\}$ be defined as
\begin{equation}
i_\mathcal{C}([a,b],[c,d])=
\begin{cases}
0 & $if $c\leq a,b \leq c+\frac{d-c}{3}\\
2 & $if $c+\frac{2(d-c)}{3}\leq a, b\leq d\\
\end{cases}
\end{equation}
\end{defn}

\begin{rmk}
$i_\mathcal{C}([a,b],[c,d])$ is related to ternary representation.
Pick $p\in \CC$.
Then we have a maximal chain $\{I_i\}_{i\in \ZZ_{>0}}\subset \mathcal I$ such that
$I_1 \supsetneq I_2 \cdots$ and $\bigcap_{i=1}^{\infty}I_i=\{p\}$.
Then ternary representation of $p$ is $0.k_1k_2\ldots_{(3)}$ where
$i_\mathcal{C}(I_i,I_{i+1})=k_i$
\end{rmk}

Let $M$ be a separable surface and let $(B(M),B'(M),B''(M))$ be its ideal boundary triple as defined in
Definition \ref{defn:ideal-bdy-triple}.
\begin{defn}[Basis $\CS$ of topology of $B(M)$]\label{defn:basis}
We consider the following sub-collection of $\cI$
\be\label{eq:cS}
\cS:=\{I\in\mathcal{I} \mid I\cap B(M) \neq \emptyset\}.
\ee
Then the collection of open subsets of $B(M)$, still denoted by $\cS$,
\be
\cS = \{I\cap B(M) \mid I \in \cS\}
\ee
forms a basis of the subspace topology of $B(M) \subset [0,1]$.
\end{defn}

We now provide an explicit construction of the surface $M$ with given triple
$(B(M),B'(M),B''(M))$ by utilizing the proof of Theorem \ref{thm:Richards}:
\begin{enumerate}[(a)]
\item We pick $D_0$ in $\cD$ during the proof of Theorem \ref{thm:Repn_surface}.
\item If $D_0\cap B'(M)=\emptyset$, we fill in the circles $C^\pm (D_0)$.
\item If $D \cap Y \neq \emptyset$ but $D\cap Z = \emptyset$,
 we identify the boundaries of $C^+(D)$ and $C^-(D)$
by reflecting $C^+(D)$ in the $x$ axis(preserving orientation in $M$).
This results in a genus contained in the complement of $\text{\rm Int}(D_0')\cup \text{\rm Int}(D_0'')$ in $D_0$.
\item If $D\cap Z \neq \emptyset$, we fill in $C^-(D)$ and glue the boundary of $C^+(D)$
with itself and make a cross-cap.
\end{enumerate}
Moreover, $D_0'$ and $D_0''$ are disjoint disks contained in $D_0$ and at least one of those two is contained in $\cD$. (See Definition \ref{defn:D'D''}.)
Therefore, we have derived that
$$
D_0 \setminus \left(\text{\rm Int}(D_0')\cup \text{\rm Int}(D_0'')\right)
$$
is homeomorphic to one of the following domains:

\begin{defn}[Building blocks]\label{defn:building-blocks}
Considers the following domains of
\begin{enumerate}
\item a pair of pants,
$$
D_0\cap B'(M)=\emptyset, \quad D_0',D_0''\in \cD,
$$
\item a pair of pants with a cross cap,
$$
D_0\cap B''(M)\neq \emptyset, \quad D_0',D_0''\in \cD,
$$
\item a pair of pants with a genus,
$$
D_0\cap B''(M)=\emptyset, \quad D_0\cap B'(M)\neq \emptyset, \quad D_0',D_0''\in \cD,
$$
\item a cylinder,
$$
D_0\cap B'(M)=\emptyset, \quad D_0'\not\in \cD \, \text{\rm or },\,   D_0''\not\in \cD,
$$
\item a cylinder with a cross cap,
$$
D_0\cap B''(M)\neq\emptyset, \quad D_0'\not\in \cD, \text{\rm or }\, D_0''\not\in \cD.
$$
\item a cylinder with a genus,
$$
D_0\cap B''(M)=\emptyset, \quad D_0\cap B'(M)\neq \emptyset, \quad D_0'\not\in \cD,
\text{\rm or }, \,  D_0''\not\in \cD.
$$
\end{enumerate}
\end{defn}
We can construct our intermediate compact surface with boundary
by taking  the union of the given base compact domain and
a collection of the above building blocks attached thereto.

\begin{defn}[Building block $B_I$]\label{defn:building-block}
Let $I \in \cS$ be given.
\begin{enumerate}
\item  we denote the above constructed building block by $B_I$ and
ay that \emph{a building block is attached to $I$.}
\item For given $J\subset I$, we write
$$
I_*: =\bigcup_{J\in \cS,J\subset I} B_I.
$$
Note that ${{I}_*}^*=I\cap B(M)$.
\end{enumerate}
\end{defn}

In other words, we construct the desired surface
by repeatedly attaching building blocks to a compact domain along boundary
at each step of the construction.

To deal with nonplanar or nonorientable ends, we need additional data on $\cS$, which will be encoded by
the following map $\chi_\cS : \cS \rightarrow \{0,1,2\}$
which encodes some structure of the ideal boundary of $M$.

\begin{defn}[Counting genus and cross-cap]
\label{defn:blue-print}
Let $\chi_\cS : \cS \rightarrow \{0,1,2\}$ be defined as
\begin{equation}
\chi_\cS(I)=
\begin{cases}
1 & $if $I\cap B''(M)\neq \emptyset\\
2 & $if $I\cap B''(M)= \emptyset$ and $I\cap B'(M)\neq \emptyset\\
0 & $otherwise$\\
\end{cases}
\end{equation}
$\chi_\cS$ counts the contribution of genus and cross-cap contained in the building block
attached to $I$ to the Euler characteristic of whole surface $M$.
Note that if there is at least one interval $I \in \cS$ such that $\chi_\cS(I)=1$,
then $M$ is nonorientable. We call the pair $(\cS,\chi_\cS)$ the \emph{blueprint pair}
of the surface $M$.
\end{defn}
So far we have constructed a blueprint pair $(\cS,\chi_\cS)$ from
the ideal boundary triple $(B(M),B'(M),B''(M))$.
Conversely, suppose we are given a blueprint pair $(\cS,\chi_\cS)$.
For any maximal chain $\{I_i\}_{i\in \ZZ_{>0}}\subset \mathcal I$ in $\cS$,
we have a point $p\in \cC$ with
$$
\bigcap_{i=1}^{\infty}I_i=\{p\}.
$$
Let $B(M)$ be the set of such points.
\begin{itemize}
\item
If a maximal chain $\{I_i\}_{i\in \ZZ_{>0}}\subset \mathcal I$, which corresponds to a point $p\in B(M)$,
contains infinitely many intervals $I_i$ such that $\chi_\cS(I_i)\neq 0$, we say $p\in B'(M)$.
\item If a maximal chain $\{I_i\}_{i\in \ZZ_{>0}}\subset \mathcal I$, which corresponds to a point
$p\in B(M)$, contains infinitely many $I_i$ such that $\chi_\cS(I_i)=1$, we say $p\in B''(M)$.
\end{itemize}
Therefore there is a correspondence between the pair $(\cS,\chi_\cS)$ and
 the ideal boundary triple.  This completes the description of $M$ as the surface
 obtained by an iterative gluing of building blocks.

 \begin{notation}[$M_\cS$] For given $\cS$,  we denote by $M_\cS$ the above constructed
 surface.
 \end{notation}
 We remark that the topology of the surface $M_\cS$ depends only on $\cS$.
For the simplicity of notation, we also write
$$
B(\cS): = B(M_\cS).
$$

\subsection{Pair-of-pants decomposition of surface}

Additionally, we use the following pair-of-pants decomposition of
$M$ which we encode by the following convenient function $\xi_\cS : \cS \rightarrow \{0,1\}$.

\begin{defn}[Counting pairs of pants]
Define a function $\xi_\cS : \cS \rightarrow \{0,1\}$ by
\begin{equation}
\xi_\cS(I)=
\begin{cases}
1 & $if $i_\mathcal{C}(I_1,I)=0$ and $i_\mathcal{C}(I_2,I)=2$ for some $I_1,I_2\in \cS\\
0 & $otherwise$\\
\end{cases}
\end{equation}
This counts the number of pairs of pants attached to $I$.
If a cylinder, a cylinder with a cross cap, or a cylinder with a genus is attached to $I$,
then $\xi_\cS(I)=0$.
\end{defn}

Using the value of $\chi_\cS(I)$ and $\xi_\cS(I)$, we can encode the
6 building blocks attached to $I$  in terms thereof as follows:
\begin{enumerate}
\item a pair of pants ($\xi_\cS(I)=1$, $\chi_\cS(I)=0$)
\item a pair of pants with a cross cap ($\xi_\cS(I)=1$, $\chi_\cS(I)=1$)
\item a pair of pants with a genus ($\xi_\cS(I)=1$, $\chi_\cS(I)=2$)
\item a cylinder ($\xi_\cS(I)=0$, $\chi_\cS(I)=0$)
\item a cylinder with a cross cap ($\xi_\cS(I)=0$, $\chi_\cS(I)=1$)
\item a cylinder with a genus ($\xi_\cS(I)=0$, $\chi_\cS(I)=2$)
\end{enumerate}

Using the homeomorphism classes of building blocks, we can now define
a homeomorphism between two blueprint pairs.
\begin{defn}
Let $(\cS,\chi_\cS)$ and $(\cS',\chi_{\cS'})$ be blueprint pairs.
We say two blueprint pairs are homeomorphic to each other if there are chains of subsets
$\cS_1\subset \cS_2 \subset \cdots$ of $\cS$, and $\cS'_1\subset \cS'_2 \subset \cdots$ of $\cS'$ respectively such that the following hold:
\begin{itemize}
\item $[0,1]\in \cS_1$, $[0,1]\in \cS'_1$.
\item Each $\cS_i$[resp.$\cS'_i$] corresponds to a connected subgraph of
the binary graph which corresponds to $(\cS,\chi_\cS)$[resp.$(\cS',\chi_\cS')$].
\item $\bigcup_i^\infty \cS_i=\cS$, $\bigcup_i^\infty \cS'_i=\cS'$.
\item Let $M_{\cS_i}$ be a surface obtained from attaching $B_I$ for $I\in \cS_i$ to the compact part of $M_\cS$.
Then $M_{\cS_i}$ is connected and compact.
Define $M_{\cS'_i}$ in the same way.
Then $M_{\cS_i}$ is homeomorphic to $M_{\cS'_i}$ for all $i$, which is equivalent to:

\begin{itemize}
\item Orientability of $M_{\cS_i}$ and $M_{\cS'_i}$ are same for all i.
\item There are $k,k'\in \ZZ_{\geq 0}$ such that for all $i$,
\begin{equation}
k+\sum_{I\in \cS_i} \chi_\cS(I)=k'+\sum_{I\in \cS'_i} \chi_\cS'(I)
\end{equation}
where  [resp.$k'$] is the Euler characteristic contribution of genus and cross-caps
contained in the compact part of $M_\cS$[resp.$M_\cS'$].
\item For all $i$,
\begin{equation}
\sum_{I\in \cS_i} \xi_\cS(I)=\sum_{I\in \cS'_i} \xi_\cS'(I)
\end{equation}
\end{itemize}

\end{itemize}

\end{defn}

This definition follows directly by computing and comparing the numbers of boundary components,
Euler characteristics, and orientability in the proof of Theorem \ref{thm:Kerek}.

\begin{rmk}
Some homeomorphism on a surface may move its ideal boundary but
it cannot change the homeomorphism class of its ideal boundary.
Some examples of such a homeomorphism are swapping of two boundary components of a pair of pants, shifting genus, and others
we refer readers to \cite{Aramayona-Vlamis} for more on the structure of
the mapping class group of  infinite type surfaces.
\end{rmk}

\section{Symplectomorphism classification and standard surfaces}
\label{sec:symplecto-classification}

In this section, we promote Richards' homeomorphism classification result to
that of symplectomorphisms and introduce the class of \emph{standard surfaces}
each of which provide a good geometric model for the Floer theory. This model
should be the replacement of the requirement of cylindrical ends for the
Liouville manifolds \emph{of finite type}.

\subsection{Diffeomorphism classification}\label{subsec:diffeom_classification}
From now on, our surfaces of interest will be a noncompact symplectic surface without boundary.
Therefore, it is orientable and does not contain any cross-cap.

\begin{lemma}\label{lem:diffeo}
Let $M$ be a noncompact symplectic surface without boundary.
Then we can construct another surface $M'$ diffeomorphic to $M$ by repeatedly attaching
one of the following types of surfaces to a disk:
\begin{enumerate}
\item a pair of pants
\item a cylinder with a genus
\item an infinite cylinder
\end{enumerate}
\end{lemma}
\begin{proof}
In Section \ref{subsec:end_structure}, we constructed a surface by taking a ``double" of a modified sphere
and adding finitely many handles or cross-caps.
Therefore, up to homeomorphism, we assume $M$ is such a surface.
Since our surface is orientable, $B''(M)=\emptyset$ and we do not need to consider cross caps.
By definition, the compact part of $M$ is constructed by identifying disks contained in the complement of
$D_0 \in \cD$ in the sphere in pair and making genus.
There are finitely many genus in the \emph{compact part} which is homeomorphic to a surface (with boundary)
constructed by attaching cylinders with genus along the boundary of a disk.
Denote this surface by $M'_0$.
We now have a homeomorphism from $M'_0$ onto the compact part of $M$ and need to extend it into
the interior of elements of $\cD$.

In the proof of Theorem \ref{thm:Richards}, if $D_0\cap B'(M)=\emptyset$, we fill in the circles
$C^\pm (D_0)$. Otherwise, we remove the interiors of $C^\pm (D_0)$ and identify the boundaries
of $C^+(D_0)$ and $C^-(D_0)$ by reflecting $C^+(D)$ in the $x$ axis
(preserving the orientation in $M$). This results in a genus contained in the complement of $\text{\rm Int}(D_0')\cup \text{\rm Int}(D_0'')$ in $D_0$. Also, $D_0'$ and $D_0''$ are disjoint disks contained in $D_0$ and at least one of those two is contained in $\cD$. Therefore, the
complement
$$
D_0 \setminus \left(\text{\rm Int}(D_0')\cup \text{\rm Int}(D_0'')\right)
$$
is homeomorphic to one of the following building blocks:
\begin{enumerate}
\item a pair of pants ($D_0\cap B'(M)=\emptyset$, $D_0',D_0''\in \cD$)
\item a pair of pants with a genus ($D_0\cap B'(M)\neq \emptyset$, $D_0',D_0''\in \cD$)
\item a cylinder ($D_0\cap B'(M)=\emptyset$, $D_0'\not\in \cD$ or $D_0''\not\in \cD$)
\item a cylinder with a genus ($D_0\cap B'(M)\neq\emptyset$, $D_0'\not\in \cD$ or $D_0''\not\in \cD$)
\end{enumerate}

For the cases of (1), (3) and (4), we attach a pair of pants, a cylinder and a cylinder with a genus
respectively  to a boundary of $M'_0$ and denote the resulting surface by $M'_1$.

For the case of (2), we attach a cylinder with a genus to a boundary of $M'_0$ and attach a pair of pants to the resulting surface and call the resulting surface $M'_1$.
After that, extend the homeomorphism between $M'_0$ and the compact part of $M$ to the homeomorphism between $M'_1$ and the complement of $\text{\rm Int}(D_0')\cup \text{\rm Int}(D_0'')$.
As we did in the proof of Theorem \ref{thm:Kerek}, we do the same for other disks in $\cD$
and can get a homeomorphism between $M$ and $\cup_{i=1}^{\infty} M'_i$.

For the case of (3), if we attach finitely many cylinders consequently,
 it does not affect the homeomorphism-type and we can skip it.
 Otherwise, the attaching operation  is the same as attaching an infinite cylinder.
 Therefore, we can construct a surface which is homeomorphic to $M$
 by repeatedly attaching one of the following types of surfaces to a disk:
\begin{enumerate}
\item a pair of pants
\item a cylinder with a genus
\item an infinite cylinder
\end{enumerate}

By the construction of each homeomorphism, the attaching boundary of every step is contained in the interior of the resulting compact bordered surface. So we may assume the attaching boundary is smooth
and so take a collar neighborhood of each boundary component in every attaching step. Furthermore,
we can replace the operation of attaching boundaries by that of gluing along collar neighborhoods.
Therefore, we have improved the homeomorphism classification by a diffeomorphism classification.
\end{proof}

In the Section \ref{subsec:standard-surface} for the symplectomorphism classification,
we will equip each building block with an area form(equivalently a symplectic form).

\begin{defn}[Pointed surface] Let $p_0$ be a given point of $M$.
We call such a pair $(M,\{p_0\})$ a \emph{pointed surface} and $p_0$
the \emph{base point}.
\end{defn}

If $M$ is a surface that is constructed by an iterative attaching of building blocks to a disk,
we take the center of the disk as the \emph{base point} of the surface $M$.
We use this base point to define a plurisubharmonic function in Section \ref{subsec:standard-surface}.
The point will be used as the ``starting point" of our homeomorphism.

Now we can associate a rooted binary graph to each homeomorphism type of $M$
once we pick a base point and get a basis $\cS$ of the topology of $M$. (See Definition \ref{defn:basis}
for the definition of $\cS$.) However, the associated graph will not be cycle-free
and so not suitable to apply a transitive reduction.

 \footnote{In graph theory, a transitive reduction is a
reduction process of a graph by removing some edges from the set of edges between two given vertices
with keeping them still incident to each other.}
Given a subset $\cS \subset \cJ$ which provides a basis of the ideal boundary $B(M)$ of
the given surface $M$, we consider a graph the set of whose vertices is described as above.

But we add one more requirement for each edge between two vertices correspond to $I,\,J$; we require
$I\in \mathcal{I}^{(m+1)}$ and $J\in\mathcal{I}^{(m)}$ for some $m\in \ZZ_{>0}$.
To show that this graph is a binary graph, we recall several properties of the basis $\cS$:

\begin{defn}[Binary graph of $\CS$]\label{defn:binary-tree}
\begin{enumerate}
\item $[0,1]\in \cS$ for every $\cS$ and $I\subset [0,1]$ for all.
\item If $I\subsetneq J$, $J\in \mathcal{I}^{(m)}$ and $I\in \mathcal{I}^{(m+k)}$ for some $k\in \ZZ_{>0}$, there exists a unique maximal chain $I_0=I\subsetneq I_1\subsetneq \cdots \subsetneq I_k=J$ so that $I_i\in \mathcal{I}^{(m+k-i)}$ for all $i$.
\item For every $J\in \cS$, there are at most two $I_1,\,I_2\in \cS$ such that
$I_1,\,I_2$ are maximal proper subsets of $J$.
Moreover, $\{i_\mathcal{C}(I_1,J),\,i_\mathcal{C}(I_2,J)\}=\{0,\,2\}$.
\end{enumerate}
\end{defn}

By Property (1), every graph has a vertex which corresponds to $[0,1]$ and it is the root of the graph.
By Property (2), this graph is constructed by a transitive reduction starting from the root vertex:
If a vertex in the graph is given, there is a unique path from the given vertex to the root and
this graph is a rooted tree.
By Property (3), the root vertex has valence 1 or 2 and every other vertex has valence either 2 or 3, so this graph is a rooted binary graph.

Moreover, we may label children of each vertex as follows.
If a vertex corresponding to $I$ is a child of a vertex corresponding to $J$ and $i_\mathcal{C}(I,J)=0$,
we call it a \emph{left child} of $J$.
Otherwise, we call it a \emph{right child}.
Figures in this paper will  be drawn according to this naming of children.

This binary graph corresponding to $\cS$ visualizes the topological structure of $B(M)$, but it lacks
the data of $B'(M)$ since it does not distinguish a cylinder with a genus from a cylinder without a genus.
To encode the data thereof for every $I\in \cS$ with $\xi_\cS(I)=0$ and $\chi_\cS(I)=2$,
we replace a vertex  corresponding to $I$ to a cycle with length two (See Figure \ref{fig:surface_and_corresp_tree}).
By Lemma \ref{lem:diffeo}, we may assume that there is no $I\in \cS$ with $\xi_\cS(I)=1$ and $\chi_\cS(I)=2$ up to diffeomorphism.
The resulting graph may not be cycle-free, but it encodes the datum of a pair of pants decomposition of $M$ and so the homeomorphism type of $M$ as follows:
A vertex with valence 2 corresponds to a cylinder and a vertex with valence 3 corresponds to a pair of pants.
Since a cylinder with a genus is generated by attaching two pairs of pants along boundaries of leg openings, we may also name each cylinder around a genus as the left or the right.

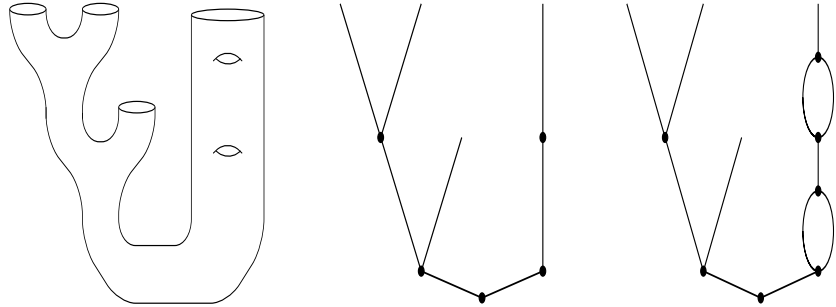
\begin{figure}\center

    \subfloat{
    \resizebox{3.5cm}{4cm}{
\begin{tikzpicture}

\draw[] (2,2.5) arc (0:180:0.5 and 0.1);
\draw[] (2,2.5) arc (0:-180:0.5 and 0.1);
\draw[] (0,2.5) arc (0:180:0.5 and 0.1);
\draw[] (0,2.5) arc (0:-180:0.5 and 0.1);

\draw[] (6,2.4) arc (0:180:1 and 0.1);
\draw[] (6,2.4) arc (0:-180:1 and 0.1);

\draw[rounded corners=13pt] (1,2.5)-- ++(0,-0.5)--++(-1,0)--++(0,0.5);

\draw[rounded corners=13pt] (1,0.6)-- ++(0,0.6)--++(1,0.8)--++(0,0.5);
\draw[rounded corners=13pt] (0,0.6)-- ++(0,0.6)--++(-1,0.8)--++(0,0.5);

\draw[rounded corners=13pt] (2,0.8)-- ++(0,-0.6)--++(-1,0)--++(0,0.5);

\draw[] (3,0.8) arc (0:180:0.5 and 0.1);
\draw[] (3,0.8) arc (0:-180:0.5 and 0.1);

\draw[rounded corners=13pt] (3,0.8)-- ++(0,-0.6)--++(-1,-0.8)--++(0,-1)--++(2,0)--++(0,4);

\draw[rounded corners=13pt] (0,0.8)-- ++(0,-0.6)--++(1,-0.8)--++(0,-1)--++(1,-1)--++(3,0)--++(1,1)--++(0,4);

\begin{scope}[scale=0.8]
\draw[rounded corners=10pt] (5.75,0)-- +(0.5,0.3)-- +(1,0);
\draw[rounded corners=8pt] (5.85,0.05)-- +(0.4,-.2)-- +(0.85,0);
\end{scope}

\begin{scope}[scale=0.8]
\draw[rounded corners=10pt] (5.75,2)-- +(0.5,0.3)-- +(1,0);
\draw[rounded corners=8pt] (5.85,2.05)-- +(0.4,-.2)-- +(0.85,0);
\end{scope}

\end{tikzpicture}
}}
    \qquad
    \subfloat{
    \resizebox{3cm}{4cm}{

\begin{tikzpicture}

\draw[fill=black] (0,0) circle (1pt);

\draw (0,0) -- +(0.75,-0.2) -- +(1.5,0);

\draw[fill=black] (1.5,0) circle (1pt);

\draw (0,0) -- +(-0.5,1);

\draw (0,0) -- +(0.5,1);

\draw[fill=black] (-0.5,1) circle (1pt);

\draw (-0.5,1) -- +(-0.5,1);
\draw (-0.5,1) -- +(0.5,1);

\draw[] (1.5,0)--+(0,1);

\draw[fill=black] (1.5,1) circle (1pt);

\draw[] (1.5,1)--+(0,1);

\draw[fill=black] (0.75,-0.2) circle (1pt);

\end{tikzpicture}

}}
    \qquad
    \subfloat{
    \resizebox{3cm}{4cm}{

\begin{tikzpicture}

\draw[fill=black] (0,0) circle (1pt);

\draw (0,0) -- +(0.75,-0.2) -- +(1.5,0);

\draw[fill=black] (1.5,0) circle (1pt);

\draw (0,0) -- +(-0.5,1);
\draw (0,0) -- +(0.5,1);

\draw[fill=black] (-0.5,1) circle (1pt);

\draw (-0.5,1) -- +(-0.5,1);
\draw (-0.5,1) -- +(0.5,1);

\draw[fill=black] (1.5,0.6) circle (1pt);

\draw[] (1.5,0) arc (-90:-540:0.2 and 0.3);

\draw (1.5,0.6) -- +(0,0.4);

\draw[fill=black] (1.5,1) circle (1pt);

\draw[] (1.5,1) arc (-90:-540:0.2 and 0.3);

\draw (1.5,1.6) -- +(0,0.4);

\draw[fill=black] (1.5,1.6) circle (1pt);
\draw[fill=black] (0.75,-0.2) circle (1pt);

\end{tikzpicture}

}}

\caption {A surface, corresponding rooted binary graph and a pair of pants decomposition graph}
\label{fig:surface_and_corresp_tree}

\end{figure}

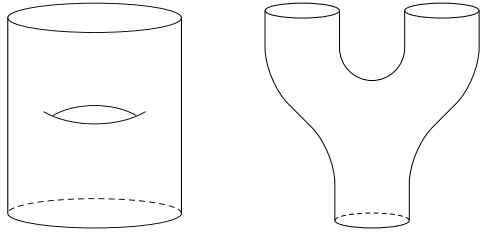
\begin{figure}

    \centering
    \subfloat{{
    \resizebox{2.5cm}{3cm}{
\begin{tikzpicture}

\draw (-1.5,-2) -- (-1.5,2);

\draw (1.5,-2) -- (1.5,2);

\begin{scope}[scale=0.8]
\path[rounded corners=24pt] (-.9,0)-- +(0.9,.6)-- +(1.8,0) +(0,0)-- +(0,0.34)--    +(1.8,0);
\draw[rounded corners=28pt] (-1.1,.1)-- +(1.1,-0.7)-- +(2.2,0);
\draw[rounded corners=24pt] (-.9,0)-- +(0.9,.6)-- +(1.8,0);
\end{scope}

\draw[] (0,1.7) arc (270:90:1.5 and 0.3);
\draw[] (0,1.7) arc (-90:90:1.5 and 0.3);

\draw[densely dashed] (1.5,-2) arc (0:180:1.5 and 0.3);
\draw[] (1.5,-2) arc (0:-180:1.5 and 0.3);
\end{tikzpicture}
    }
    }}
    \qquad
    \subfloat{{

    \resizebox{3cm}{3cm}{
\begin{tikzpicture}

\draw[] (2.5,2.5) arc (0:180:0.8 and 0.16);
\draw[] (2.5,2.5) arc (0:-180:0.8 and 0.16);
\draw[] (-0.5,2.5) arc (0:180:0.8 and 0.16);
\draw[] (-0.5,2.5) arc (0:-180:0.8 and 0.16);

\draw[rounded corners=20pt] (0.9,2.5)-- ++(0,-1.5)--++(-1.4,0)--++(0,1.5);
\draw[rounded corners=20pt] (1,-2)-- ++(0,1.5)--++(1.5,1.5)--++(0,1.5);
\draw[rounded corners=20pt] (-0.6,-2)-- ++(0,1.5)--++(-1.5,1.5)--++(0,1.5);

\draw[densely dashed] (1,-2) arc (0:180:0.8 and 0.16);
\draw[] (1,-2) arc (0:-180:0.8 and 0.16);
\end{tikzpicture}
    }
    }}
    \caption{Building blocks}
\end{figure}

\subsection{Standard surface and its building blocks}
\label{subsec:standard-surface}

By the classification results from the previous section, especially for the orientable surfaces, there
is one-to-one correspondence between the topological types of orientable surfaces and the set of \emph{trivalent planar graphs}: Each orientable surface is obtained by replacing each edge of the graph by a
cylinder $C$ and replacing a neighborhood of each vertex by a pants constructed as above.

In this subsection, we represent the graph constructed above in the first construction
as the \emph{Reeb graph} of a proper Morse function. We first recall the definition of
Reeb graph \cite{reeb:graph}.

\begin{defn}[Reeb graph]\label{defn:reeb-graph} Given a topological space $X$ and a continuous function
$f: X \to \R$, consider the equivalence relation on $X$ where $p \sim q$ if $p$ and $q$
belong to the same connected component of a single level set $f^{-1}(c)$ for some real
number $c$. The \emph{Reeb graph} of $(X,f)$ is the quotient space $X/\sim$
endowed with the quotient topology. We denote  by $G_{(X,f)}$
the Reeb graph associated to $(X,f)$.
\end{defn}

We now consider a proper Morse function $f=\psi: M \to \R$ with $\psi \geq 0$
satisfying
$$
\psi^{-1}(0) = \{p_{\text{\rm rt}}\}.
$$
We have a natural fibration
$$
\pi_{(M,\psi)}: M \to G_{(M,\psi)}
$$
whose preimages are either a union of finite number of circles which
corresponds to a regular value of $f$ or a figure 8 which corresponds to
a critical level other than the $\psi = 0$ which corresponds to the base point.
We can adjust the given good Morse function so that its critical levels
lie at some integer $2k \geq 0$.

We then parameterize each connected component of $\psi^{-1}([2k-1/2, 2k + 1/2])$
by a finite number of unit cylinders $ [0,1] \times S^1$ equipped with
the cylindrical coordinates $(s,t)$ with $0\leq s\leq 1$, $0\leq t \leq 1 (\mod 1)$.
We parameterize the semi-infinite cylinder by $(s,t)$ with $s \geq 0$ and $t \in S^1$.

Equip these cylinders with the standard flat metric  $g = ds^2 + dt^2$
and with the standard (almost) complex structure $J$ satisfying
$$
J\left({\frac{\partial}{\partial s}}\right)={\frac{\partial}{\partial t}},
J\left({\frac{\partial}{\partial t}}\right)=-{\frac{\partial }{\partial s}}
$$
Then the standard symplectic form $ds\wedge dt$ is $J$-compatible.

On the other hand, the preimage $\psi^{-1}([2k-1/2, 2k+1/2])$
is a union of a finite number of pants each of which contains a unique
critical point at the level $2k$. We conformally identify each such component with
the standard pair of pants consisting of the union of three
cylinders $C_i$ of height 1 constructed in the following way:
we glue three cylinders $C_1,C_2,C_3$ with the aforementioned flat metric by
parameterizing each $\partial C_i$  as $\{(0,t)|0\leq t<1\}$.
We glue $\{(0,t) \mid 0< t < \frac{1}{2}\}$ of $C_i$ to
$\{(0,t) \mid \frac{1}{2}< t < 1\}$ of $C_{i+1}$,
and $\{(0,t) \mid \frac{1}{2}< t <1\}$ to $\{(0,1) \mid 0< t < \frac{1}{2}\}$ of $C_{i-1}$,
where index is written in mod $3$. Then
the conformal structure is realized by the glued flat metric. The metric
is singular only at two points $p, \, \overline p \in \Sigma$ which
lie on the boundary circles of $\Sigma_i$. Therefore the conformal
structure induced from the metric naturally extends
over the two points $p, \overline p$ by unique continuation.
One important property of this {\it singular} metric is the flatness
everywhere except at the two points $p, \overline
p$ where the metric is singular but Lipschitz.

We recall from Lemma \ref{lem:diffeo} that each separable orientable surface $M$
can be also obtained by iteratively gluing the following building blocks:
\begin{enumerate}
\item a pair of pants,
\item a cylinder with a genus,
\item an infinite cylinder.
\end{enumerate}

We take an area form on each building block $C_i$
so that near each boundary component it has the form
$$
\omega = \pi^*(ds \wedge dt)
$$
which can be smoothly glued to one another except at the two singular points
$p, \, \overline p$:
we arbitrarily smoothen the form near the singularities.

For a cylinder with a genus, we glue two of pairs of pants.

For an infinite cylinder in $(s,t) \in \RR_{\geq 0} \times S^1$,
we take the form $\omega=ds \wedge dt$.
We glue these building blocks and define a global area form
which we denote by
$$
\omega
$$
regarded as a symplectic form. We provide the following formal definition.

\begin{defn}[Standard surface] A \emph{standard surface} is a noncompact Riemann surface
$(M,\CT)$ equipped with a pair-of-pants decomposition whose Reeb graph
as a planar graph is induced by the level sets of a plurisubharmonic function
$\psi$ where each pair-of-pants is an elementary cobordism equipped with the symplectic form
described as above.
\end{defn}
By definition, a standard surface is a (tame) Weinstein surface whose Weinstein structure
$(M,\omega_\CT,\psi)$
depends only on $\CT$ up to Weinstein homotopy in the sense of \cite{cieliebak-eliashberg}.

We then take an $\omega$-compatible almost complex structure $J \in \CJ(\CT)$ and regard the triple $(M,\omega_\CT,J)$ as an almost K\"ahler manifold.

\subsection{Symplectomorphism classification}
In this subsection, we will show that our standard surface $M'$ which is constructed by attaching building blocks and has the same ideal boundary as $M$ is symplectomorphic to $M$ if every end of $M$ is of infinite volume.

\begin{lemma}\label{lem:endvolume}Let $(M,\omega)$ be a noncompact symplectic surface every end of which is of infinite volume. For every compact subset $C$ whose boundary components are diffeomorphic to $S^1$, given an arbitrary real number $v>0$ and a boundary component $\partial C_1$, there is a cylinder $D$ attached to $C$ along $\partial C_1$ with volume $v$.
On the other hand, given an arbitrary real number $0<m<\text{(Volume of C)}$ and a boundary component $\partial C_1$, there is a cylinder in $C$ one of whose boundary is $\partial C_1$ with volume $m$.
\end{lemma}
\begin{proof}Since every end of $M$ is of infinite volume, each connected component of $M\setminus C$ attached to $C$ along $\partial C_1$ is of infinite volume. Since this component is of infinite volume, we can find a cylinder attached to $C$ with arbitrary volume by extending it until its volume reach $v$. For the second part, extend from $\partial C_1$ to the inside of $C$ until its volume reaches $m$.
\end{proof}

To conclude this chapter, we will use the following lemma of Greene-Shiohama \cite{Greene-Shiohama}.

\begin{lemma}\label{lem:shio}
Suppose $M$ is a noncompact orientable manifold of dimension $n$ and $\{K_i|i=1,2,\dots\}$ is a sequence of $n$-dimensional compact connected submanifolds-with-boundary such that $\bigcup^{\infty}_{i=1}K_i=M$ and $K_i\cap K_j$ for all $i,j,i\neq j$, is either empty or is an $(n-1)$-dimensional submanifold of $M$ which is contained in the boundary of $K_i$ and also in the boundary of $K_j$. Suppose also that $\omega$ and $\tau$ are volume forms on $M$ such that $\int_{K_i}\omega=\int_{K_i}\tau$ for each $i=1,2,\dots$. Then there is a diffeomorphism $\varphi:M\rightarrow M$ such that $\varphi^*\omega=\tau$.
\end{lemma}

\begin{theorem}
A noncompact symplectic surface every end of which has infinite volume is symplectomorphic to a standard surface.
\end{theorem}
\begin{proof} By Lemma \ref{lem:diffeo}, there is a standard surface $(M',\omega')$ such that given manifold is diffeomorphic to $(M,\omega)$. Let $\tau$ be a volume form on $M'$ given from pulling back $\omega$ along the diffeomorphism. In the proof of Lemma \ref{lem:diffeo}, each closure of $M'_{k+1}-M'_k$ consists of finitely many building blocks. We can order them and denote as $\{K_i|i=1,2,\dots\}$. Then their union becomes $M'_1,M'_2,\dots$ and $\bigcup^{\infty}_{i=1}K_i=M'$ holds. Note that each boundary component of $K_i$ is diffeomorphic to $S^1$. Intersection of $K_i$ and $K_j$ for $i\neq j$ is an empty set or a disjoint union of their boundary components. Compare the volume of $K_1$ w.r.t. $\tau$ and $\omega$. If they are equal, move to $K_2,K_3,\dots$. Else, find a cylinder attached to $K^1$ or inside of $K^1$ with the volume equal to the difference using Lemma \ref{lem:endvolume}. By varying the diffeomorphism, shrink or expand near the boundary of $K^1$ and get a new volume form, say $\tau_1$. Then $\tau_1$ and $\omega'$ give us the same volume on $K^1$. Repeating this, we may assume that volume on $K_i$ with respect to $\tau$ and $\omega'$ is equal for all $i\geq 1$. Therefore, our $\{K_i \mid i=1,2,\dots\}$ satisfies the condition of Lemma \ref{lem:shio} and we can get a symplectomorphism between two surfaces.
\end{proof}

Therefore, each symplectomorphism class of noncompact Riemann surface with infinite volume ends contains a standard surface. From now on, we will regard standard surface as a representative of such a symplectomorphism class.

\section{Surfaces with fractal ends}\label{sec:fractal_generators}

In Part III of the present paper, we will provide some algebraic description of the endomorphism algebra of
the Fukaya category. In general the end structure can be
too complex to provide a simple description of the algebra. For example
the ideal boundary of a surface may have uncountably many ends. Because of this, we will
restrict ourselves to some good cases for which they admit some
reasonably concrete description. These are the cases what we call are of \emph{fractal ends}
whose definition is in order. The Cantor set is the simplest and most famous example of a fractal set
and we can extract similar fractal structure from some of its subsets.

\begin{defn}\label{defn:fractal}
Let $\CS$ be the sub-collection of $\CI$ given in \eqref{eq:cS}.
Given $I\in \cS$, a finite set of mutually disjoint intervals $\{I_1,\dots,I_n\}\subset \cS$
is called a finite cover of $I$ if
$I_i\subsetneq I$ for all $i$ and $I\cap B(M)\subset I_1\cup \dots \cup I_n$.
\begin{enumerate}
\item An interval $I\in \cS$ is called \emph{fractal} if there is a finite cover $\{I_1,\dots,I_n\}$ of $I$
such that $\{K\in \cS \mid K\subset I \}$ is homeomorphic to $\{K\in \cS \mid K\subset I_i \}$ for all $i$.
We will call such a cover a \emph{fractal cover} of $B(M)$.
\item An interval $I\in \cS$ is called \emph{quasi-fractal} if
there is a finite cover $\{I_1,\dots,I_n\}$ of $I$ such that there is a set of intervals $J_i\in \cS$ for $i = 1, \ldots, n$ such that
\begin{enumerate}
\item
$I_i\subsetneq J_i\subset I$,
\item  $\{K\in \cS|K\subset J_i \}$ is homeomorphic to $\{K\in \cS|K\subset I_i \}$.
\end{enumerate}
We will call such a cover a \emph{quasi-fractal cover}.
\end{enumerate}
\end{defn}

\begin{defn}[Fractal and quasi-fractal surface]\label{defn:fractal_surface}
We say a finite set of mutually disjoint intervals $\{I_1,\dots,I_n\}\subset \cS$
a finite cover of the ideal boundary $B(M)$ of $M$ if it is a finite cover of $[0,1]\in \cS$.
\begin{enumerate}
\item A separable surface $M$ with the ideal boundary $\cS$ is called (finitely generated) fractal
if there is a finite cover $\{I_1,\dots,I_n\}$ of $B(M)$ such that all $I_i$ are fractal.
\item A separable surface $M$ with the ideal boundary $\cS$ is called (finitely generated) quasi-fractal
if $[0,1]\in \cS$ is quasi-fractal.
We will call a quasi-fractal cover of $B(M)$ as a quasi-fractal cover of $M$.
\end{enumerate}
\end{defn}

\begin{rmk}

\begin{enumerate}[(a)]
\item If an infinite cylinder is attached to an interval $I\in \cS$,
i.e. $\xi_\cS(J)=0$ and $\chi_\cS(J)=0$ for all $J\subset I$, $I$ is fractal.
Therefore, every finite type surface is fractal.
\item A fractal surface is quasi-fractal.
A quasi-fractal surface is fractal if and only if $J_i=J_j$ or $J_i\cap J_j=\emptyset$ for all $i\neq j$.
\item If an interval $I$ is fractal [resp.quasi-fractal] with a fractal [resp.quasi-fractal] cover
$\{I_1,\dots,I_n\}$, then each $I_i$ will also be fractal [resp.quasi-fractal].
\end{enumerate}

\end{rmk}

Now we provide a few examples of fractal surfaces.

\begin{example}\label{exam:fractal}
\begin{enumerate}
\item  A surface $M$ with $B(M)=\CC$, $B'(M)=\emptyset$ is fractal.
$\CS=\mathcal{I}$ and $\xi_\CS(I)=1$, $\chi_\CS(I)=0$
for all $I\in \CS$.
This surface is generated by attaching pair of pants repeatedly.
Then $M$ is finitely generated fractal with cover $\{[0,1]\}$.
$[0,1]$ has a fractal cover $\{[0,1/3],[2/3,1]\}$.
\item  A surface with two ends, one planar and one non-planar.
We may write it as
$$
S=\{[0,1]\}\cup\{[0,1/3],[0,1/9],\dots\}\cup\{[2/3,1],[8/9,1],\dots\},
$$
with
\beastar
\chi_S([0,1/3]) & = & \chi_S([0,1/9])=\dots=0, \\
\chi_S([2/3,1]) & = & \chi_S([8/9,1])=\dots=2.
\eeastar
Then $M$ is finitely generated fractal with cover $\{[0,1/3],[2/3,1]\}$.
$[0,1/3]$ is of finite type and $[2/3,1]$ has a fractal cover $\{[8/9,1]\}$.
\end{enumerate}
\end{example}

Here are some quasi-fractal surfaces which are not fractal.
\begin{example}\label{exam:quasi-fractal}
\begin{enumerate}
\item
A surface $M$ without genus whose ideal boundary has one limit point.
$B(M)=\{0,\frac{2}{3},\frac{8}{9},\dots\}\cup \{1\}$.
Then $M$ is not finitely generated fractal, but finitely generated quasi-fractal with quasi-fractal cover $\{[0,\frac{1}{9}],[\frac{2}{3},1]\}$.
$\{K\in \cS\mid K\subset [0,\frac{1}{9}]\}$ is homeomorphic to $\{K\in \cS|K\subset [0,\frac{1}{3}]\}$ and
$\{K\in \cS \mid K\subset [\frac{2}{3},1]\}$ is homeomorphic to $\cS$.
\item
Another surface $M$ without genus whose ideal boundary has one limit point.
$B(M)=\{0,1,\frac{2}{3},\frac{7}{9},\frac{20}{27},\ldots\}\cup \{\frac{3}{4} \}$.
It is homeomorphic to the surface given in Item (1) above.
Then $M$ is not finitely generated fractal, but finitely generated quasi-fractal with quasi-fractal cover $\{[0,\frac{1}{9}],[\frac{2}{3},1]\}$.
$\{K\in \cS \mid K\subset [0,\frac{1}{9}]\}$ is homeomorphic to $\{K\in \cS \mid K\subset [0,\frac{1}{3}]\}$ and
$\{K\in \c \mid K\subset [\frac{2}{3},1]\}$ is homeomorphic to $\cS$.
\end{enumerate}
\end{example}

\begin{figure}

    \centering
    \subfloat{{
    \resizebox{3.5cm}{4cm}{
\begin{tikzpicture}

\draw[dashed] (8,3) arc (0:360:0.5 and 0.1);
\draw[] (6,2.5) arc (0:360:0.5 and 0.1);
\draw[] (5,1) arc (0:360:0.5 and 0.1);
\draw[] (4,-0.5) arc (0:360:0.5 and 0.1);

\draw[rounded corners=13pt] (7,3)-- ++(0,-1)--++(-1,0)--++(0,0.5);
\draw[rounded corners=13pt] (8,3)-- ++(0,-1)--++(-1,-0.5)--++(0,-0.5);

\draw[rounded corners=13pt] (5,2.5)-- ++(0,-0.5)--++(1,-0.5)--++(0,-0.5);

\draw[rounded corners=13pt] (6,1)-- ++(0,-0.5)--++(-1,0)--++(0,0.5);
\draw[rounded corners=13pt] (7,1)-- ++(0,-0.5)--++(-1,-0.5)--++(0,-0.5);
\draw[rounded corners=13pt] (4,1)-- ++(0,-0.5)--++(1,-0.5)--++(0,-0.5);

\draw[rounded corners=13pt] (5,-0.5)-- ++(0,-0.5)--++(-1,0)--++(0,0.5);
\draw[rounded corners=20pt] (6,-0.5)-- ++(0,-0.8)--++(-1.5,-0.5)--++(-1.5,0.5)--++(0,0.8);

\draw[fill=black] (7.5,3.5) circle (1pt);
\draw[fill=black] (7.5,3.7) circle (1pt);
\draw[fill=black] (7.5,3.9) circle (1pt);

\end{tikzpicture}
    }
    }}
    \qquad
    \subfloat{{

    \resizebox{3cm}{4cm}{
\begin{tikzpicture}

\draw[dashed] (6,3) arc (0:360:0.5 and 0.1);
\draw[] (4,2.5) arc (0:360:0.5 and 0.1);
\draw[] (7,1) arc (0:360:0.5 and 0.1);
\draw[] (4,-0.5) arc (0:360:0.5 and 0.1);

\draw[rounded corners=13pt] (5,3)-- ++(0,-1)--++(-1,0)--++(0,0.5);
\draw[rounded corners=13pt] (6,3)-- ++(0,-1)--++(-1,-0.5)--++(0,-0.5);

\draw[rounded corners=13pt] (3,2.5)-- ++(0,-0.5)--++(1,-0.5)--++(0,-0.5);

\draw[rounded corners=13pt] (6,1)-- ++(0,-0.5)--++(-1,0)--++(0,0.5);
\draw[rounded corners=13pt] (7,1)-- ++(0,-0.5)--++(-1,-0.5)--++(0,-0.5);
\draw[rounded corners=13pt] (4,1)-- ++(0,-0.5)--++(1,-0.5)--++(0,-0.5);

\draw[rounded corners=13pt] (5,-0.5)-- ++(0,-0.5)--++(-1,0)--++(0,0.5);
\draw[rounded corners=20pt] (6,-0.5)-- ++(0,-0.8)--++(-1.5,-0.5)--++(-1.5,0.5)--++(0,0.8);

\draw[fill=black] (5.5,3.5) circle (1pt);
\draw[fill=black] (5.5,3.7) circle (1pt);
\draw[fill=black] (5.5,3.9) circle (1pt);

\end{tikzpicture}
    }
    }}
    \caption{Quasi-fractal surfaces in Example \ref{exam:quasi-fractal}}
\end{figure}
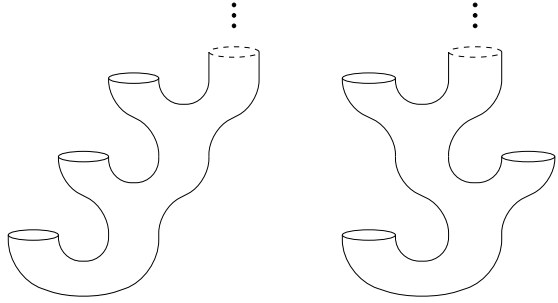

Since the symplectomorphism type of a standard surface is classified by
the homeomorphism type of its ideal boundary, from now on
we may assume that the homeomorphisms mentioned in Definition \ref{defn:fractal}
are translation maps in $\RR$ of the form $x\mapsto 3^k x+c$, where $k\in \ZZ$, $c\in \RR$
and we regard the intervals as subsets of $\RR$.
This translation map can be extended to a translation map on $\RR^2$ which is a homeomorphism
between two disks in $\cD$ defined in the proof of Lemma \ref{lem:diffeo}.
This corresponds to raising the level of the Morse function $f$
given in Definition \ref{defn:reeb-graph} by $k$.
Moreover, if $I$ has a finite cover $\{I_1,\dots,I_n\}$, $I_i\subsetneq I$ implies $k<0$ and
$|\bigcap_{j=1}^{\infty} f^j(I)|=1$.
This provides the description of a fractal structure of the standard surface
that is fractal.

\begin{prop}\label{prop:frac_span}
Let a standard surface $(M,\omega)$ be given and let
$\{I_1,\ldots,I_n\}$ be a finite cover of $B(M)$ such that
$I_i$ is fractal with a fractal cover $\{I_{i,1},\ldots,I_{i,k_i}\}$ for all $1\leq i \leq n$. Then
the following hold:
\begin{enumerate}
\item There is a symplectomorphism $\phi_{i,j}:{I_i}_*\rightarrow {I_{i,j}}_*$ for all $1\leq j \leq k_i$.
\item For all $i$ there is a compact subset $C_i$ of ${I_i}_*$ such that
symplectomorphism image of $C_i$ covers ${I_i}_*$.
\end{enumerate}
\end{prop}
\begin{proof}
Since we constructed a standard surface by repeatedly attaching a sequence of
building blocks along the boundary, we have a trivial one-to-one correspondence between two building blocks of the same type, which will be a translation map if we embed our standard surface in $\RR^3$. This map is trivially symplectomorphism. Also by the construction of standard surfaces, if homeomorphism between $\{J\in \cS|J\subset I_i \}$ and $\{J\in \cS|J\subset I_{i,j} \}$ is of the form $f(x)=3^k x+c$, we have the same procedure to construct ${I_i}_*$ and ${I_{i,j}}_*$.
Use the symplectomorphism between same building blocks and gluing to obtain a symplectomorphism $\phi_{i,j}:{I_i}_*\rightarrow {I_{i,j}}_*$.

For the second statement, we will define $C_i$ as
\be
C_i:={I_i}_*\setminus \bigcup_{j=1}^{i_k}\text{\rm Int}({{I_{i,j}}_*})={I_i}_*\setminus \bigcup_{j=1}^{i_k}\text{\rm Int}(\phi_{i,j}({{I_{i}}_*}))
\ee
In other words, $C_i$ is a union of building blocks $B_J$ attached to interval $J$ such that $J\subset I_i$ and $J\not\subset I_{i,j}$ for all $j$. (See Definition \ref{defn:building-block} for the definition
of $B_J$.)
We can also find a symplectomorphism between ${I_{i,j}}_*$ and $\phi_{i,l}({I_{i,j}}_*)$ for all $j,\,l$ and we will abuse notation and also call such symplectomorphism as $\phi_{i,l}$.

Then
\be
C_i\cup \bigcup_{j=1}^{i_k}{{\phi_{i,j}}(C_i)}={I_i}_*\setminus \bigcup_{j,l=1}^{i_k}
\text{\rm Int}({\phi_{i,l}} \circ {\phi_{i,j}}({I_i}_*))
\ee
and we can repeat it.
Also note that $\{{{\phi_{i,j_m}}\circ \dots \circ {\phi_{i,j_1}}({I_i}_*)}\}_{m\geq 0}$ for a sequence
$\{j_1,j_2,\ldots\}$ is an end. Therefore,
$$
\left\{\bigcup_{j_1,\ldots,j_m=1}^{i_k}{{\phi_{i,j_m}}\circ \dots \circ
 {\phi_{i,j_1}}(C_i)})\right\}_{m\geq 0}
$$
is a compact exhaustion of ${I_i}_*$ and the images of $C_i$ under
the symplectomorphisms covers ${I_i}_*$.
\end{proof}

We can do the similar thing for the case of quasi-fractal surface, but we will omit the proof.

\begin{lemma}\label{lem:quasifrac_span}
Let a standard surface $(M,\omega)$ be given.
Suppose $M$ is quasi-fractal with a quasi-fractal cover $\{I_1,\ldots,I_n\}$ such that
there exists $I_i,\,J_i\in \cS$ with $I_i\subsetneq J_i$ and $\{K\in \cS|K\subset J_i \}$ is homeomorphic to $\{K\in \cS|K\subset I_i \}$. Then there is a symplectomorphism $\phi_i:{J_i}_*\rightarrow {I_i}_*$ for all $i$.
Moreover, for all $i$ there is a compact subset $C_i$ of ${J_i}_*$ such that symplectomorphism image of $C_1,\ldots,C_n$ covers $\bigcup_{i=1}^{n}{J_i}_*$.

\end{lemma}

Therefore, quasi-fractal surface can be spanned by the symplectomorphism image of finitely many compact subsets.

\part{Definition of Fukaya category of infinite type surfaces}

Now we specialize to the case of a noncompact exact symplectic manifold
$(M,\omega)$ with $\omega = d\alpha$. We start with introducing a
general definition of tame symplectic manifolds for which the standard
analytic package of pseudoholomorphic curves can be applied.

\section{Liouville tame manifolds and gradient-sectorial Lagrangians}

We first recall the notion of tame symplectic manifolds: A symplectic manifold
$(M,\omega)$ is called tame if it admits an almost complex structure $J$ such
that the bilinear form $\omega(\cdot, J \cdot)=: g_J $ defines a Riemannian metric
such that its injectivity radius is bounded from below and its curvature is
uniformly bounded. As usual, we denote by
$$
\cJ_\omega
$$
the set of almost complex structures tame to $\omega$.

\subsection{Plurisubharmonic functions and $C^0$-estimates in general}
\label{subsec:quasi}

To study $C^0$ control of the solutions of perturbed Cauchy-Riemann equations,
we will use a class of barrier functions whose (classical) Laplacian is pinched from below.

The following definition is motivated by that of \cite{oh:intrinsic}.

\begin{defn} Let $J$ be any tame almost complex structure of $(M,\omega)$.
We call a smooth function $\psi: M \to \R$ that is
tame and $J$-\emph{plurisubharmonic} or a \emph{$J$-convex}
if it satisfies the following:
\begin{enumerate}
\item There exists a constant $C > 0$ such that
$$
\|\psi\|_{C^2} < C.
$$
\item There exists a compact subset $K \subset M$ such that
$$
-d(d\psi \circ J) = g \omega
$$
for some nonnegative function $g$.
on $M \setminus K$.
\end{enumerate}
When the above holds, we call the pair $(\psi,J)$ a \emph{pseudoconvex} pair.
\end{defn}
Note that if $K = M$, it is nothing but the definition of \emph{plurisubharmonic function}
for the complex structure $J$. By Condition (1), $g$ appearing in Statement (2) satisfies
$\|g\|_{C^0} \leq C'$ for some constant $C'$.

\begin{defn}[Tame Liouville manifolds] We call a tame exact symplectic manifold $(M,d\alpha)$,
not necessarily of finite type, a \emph{Liouville-tame symplectic manifold} if the following hold:
 \begin{enumerate}
 \item It admits a pseudoconvex pair $(\psi, J)$.
\item The associated Liouville vector field is gradient-like for the function $\psi$.
\end{enumerate}
\end{defn}
We define the ideal boundary of each cylindrical end $e$ of $M$ to be the set of
equivalence classes of Liouville rays of the end $e$ and denote by $\del_{e\infty}M$.

Now we recall the class of \emph{gradient-sectorial Lagrangian branes} with respect to
the given pseudoconvex pair $(\psi,J)$ from \cite{oh:intrinsic}
and define the notion of $\psi$-wrapped Fukaya category.
 We consider its normalized gradient vector field
\eqn
Z_{\psi}: = \frac{\grad \psi}{|\grad \psi|^2}
\eqnd
with respect to the usual metric
$$
g_J(v, w): = \frac{d\alpha(v, J w) + d\alpha(w,Jv)}{2}.
$$
\begin{defn}[Gradient-sectorial Lagrangian branes]
Let $(M,\alpha)$ be a Liouville-tame symplectic manifold equipped with a
pseudoconvex pair $(\psi,J)$.
We say that an exact Lagrangian submanifold $L$ of $(M,\alpha)$ is \emph{gradient-sectorial}
if
\begin{enumerate}
\item $L \subset \Int M \setminus \del M$ and $\text{\rm dist}(L,\del M) > 0$.
\item There exists a sufficiently large $r_0> 0$ such that
$L \cap {\psi}^{-1}([r_0,\infty))$ is $Z_{\psi}$-invariant, i.e., $Z_{\psi}$ is tangent to
$L \cap {\psi}^{-1}([r_0,\infty))$.
\end{enumerate}
\end{defn}

We now specialize to the case of 2 dimensional surfaces.

\subsection{Construction of tame $J$-plurisubharmonic functions}

We now explain the procedure of construction of plurisubharmonic exhaustion functions
on a noncompact orientable surface equipped with a hyperbolic
structure, utilizing the details of Weinstein's contact surgery \cite{alan:surgery}.

\begin{prop}\label{prop:plurisubh_const}
Let $(M,\omega)$ be a standard surface constructed in Lemma \ref{lem:diffeo}. There is a plurisubharmonic function $\psi:M\rightarrow \RR$ defined on $M'$ such that each building block is a connected component of $\psi^{-1}([k,k+1])$ for some integer $k\geq 1$.
\end{prop}

\begin{proof}

First, embed the disk $D$ into $\RR^3$, with the height function $\psi:D^2\rightarrow [0,1]$
where $\psi^{-1}(0)=\{0\}$ and $\psi^{-1}(1)=\partial D$ and $\psi$ is convex.
Then it is a plurisubharmonic function on $D$.
Recall that in Lemma \ref{lem:diffeo}, we attached building blocks to $D$ along $\partial D$.
If an infinite cylinder is attached, the resulting surface is an infinite disk and it has no boundary.
Extend this height function $\psi$ to the whole surface and we are done.

If we attach a pair of pants, we do the contact surgery as in \cite{alan:surgery}.
We define a standard handle which we would use for the surgery.
In the standard symplectic space $\RR^2$ with canonical symplectic form $\omega=dx \wedge dy$ and
Liouville vector field $\xi=-2x\frac{\partial}{\partial x} +y\frac{\partial}{\partial y}$,
it is the negative gradient with respect to the standard metric of the Morse function
$f=x^2-\frac{1}{2}y^2+2$. Then the unstable manifold is $E_{-}=\{x=0\}$ and descending sphere is $S_{-}=E_{-}\cap f^{-1}(1)=\{(0,\pm \sqrt{2})\}$

A standard handle is a region in $\RR^2$ bounded by a neighborhood of $S_{-}$ in $f^{-1}(1)$ together with a connecting manifold $\sigma$ diffeomorphic to $S^0\times D^1$. We may choose this handle so that it is transverse to the Liouville vector field $\xi$ and so that its intersection with $f^{-1}(1)$ is contained in
an arbitrarily small neighborhood of the descending sphere. Note that $f$ is plurisubharmonic in this handle.

By Theorem 5.1 of \cite{alan:surgery}, we may attach this handle to our disk $D$ and $\psi$ will be extended to the resulting surface using $f$.
We reparametrize the domain or rescale the value if needed and assume that $\psi=2$ on the boundary.
If we want to attach a cylinder with a genus, then we attach the handle twice and get the desired result.
Plurisubharmonicity is a local property and so $\psi$ is plurisubharmonic on the whole surface.

By our construction, $\psi^{-1}([0,1])$ is a disk and connected components of $\psi^{-1}([1,2])$, $\psi^{-1}([2,3])$, $\dots$ are building blocks listed above and we are done(In this case, we regard attaching an infinite cylinder as attaching a finite cylinder repeatedly).
\end{proof}

The Morse function
$\psi$ in the proposition above provides us with a compact exhaustion sequence
$$
\psi^{-1}([0,1])\subset \psi^{-1}([0,2]) \subset \cdots
$$
of $M$. By the construction of $\psi$, for any $k\in \ZZ_{>0}$, $\psi^{-1}([0,k])$ is the result of attaching building blocks $k-1$ times to a disk. Also, each boundary component of $\psi^{-1}([0,k])$ is a simple closed curve and $\psi^{-1}([0,k])$ is a Liouville subdomain of $M$.

\begin{defn}Let $M$ be a standard surface with $\psi$ defined above.
We denote by
\be\label{eq:Mn}
M^{\leq n}:=\psi^{-1}([0,n])
\ee
\be
M^{\geq n}:=\psi^{-1}([n,\infty))
\ee
for $n\in \ZZ_{>0}$.
\end{defn}

Combined with the classification of surfaces, what we have done so far in this section
can be written as the following theorem.

\begin{theorem}
Let $M$ be an orientable separable surface without boundary.
Then there is a standard surface $(M',\omega)$ homeomorphic to $M$ which is
a Liouville-tame symplectic surface equipped with a pseudoconvex pair $(J,\psi)$.
\end{theorem}

\section{Maximum principle, $C^0$-estimates and energy estimates}
\label{sec:maximum-principle}

\subsection{Hamiltonian-perturbed Cauchy-Riemann equation}

In this subsection only, we consider the general Hamiltonian-perturbed
Cauchy-Riemann equation which is given as follows:
\begin{equation} \label{eq:cr}
\begin{cases}
u: \dot \Sigma \rightarrow M\\
u(\partial \dot \Sigma)\subset L\\
\lim_{s\rightarrow \pm \infty}u(\epsilon^k(s,\cdot))=x^k\in \mathfrak{X}\\
(du-X\otimes \alpha)^{0,1}=0
\end{cases}
\end{equation}
where $X = X_H$ for a domain-dependent function $H = H(z,x): \dot \Sigma \times M \to \R$
and $\alpha$ is a one-form on $\dot \Sigma$ that satisfies the conditions imposed in
\cite{abouzaid-seidel} which we do not elaborate leaving the details thereto.

The geometric and topological energies of a solution of above equation are defined by
\begin{equation}
E^{geom}(u)=\int_{\dot \Sigma}\frac{1}{2}|du-X\otimes \alpha|^2=\int_{\dot \Sigma}u^*\omega-u^*dH\wedge \alpha
\end{equation}
\begin{equation}
E^{top}(u)=\int_{\dot \Sigma}u^*\omega-d(u^*H\cdot \alpha)=E^{geom}(u)-\int_{\dot \Sigma}u^*H\cdot d\alpha
\end{equation}
When $H\geq 0$ and $\alpha$ is sub-closed , we get
\begin{equation}\label{eq:energyineq}
0\leq E^{geom}(u) \leq E^{top}(u)
\end{equation}
The first equality holds if and only if
 $du=X\otimes \alpha$ and the second holds iff $d\alpha =0$ \cite{abouzaid-seidel}.

Since we assume the Lagrangian is exact, i.e., $\iota_L^*\alpha = dh$, then we
define action of $\gamma \in \mathfrak{X}$ by
\begin{equation}
\mathcal{A}(\gamma)=\int^{1}_{0}-\gamma^*\theta+H(x(t))dt+h(\gamma(1))-h(\gamma(0))
\end{equation}

Then for any solution of (\ref{eq:cr}), the following holds.
\begin{equation}\label{eq:topenergy}
E^{top}(u)=\mathcal{A}(\gamma^0)-\sum^d_{k=1}\mathcal{A}(\gamma^k)
\end{equation}
Following as in \cite{oh:sectorial,oh:intrinsic},
we now explain how the pseudoconvex pair $(\psi,J)$ are paired with
 gradient-sectorial Lagrangian branes so that they become amenable to the
strong maximum principle and hence give rise to fundamental confinement results
for the $J$-holomorphic curve equation and other relevant
Floer-type equations such as the perturbed $J$-holomorphic curves associated
to the wrapping Hamiltonians of the type $H = \rho(\psi)$ for a function $\rho$ with $\rho' > 0$.

After this general set of Hamiltonian-perturbed Cauchy-Riemann equation is mentioned,
we will restrict ourselves to the case $H = 0$ for our purpose of construction of
a Fukaya category on $M$.

\subsection{Pseudoconvex pair, gradient sectorial Lagrangians and strong maximum principle}

Let $J$ be a tame almost complex structure of $M$.
Consider a $(k+1)$-tuple $(L_0, \ldots, L_k)$
of gradient-sectorial Lagrangian submanifolds. We denote
$$
\Sigma = D^2 \setminus \{z_0, \ldots, z_k\}
$$
and equip $\Sigma$ with strip-like coordinates $(\tau,t)$
with $\pm \tau \in [0,\infty)$ and $t \in [0,1]$ near each $z_i$.

\begin{rmk}
Here and hereafter the suffix $i$ is regarded as modulo $k+1$.
\end{rmk}

Then for a given collection of intersection points $p_i \in L_i\cap L_{i+1}$ for
$i = 0, \ldots, k$,
we wish to study maps $u: \Sigma \to \Mliou$ satisfying the Cauchy-Riemann equation
\eqn\label{eq:unwrapped-structure-maps}
\begin{cases}
\delbar_J u = 0\\
u(\overline{z_iz_{i+1}}) \subset L_i \quad & i = 0, \ldots ,k\\
u(\infty_i,t) = p_i, \quad & i = 0, \ldots ,k.
\end{cases}
\eqnd
The following theorem is proved in \cite[Theorem 8.7]{oh:sectorial} for the
particular case $\psi = \mathfrak{s} $ of the end-profile function but
the same proof applies to general pseudoconvex pair. For readers' convenience, we
duplicate the proof here in the current context.

\begin{theorem}\label{thm:pseudoconvex} Let $(\psi, J)$ be a pseudoconvex pair.

Let $u$ be a solution to~\eqref{eq:unwrapped-structure-maps}.
Then there exists a sufficiently large $r > 0$ such that
\eqn\label{eq:confinement}
\Image u \subset ({\psi} )^{-1}((-\infty,r])
\eqnd
\end{theorem}
\begin{proof}
Since a neighborhood of $\del_\infty M$ is exhausted by the family of compact subsets
$$
{\psi}^{-1}((-\infty,r])
$$
for $r \geq 0$, it is enough to prove \eqref{eq:confinement}
for some $r > 0$.
We first recall that $du$ is $J$-holomorphic and satisfies
$
- d(d {\psi}    \circ J) \geq 0
$
from the definition of pseudoconvex pair $( {\psi}  , J)$.
Since $u$ is $J$-holomorphic, we obtain
$$
d\left({\psi}   \circ u\right) \circ j = d{\psi}    \circ J \circ du
= u^*(d{\psi}    \circ J)
$$
By taking the differential of the equation, we derive
$$
-d\left(d\left({\psi}   \circ u\right) \circ j\right) = - u^*(d(d{\psi}    \circ J)) \geq 0.
$$
In particular, the function ${\psi}   \circ u$ is a subharmonic function and cannot carry an interior
maximum on $\RR \times [0,1]$ by the maximum principle.

Next we will show by the strong maximum principle that $u$ cannot have a boundary maximum
in a neighborhood of
$
\del_\infty M \cup \del M
$
either. This will then enable us to obtain a
$C^0$ confinement result
$$
\Image u \subset \{{\psi}   \leq r_0\}
$$
for any finite energy solution $u$ with fixed asymptotics given in \eqref{eq:unwrapped-structure-maps}
provided $r_0$ is sufficiently large.

Now suppose to the contrary that ${\psi}   \circ u$ has a boundary local maximum
point $z' \in \del D^2\setminus \{z_0,\ldots, z_k\}$. By the strong maximum principle, we must have
\eqn\label{eq:lambda(dudtheta)}
0 < \frac{\del}{\del \nu}({\psi} (u(z')))
= d{\psi}  \left(\frac{\del u}{\del \nu}(z')\right)
\eqnd
for the outward unit normal $\frac{\del}{\del \nu}|_{z'}$ of $\del \Sigma$,
unless ${\psi} \circ u$ is a constant function in which case there is nothing to prove.
Let $(r,\theta)$ be an isothermal coordinate of a neighborhood of $z' \in \del \Sigma$ in $(\Sigma,j)$
adapted to $\del \Sigma$, i.e., such that $\frac{\del}{\del \theta}$ is tangent to $\del \Sigma$ and
$|dz|^2 = (dr)^2 + (d\theta)^2$ for the complex coordinate $z = r+ i\theta$ and
\eqn\label{eq:normal-derivative}
\frac{\del}{\del \nu} = \frac{\del}{\del r}
\eqnd
along the boundary of $\Sigma$.
Since $u$ is $J$-holomorphic, we also have
$$
\frac{\del u}{\del r} + J \frac{\del u}{\del \theta} = 0.
$$
Therefore we derive
$$
d{\psi}  \left(\frac{\del u}{\del \nu}(z')\right)
= d{\psi}   \left(-J \frac{\del u}{\del \theta}(z')\right).
$$
By the ${\psi}  $-gradient sectoriality of $L$ and the boundary condition $u(\del \Sigma) \subset L$,
both $Z_{{\psi} }(u(z'))$ and $\frac{\del u}{\del \theta}(z')$ are contained in
$T_{u(z')}L$, which is a $d\alpha$-Lagrangian subspace. Therefore we have
\beastar
0 & = & d\alpha\left(Z_{{\psi} }(u(z')),\frac{\del u}{\del \theta}(z')\right)
= d\alpha\left(Z_{{\psi} }(u(z')), J \frac{\del u}{\del \nu}(z')\right)\\
& = & g_J\left(Z_{{\psi} }(z'), \frac{\del u}{\del \nu}(z')\right)
=\frac{1}{|Z_{\psi }(u(z'))|^2} d \psi \left(\frac{\del u}{\del \nu} (z')\right)
\eeastar
where the last equality follows from the definition of normalized gradient vector field $Z_{\psi }$.
This is a contradiction to \eqref{eq:lambda(dudtheta)} (unless $\psi \circ u$ is constant) and hence
the function ${\psi} \circ u$ cannot have a boundary maximum either.
This then implies
$$
\max {\psi \circ u} \leq \max \{{\psi} (p_i) \mid i=0,\ldots, k \}
$$
By setting
$$
r_0 = \max \{{\psi} (p_i) \mid i=0,\ldots, k \} + 1,
$$
we have finished the proof.
\end{proof}
We remark that the constant $ \max \{{\psi} (p_i) \mid i=0,\ldots, k \}$ (and so $r_0$)
depends only on the intersection set
$$
\bigcup_{i=0}^k L_i \cap L_{i+1}
$$
and not on the maps $u$ itself satisfying \eqref{eq:unwrapped-structure-maps}.

\section{Definition of Fukaya category of $(M,\omega)$}

We choose a countable collection of
properly embedded gradient-sectorial exact Lagrangian branes $\CL = \{L_i\}$.

In the study of (unwrapped) Fukaya category, we consider a disc $D^2$ with a finite number of boundary marked points $z_i \in \del D^2$ equipped with strip-like coordinates $(\tau,t)$ (or on the sphere $S^2$) with a finite number of marked points.
We denote by $\overline{z_iz_{i+1}}$ the arc-segment between $z_i$ and $z_{i+1}$, and $\tau = \infty_i$
the infinity in the strip-like coordinates at $z_i$.

Let $L_0, \ldots, L_k$ be a $(k+1)$-tuple of gradient-sectorial Lagrangian branes. We
denote
$$
\dot \Sigma = D^2 \setminus \{z_0, \ldots, z_k\}
$$
and equip $\Sigma$ with strip-like coordinates $(\tau,t)$
with $\pm \tau \in [0,\infty)$ and $t \in [0,1]$ near each $z_i$.

Then for a given collection of intersection points $p_i \in L_i\cap L_{i+1}$ for $i = 0, \ldots, k$,
we wish to study maps $u: \Sigma \to \Mliou$ satisfying the Cauchy-Riemann equation
\eqn
\begin{cases}
\overline \partial_J u = 0\\
u(\overline{z_iz_{i+1}}) \subset L_i \quad & i = 0, \ldots, k\\
u(\infty_i,t) = p_i, \quad & i = 0, \ldots, k.
\end{cases}
\eqnd

\subsection{Floer cochain complex}
\label{sec:chaincomplex}
In this subsection, we will describe construction of the boundary
map. Let $(L_i,\gamma_i)$ $i=0,1$ be a pair of exact sectorial Lagrangians.

Let $p, \, q \in L_0 \cap L_1$.

\begin{defn}
$CF(L_0,L_1)$
is a free $R$ module over the basis $p$
where $p \in L_0\cap L_1$ is an intersection point.
\end{defn}

We next take a grading to $L_i$.
It induces a grading of $p$, which gives the
graded structure on $CF(L_0,L_1)$
$$
CF(L_0,L_1) = \bigoplus_k CF^k(L_0,L_1)
$$
where
$
CF^k(L_0,L_1) = \operatorname{span}_R\{p\mid
\mu(p) = k\}.
$

Orientations of the Floer moduli space $\CM(p,q)$ obtain a system of integers
$n(p,q) = \#(\CM(p,q))$
whenever the dimension of $\CM(p,q)$ is zero.
Finally we define the Floer `boundary' map $\partial : CF(L_0,L_1) \to
CF(L_0,L_1)$ by the sum
\begin{equation}\label{eq:boundary}
\partial \langle p \rangle  = \sum_{q \in L_0\cap L_1}
n(p,q) \langle q \rangle .
\end{equation}

\begin{defn}\label{efilt}
We define the {\it energy filtration}
$
F^{\lambda}CF(L_0,L_1)
$
of the Floer chain complex
$CF(L_0,L_1)$ (here $\lambda \in \R$)
such that $p$ is in $F^{\lambda}CF(L_0,L_1)$
if and only if $\mathcal A(p) \ge \lambda$.
\end{defn}

It is easy to see the following from the definition of $\partial$ above:
\begin{lemma}\label{filtpres}
$$
\partial \left(F^{\lambda}CF(L_0,L_1)
\subseteq F^{\lambda}CF(L_0,L_1)\right).
$$
\end{lemma}

\subsection{Grading and orientations}

Let a standard surface $(M,\omega)$ be given.
For the $\ZZ_2$-grading, start with assigning orientation on each Lagrangians.
On each transversal intersection $x$ of two Lagrangians $L_0$ and $L_1$, we
pick a trivialization near $x$ such that
$$
T_xL_0\cong \RR, \quad T_xL_1\cong i\RR.
$$
We compare the orientation of this trivialization with that of $M$:
We set $\text{\rm deg}(x)=0$ if they agree and $\text{\rm deg}(x)=1$ otherwise.
For the $\ZZ$-grading, we will use Seidel's absolute grading
from \cite{seidel:graded} a brief explanation of how it goes is now in order.

In our two dimensional case, we use the following facts:
\begin{itemize}
\item Lagrangian Grassmannian $\Lambda(TM)$ of $(M,\omega)$ admits a fiberwise universal covering $\widetilde{\Lambda}(TM)$ if and only if $2c_1(M,\omega)$ goes to zero in $H^2(M;\ZZ)$. (See \cite[Lemma 2.2]{seidel:graded}.)
\item Any open 2-dimensional surface is homotopy equivalent to 1-dimensional
CW complex and so $c_1(TM) = 0$ in $H^2(M;\ZZ)$ by dimensional reason.
\end{itemize}

For each given Lagrangian submanifold $L\subset M$, we have a natural section
$$
s_L:L\rightarrow \Lambda(TM)|_L; \quad  s_L(x): =TL_x\in \Lambda(T_xM,\omega_x).
$$
We have a lift $\tilde{L}:L\rightarrow \Lambda(TM)$ of $s_L$ and the pair $(L,\tilde{L})$
is called a graded Lagrangian submanifold.

Let a pair of graded Lagrangians $(L_0,\tilde{L}_0)$ and $(L_1,\tilde{L}_1)$
which intersect transversally be given.
Pick any point $x\in L_0\cap L_1$ and choose two paths $\tilde{\lambda}_0,\tilde{\lambda}_1:[0,1]\rightarrow \widetilde{LGr}(T_xM)$ with
$\tilde{\lambda}_0(0)=\tilde{\lambda}_1(0)$, $\tilde{\lambda}_0(1)=\tilde{L}_0$ and
$\tilde{\lambda}_1(1)=\tilde{L}_1$. Let $\lambda_0,\lambda_1$ be the projections of these paths to $\Lambda(T_xM)$ and $\mu(\lambda_0,\lambda_1)$ be the Maslov index for paths.
We assign to $x$ an absolute index
\begin{equation}
\tilde{I}(\tilde{L}_0,\tilde{L}_1;x):=\frac{1}{2}-\mu(\lambda_0,\lambda_1)
\end{equation}
following \cite{seidel:graded}.

\subsection{Products}\label{subsec:prod}
Let $\frak L = (L_0, L_1, \cdots, L_k)$ be a
chain of compact Lagrangian submanifolds in $(M,\omega)$
that intersect pairwise transversely without triple intersections.
\par
Let $\vec z = (z_0,z_1,\cdots,z_k)$ be a set of distinct points on $\partial D^2
= \{ z\in \C \mid \vert z\vert = 1\}$. We assume that
they respect the counter-clockwise cyclic order of $\partial D^2$.
The group $PSL(2;\R)\cong \operatorname{Aut}(D^2)$ acts on the set
in an obvious way. We denote by $\mathcal M^{\text{main},\circ}_{2+1}$ be
the set of $PSL(2;\R)$-orbits of $(D^2,\vec z)$.
\par
In this subsection, we consider only the case $k \geq 2$
since the case $k=1$ is already discussed in the last subsection.
In this case there is no automorphism on the domain $(D^2, \vec z)$, i.e.,
$PSL(2;\R)$ acts freely on the set of such $(D^2, \vec z)$'s.
\par
Let
$p_{j} \in L_j \cap L_{j-1}$
($j = 0,\cdots ,k$), be a set of intersection points.
\par
We consider the pair $(w;\vec z)$ where $w: D^2 \to M$ is a
pseudo-holomorphic map that satisfies the boundary condition
\begin{subequations}\label{54.15}
\begin{eqnarray}
&w(\overline{z_{i}z_{i+1}}) \subset L_i, \label{54.15.1} \\
&w(z_{i}) = p_{i}\in L_i \cap L_{i+1}. \label{54.15.2}
\end{eqnarray}
\end{subequations}
We denote by $\widetilde{\CM}^{\circ}(\frak L, \vec p)$
the set of such
$((D^2,\vec z),w)$.
\par
We identify two elements $((D^2,\vec z),w)$, $((D^2,\vec z'),w')$
if there exists $\psi \in PSL(2;\R)$ such that
$w \circ \psi = w'$ and $\psi(z'_{j}) = z_{j}$.
Let ${\CM}^{\circ}(\frak L, \vec p)$ be the set of equivalence classes.
We compactify it by including the configurations with disc or sphere bubbles
attached, and denote it by ${\CM}(\frak L, \vec p)$.
Its element is denoted by $((\Sigma,\vec z),w)$ where
$\Sigma$ is a genus zero bordered Riemann surface with one boundary
components, $\vec z$ are boundary marked points, and
$w : (\Sigma,\partial\Sigma) \to (M,L)$ is a bordered stable map.
\par
We can decompose $\CM(\frak L, \vec p)$ according to the homotopy
class $B \in \pi_2(\frak L,\vec p)$ of continuous maps satisfying
\eqref{54.15.1}, \eqref{54.15.2} into the union
$$
\CM(\frak L, \vec p) = \bigcup_{B \in \pi_2(\frak L;\vec p)}
\CM(\frak L, \vec p;B).
$$
\par
In the case we fix an anchor $\gamma_i$ to each of $L_i$ and put $\CE =
((L_0,\gamma_0),\cdots,(L_k,\gamma_k))$, we consider only
admissible classes $B$ and put
\par
$$
\CM(\CE, \vec p) = \bigcup_{B \in \pi_2^{ad}(\CE;\vec p)}
\CM(\CE, \vec p;B).
$$
\begin{theorem}\label{58.21} Let $\frak L
= (L_0,\cdots,L_k)$ be a chain of
Lagrangian submanifolds and
$B \in \pi_2(\frak L;\vec p)$.
Then $\CM(\frak L, \vec p;B)$ is a smooth manifold
(with boundary and corners) of (virtual) dimension given by
\begin{equation}\label{dimensionformula}
\dim \CM(\frak L, \vec p;B) = \mu(\frak L,\vec p;B) + n + k-2,
\end{equation}
where $\mu(\frak L,\vec p;B)$ is the polygonal Maslov index of $B$.
\end{theorem}

\subsection{$A_\infty$-structure}

Under the same assumption as the above subsection, now consider solution $u$ of (\ref{eq:cr}). We may assume that each $x^i$ lies in $M^{\leq l}$. Suppose $u$ intersects $\partial M^{\leq l}$. Then we can apply Theorem \ref{thm:pseudoconvex} to the part of $u$ that gets mapped to $M\setminus M^{\leq l}$ and conclude that $u$ does not meet the interior of $M\setminus M^{\leq l}$. To remove transversality condition, we can to the same on little bit bigger component, say $M^{\leq l+\epsilon}$. $H$ is still linear on that boundary, although we may need some rescaling.
Therefore, every solution is contained in some compact Liouville subdomain of $M$ and we can regard them as solutions in $\widehat{M^{\leq l}}$. We can explicitly describe solutions out of $M^{\leq l}$ and compactness and transversality of moduli space can be derived as we did on Liouville domain.

Using this, we can also show $A_{\infty}$-associativity.

\begin{lemma}\label{lem:sol_infty}The maps $\mu^1,\mu^2,\dots$ satisfy the following $A_{\infty}$-associativity equations.
\begin{equation}\label{eq:a_infty asso}
{
\begin{split}
\sum_{d_++d_-=d+1}(-1)^{\overline{deg}(\gamma_1)+\dots+\overline{deg}(\gamma_i)}\mu^{d_+}(\gamma_d,\dots,\gamma_{i+d_-+1},\\
\mu^{d_-}(\gamma_{i+d_-},\dots,\gamma_{i+1}),\gamma_i,\dots,\gamma_1)=0
\end{split}
}
\end{equation}
\end{lemma}
\begin{proof}Fix inputs $\gamma_1,\dots,\gamma_d$ and consider coefficient of a chord $\gamma_0$ in the LHS of (\ref{eq:a_infty asso}). Then there exists $k$ such that every $\gamma_i$ is contained in $M^{\leq k}$. However, it cannot guarantee equations since we do not know about output of $\mu^d_-(\gamma_{i+d_-},\dots,\gamma_{i+1})$, say $x'$, and directly apply the result of this subsection. However, recall that each contribution to $A_{\infty}$-associativity equation comes from a polygon with one nonconvex corner. In that polygon, $x'$ should be located between two chords $\gamma_j,\gamma_{j+1}$ in a boundary Lagrangian of the polygon. The only possible case where $x'$ is out of $M^{\leq k}$ is the Lagrangian between $\gamma_j$ and $\gamma_{j+1}$ goes out $M^{\leq k}$ and meet saddle point and come back. However, such Lagrangian cannot be a boundary of polygon and that is not the case. Therefore, $x'$ is contained in $M^{\leq k}$ and we may regard it as $A_{\infty}$-associativity on $M^{\leq k}$, so the coefficient for $\gamma_0$ would be 0. We can do this for any output chords and we are done.
\end{proof}

By the Lemma \ref{lem:sol_infty}, our $\mu^k$ on wrapped Fukaya category of infinite type surface satisfies $A_\infty$-equation. And we regard it as an $A_\infty$-category.
Therefore, what we have done so far in this section can be written as the following theorem.

\begin{theorem}
Let $M'$ be a standard surface with a pseudoconvex pair $(J,\psi)$.
Then we can define a Fukaya category $Fuk(M')$ whose objects are gradient sectorial Lagrangians.
\end{theorem}

Now we have defined a Fukaya category on a standard surface. By the one-to-one
correspondence with hyperbolic Riemann surface structure and the standard surface representation of
the surface,
we derive that our definition of Fukaya category of infinite type surface
for given hyperbolic structure is well-defined.

\begin{theorem}
Let $(M_1,\omega_1)$, $(M_2,\omega_2)$ be orientable separable surfaces without boundary equipped with hyperbolic structure and associated symplectic forms respectively.
If $(M_1,\CT_1)$, $(M_2,\CT_2)$ are quasi-isometric, then
Fukaya categories $Fuk(M_1,\CT_1)$, $Fuk(M_2,\CT_2)$ are quasi-equivalent.
\end{theorem}
\begin{proof}
We equip tame almost complex structures $J_i$ and $J_i$-plurisubharmonic function $\psi_i$
to define $Fuk(M_i,\CT_i)$ respectively for $i=1,\,2$. Let $\phi:M_1 \to M_2$ be
an quasi-isometric symplectormorphism such that $\phi^*\omega_2 = \omega_1=: \omega$.
We now reduced the equivalence problem to the problem of two choice of  two tame almost complex structures $J_1$ and $J_2$ on a symplectic manifold $(M_1,\omega_1)$ such that the associated
K\"ahler metric $g_1$ and $g_2$ are quasi-isometric.
Since the set of $\CT$-tame almost complex structure $\cJ_\omega$ tame to $\omega$ is contractible,
this finishes the proof as usual.
\end{proof}

\part{Calculations}

In this part, we give some concrete description of our Fukaya category $Fuk(M,\CT)$
using the standard surface representation of the surface $(M,\CT)$.

\section{Generation of the Fukaya category}

In this section, we provide a description of generating set of $Fuk(M,\CT)$.

We first recall the following standard definition of Liouville isomorphisms in general.

\begin{defn}\label{defn:Liouville-isomorphism}
A Liouville isomorphism between Liouville domains $(M_1,\theta_1)$, $(M_2,\theta_2)$ is a diffeomorphism
$\phi:\hat{M}_1\rightarrow \hat{M}_2$ between their completion $(\hat{M}_1,\hat{\theta}_1)$ and
$(\hat{M}_2,\hat{\theta}_2)$ satisfying $\phi^*\hat{\theta}_2=\hat{\theta}_1+df$ for some compactly supported $f$.
Such map is symplectic and compatible with the Liouville flow at infinity.

Two Lagrangians $L_1$, $L_2$ are Liouville isotopic if there is a smooth family of Liouville isomorphism $\{\phi_i\}_{i\in[0,1]}$ such that $\phi_0=id_M$ and $\phi_1(L_1)=L_2$.
\end{defn}

The following is the well-known standard fact.

\begin{lemma}
If $L_1$ and $L_2$ are Liouville isotopic, then $L_1,\, L_2$
are quasi-isomorphic as an element of $Fuk(M)$.
\end{lemma}
\begin{proof} By definition, we have an isotopy $\phi^t$ of Liouville isomorphism with
$L_2 = \phi^1(L_1)$ and $\phi^0 = id$. For any test Lagrangian $(K_0,\cdots, K_\ell)$, the isotopy induces an
$A_\infty$ homotopy
$$
\mathfrak{n}_{\{\phi^t\}}: CF(L_1,K_1,\cdots, K_\ell) \to CF(L_2,K_0).
$$
By considering the time-reversal isotopy, we derive that
the pushforward $\phi_*: Fuk(M) \to Fuk(M)$ is a quasi-isomorphism
and so $L_1$ and $L_2$ are quasi-isomorphic.
\end{proof}

For our current purpose, we need to introduce
the notion of a Liouville isomorphism between standard Weinstein surfaces.
\begin{defn}\label{defn:isomorphism-standard-surface}
A Liouville isomorphism between standard Weinstein surfaces $(M_1, \psi_1)$, $(M_2,\psi_2)$
is a Liouville
diffeomorphism $\phi:M_1\to M_2$ that satisfies the following:
\begin{enumerate}
\item There exists $\ell_0 \in \N$ such that $\phi({M_1}^{\leq k}) \subset M_2^{\leq{k+\ell_0}}$  for all $k\in \ZZ_{>0}$.
\item There is some  $n_0 \in \ZZ_{>0}$ such that $\phi$ intertwines  gradient trajectories of $\psi_1$
and $\psi_2$ on $M_1 \setminus {M_1}^{\leq n}$.
\end{enumerate}
\end{defn}

Note that a Liouville isomorphism between finite type standard surfaces $M_1, \, M_2$ restricts to
 a Liouville isomorphism between Liouville subdomains thereof by definition.

\subsection{The ideal boundary of gradient sectorial Lagrangian brane}

In this subsection, we study and classify the behaviour of gradient-sectorial Lagrangians near end.
Let a gradient-sectoral Lagrangian submanifold $L$ of a surface $M$ be given.
If $L$ is closed (i.e., compact without boundary),  there exists a $n\in \ZZ_{>0}$ such that $L\subset M^{\leq n}$ since $M^{\leq 1}\subset M^{\leq 2}\subset \dots$ being a compact exhaustion of $M$.

For open Lagrangian submanifolds, the following holds. First of all
we recall that the elements of our generating set do not pass through
the saddle critical set of $\psi$.

\begin{prop}\label{prop:grad-sect_lag_ends}
Let $L$ be a connected gradient-sectorial Lagrangian submanifold of $M$ such that
$L \cap M^{\geq k} \cap \Crit \psi = \emptyset$ for all sufficiently large $k$.
If $L$ is open, then there is $n\in \ZZ_{>0}$ such that $L \cap  M^{\geq n}$ only has
two connected components and they are $\psi$-gradient trajectories of a point in
$L \cap \partial M^{\leq n}$.
\end{prop}

\begin{proof}
By the definition of gradient-sectorial Lagrangian and our assumption on $L$, there is $k_0\in \ZZ_{>0}$ such that $\grad \psi$ is tangent to $L \cap M^{\geq k_0}$ and $L \cap M^{\geq k_0} \cap \Crit \psi = \emptyset$.
Since $M^{\leq 1}\subset M^{\leq 2}\subset \dots$ is a compact exhaustion of $M$,
 $L$ intersects $M^{\leq k}$ for all sufficiently large $k$, say $k \geq k_0$.

Since $L$ is connected and $L$ intersects $M^{\leq k}$, every connected component of
$L \cap M^{\geq k}$ contains a point in $\partial M^{\leq k}$.
Since $L$ is properly embedded and $\partial M^{\leq k}$ is compact, $L\cap \partial M^{\leq k}$ consists of finitely many points.
Therefore, there are finitely many connected components of $L \cap M^{\geq k}$.

There are two possibilities for each connected component of $L \cap M^{\geq k}$.
If the component is noncompact, it must be a complete gradient trajectory issued at the given
intersection point in $L \cap \del M^{\leq k}$.
If the component is compact, since $L$ does not cross critical point of $\psi$ in $M^{\geq k}$, the component
must be homeomophic to a closed interval with its two boundary points lying on the set
$\{\psi = k\}$ and hence must be a constant trajectory which means the point must be a critical
point of $\psi$, a contradiction to the hypothesis that $L \cap M^{\geq k} \cap \Crit \psi = \emptyset$
for all $k \geq k_0$. This finishes the proof.
\end{proof}
Recall that any connected Lagrangian brane
is equipped with orientation and is homeomorphic to $\R$ as an oriented manifold.

The above proposition enables us to define the following.

\begin{defn}[Ideal boundary of open gradient-sectorial Lagrangians]
\label{defn:ideal-bdy} Let $L$ be a gradient-sectorial Lagrangian brane.
\begin{enumerate}
\item
We call each connected component of a $\psi$-gradient trajectory of $L \cap M^{\geq n}$
an \emph{end} of $L$ in $(M,\CT)$.
\item
We call the unique ideal boundary point $p =\{P_1\supset P_2\supset \dots \}\in B(M)$ associated to the
$\psi$-gradient trajectory appearing in Proposition \ref{prop:grad-sect_lag_ends} an \emph{ideal boundary
point}. We denote by $\del_\infty L$ the set of ideal boundary points
$$
\del_\infty L = \{p, q\}
$$
where $p,q$ are the ideal boundary points at $\pm \infty$ of $L \cong \R$ respectively.
\item We define the asymptotic evaluation maps
\be\label{eq:ev+-}
\ev_\infty^L: \{\pm\infty\} \to B(M)
\ee
by $\ev_\infty^L(\infty) := p$, $\ev_{\infty}^L(-\infty) := q$ given above respectively.
\end{enumerate}
\end{defn}
If $p = q$, we call the point $p$ the \emph{double ideal boundary point} of $L$.

\subsection{Generation and classification of objects of $Fuk(M,\CT)$}

In this subsection, we want to make a classification of the objects of Fukaya category $Fuk(M,\CT)$,
especially those of open Lagrangian branes, which respects the behaviour of the ends of open Lagrangian submanifolds.

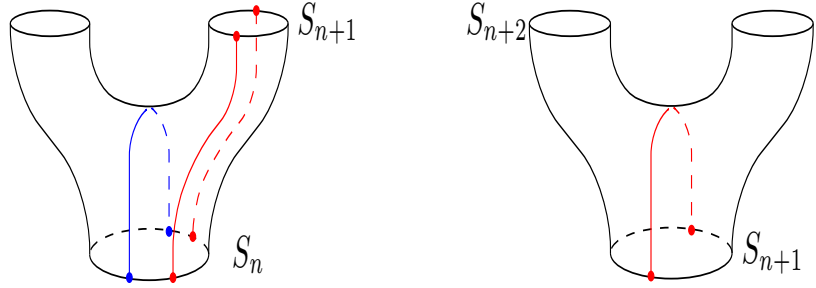
\begin{figure}
    \centering
    \subfloat{{
    \resizebox{5cm}{4cm}{
\begin{tikzpicture}
\draw[] (2.5,2.5) arc (0:360:0.5 and 0.1);
\draw[] (0,2.5) arc (0:360:0.5 and 0.1);

\draw[] (1.5,0.7) arc (0:-180:0.75 and 0.2);

\draw[dashed] (1.5,0.7) arc (0:180:0.75 and 0.2);

\draw[rounded corners=10pt] (1.5,2.5)-- ++(0,-0.5)--++(-0.75,-0.2)--++(-0.75,0.2)--++(0,0.5);
\draw[rounded corners=13pt] (1.5,0.7)-- ++(0,0.5)--++(1,0.8)--++(0,0.5);
\draw[rounded corners=13pt] (0,0.7)-- ++(0,0.5)--++(-1,0.8)--++(0,0.5);

\draw[blue,rounded corners=7pt] (0.5,0.52)-- ++(0,1.2)--++(0.2,0.1);
\draw[blue,dashed,rounded corners=7pt] (0.8,1.83)-- ++(0.2,-0.1)--++(0,-0.86);

\draw[red,rounded corners=7pt] (1.85,2.4)-- ++(0,-.5)--++(-0.5,-0.4)--++(-0.3,-0.5)--++(0,-0.5);
\draw[red,dashed,rounded corners=7pt] (2.1,2.6)-- ++(0,-0.7)--++(-0.5,-0.4)--++(-0.3,-0.4)--++(0,-0.2);

\draw[blue,fill=blue] (0.5,0.52) circle (1pt);
\draw[blue,fill=blue] (1,0.88) circle (1pt);
\draw[red,fill=red] (1.05,0.52) circle (1pt);
\draw[red,fill=red] (1.3,0.84) circle (1pt);

\draw[red,fill=red] (1.85,2.4) circle (1pt);
\draw[red,fill=red] (2.1,2.6) circle (1pt);

\node[] at (3,2.5) {$S_{n+1}$};
\node[] at (2,0.7) {$S_n$};

\end{tikzpicture}
    }
    }}
    \qquad
    \subfloat{{
    \resizebox{5cm}{4cm}{
\begin{tikzpicture}
\draw[] (2.5,2.5) arc (0:360:0.5 and 0.1);
\draw[] (0,2.5) arc (0:360:0.5 and 0.1);

\draw[] (1.5,0.7) arc (0:-180:0.75 and 0.2);
\draw[dashed] (1.5,0.7) arc (0:180:0.75 and 0.2);

\draw[rounded corners=10pt] (1.5,2.5)-- ++(0,-0.5)--++(-0.75,-0.2)--++(-0.75,0.2)--++(0,0.5);

\draw[rounded corners=13pt] (1.5,0.7)-- ++(0,0.5)--++(1,0.8)--++(0,0.5);
\draw[rounded corners=13pt] (0,0.7)-- ++(0,0.5)--++(-1,0.8)--++(0,0.5);

\draw[red,rounded corners=7pt] (0.5,0.52)-- ++(0,1.2)--++(0.2,0.1);
\draw[red,dashed,rounded corners=7pt] (0.8,1.83)-- ++(0.2,-0.1)--++(0,-0.86);

\draw[red,fill=red] (0.5,0.52) circle (1pt);
\draw[red,fill=red] (1,0.88) circle (1pt);

\node[] at (2,0.7) {$S_{n+1}$};
\node[] at (-1.4,2.5) {$S_{n+2}$};

\end{tikzpicture}
    }
    }}
    \qquad

    \caption{Pairs of pants contained in $C$}
    \label{fig:pairsofpants_in_C}
\end{figure}

Now by applying a Liouville isomorphisms of standard surfaces given in Definition \ref{defn:isomorphism-standard-surface}, the following 
we immediately derive the following two lemmata.

\begin{lemma}\label{lem:partition}
Let $M$ be a standard surface with a blueprint pair $(\cS,\chi_\cS)$
given in Definition \ref{defn:blue-print}.
We can partition $\cS$ into the subsets consisting of maximal chains
$I_1\supset I_2\supset \dots$ of $\cS$ so that the following hold:
\begin{enumerate}
\item  For any maximal chain with minimal element
$$
I_1\supset I_2\supset \dots \supset I_{i_0}
$$
we have
$$
\xi_\cS(I_{i_0})= 1, \quad  \xi_\cS(I_i)=0 \, \forall i \neq i_0.
$$
\item For any maximal chain without minimal element
$$
I_1\supset I_2\supset \dots ,
$$
we have $\xi_\cS(I_i)=0$ for all $i$.
\end{enumerate}
\end{lemma}

\begin{lemma}\label{lem:control_of_asc,dec_mflds}
Let $(\CS, \chi_{\CS})$ be the blueprint pair of $M$ and
consider the partition obtained in Lemma \ref{lem:partition}.
Let $D$ be the base disk of $M$. We can partition $M\setminus D$
into the union of building blocks so that the union
is partitioned into sub-unions of those corresponding to each
maximal chain.  We can enumerate the
subset of saddle critical points of $\psi$ appearing
in the sub-union associated to each maximal chain into
$$
\{q_1, q_2, \cdots, q_i, \cdots\}
$$
and modify the given plurisubharmonic function $\psi$ so that
the following hold:
\begin{enumerate}
\item For all $i = 1, \cdots, k, \cdots$,
\be\label{eq:good}
\psi(q_i)<\psi(q_{i+1}).
\ee
In particular, we have $\psi^{-1}(c_i) = \{q_i\}$ for $c_i = \psi(q_i)$ for all $i \geq 1$.
\item All saddle connections connect $q_i$ and $q_{i+1}$ for all $i \geq 1$.
\end{enumerate}
We call the sub-union associated to each maximal chain a \emph{cell of partition}.
\end{lemma}

We can also classify the proper $\psi$ gradient trajectories, up to Liouville isomorphism,
with the same ideal boundary.

\begin{lemma}\label{lem:eq_class_of_end_of_lagrangians}
For any end $p\in B(M)$, consider $\psi$-gradient trajectories ending at $p$.
Then there are at most two equivalence classes of such $\psi$-gradient trajectories
up to Liouville isotopy. More specifically, the following hold:
\begin{enumerate}
\item If $p \not\in B'(M)$ and $p$ is an isolated end, any two connected components of $L$ with
the same end $p$ is Liouville isotopic.
\item If $p\in B'(M)$ is an isolated end, there are exactly two such equivalence classes.
\item If $p$ is a limit point in $B(M)$ with respect to the subset topology of $\RR$,
there are at most two such equivalence classes.
\end{enumerate}
\end{lemma}
\begin{proof}
Let $\ell_1,\,\ell_2$ be the $\psi$ gradient trajectories corresponding to the ends of a sectorial Lagrangian,
and let $\phi^t$ be a Liouville isotopy
such that $\phi^0=id_M$ and $\phi^1(\ell_1)=\ell_2$.
Then since the isotopy induces the identity map on the ideal boundary $B(M)$, there exist $n\in \ZZ_{>0}$ such that $\ell_1,\,\ell_2$ should intersect the same boundary component of $M^{\leq k}$ for all $k\geq n$.
Pick $k$ and a boundary component, say $S_k$, of $M^{\leq k}$ that intersect $\ell_1$ and $\ell_2$.
Let $x_1,\,x_2$ be the intersection points $\ell_1\cap S_k,\, \ell_2\cap S_k$ respectively.
Then the intersection point between $\phi^t(\ell_1)$ and $S_k$ gives us a path from $x_1$ to $x_2$
in $S_k$.
Note that this path cannot intersect a descending manifold of a saddle point of $\psi$.

Conversely, if we can find a path in $S_k$ not intersecting
any descending manifold of a critical point of $\psi$, the path induces
a corresponding Liouville isotopy.
Therefore, we have a one-to-one correspondence between the set of connected components of
$$
S_k\setminus \{\text{descending manifold of saddle points of }\psi\}
$$
and the set of equivalence classes of gradient trajectories intersecting $S_k$.

From now on, we will count such connected components in each cases.

\medskip

\noindent{\bf Case 1:} \emph{$p$ is an isolated point of $B(M)$ and $p\not\in B'(M)$}.

By the standing hypothesis, for a sufficiently large $n \in \ZZ_{>0}$,
there is a connected component $C$ of $M\setminus M^n$ such that $C^*$ contains $p$ with respect to the topology of $B(M)$ defined in
Definition \ref{defn:ideal_bdy} and $C^*$ contains only $p$.
We also recall therefrom that $C^*$ is the set of all ends
$$
C^* = \big\{\{P_1\supset P_2 \supset \dots\} = p \in B(M) \mid  P_n \subset C\,
\text{\rm for all sufficiently large $n$} \big\}.
$$
Since $p\not\in B'(M)$ and $C^*$ contains only one end,  $C$ must be of finite type
and contain finitely many genus from the compact part and has one cylindrical end.
By taking larger $n$ if needed, we may assume every such genus is contained in $M^n$.
Therefore, $C$ does not contain a critical point and any two gradient trajectories starting from a point in $S_n$ are Liouville isotopic by an isotopy extending
the rotation near $S_n$ as before. This shows that there is the only equivalence class
in this case.

\medskip
\noindent{\bf Case 2:} \emph{ $p$ is an isolated point of $B(M)$ and $p\in B'(M)$}.

By assumption, there is a sufficiently large $n \in \ZZ_{>0}$ such that
$M\setminus M^n$ carries a connected component, say $C$, the
associated $C^*$ of which contains $p$ with respect to the topology of $B(M)$. By the
isolatedness assumption, $C^*$ again contains only $p$.
Since $p\in B'(M)$, $C^*$ contains a nonplanar end and $C$ contains infinitely many genus.
By our construction of $\psi$ as we referred in the Lemma \ref{lem:control_of_asc,dec_mflds},
$C$ is contained in a cell of partition and the descending manifold of the saddle critical
point of $\psi$ with minimum value in $C$
intersects $\partial C$ at two distinct points and the descending manifolds of
other critical points of $\psi$ in $C$ do not intersect $\partial C$.
Therefore, the set of points in $\partial C$ not contained in the descending manifold
of the minimum saddle critical point of $C$ consists of two connected open arcs.
Since the gradient trajectories issued from two points in the same open arc are Liouville
isotopic as before, we have derived that there are two Liouville isotopy classes in this case.

\medskip

\noindent{\bf Case 3:} \emph{$p$ is a limit point of $B(M)$}.

Since $p$ is a limit point of $B(M)$, any open neighbourhood of $p$ in $B(M)$ should contain other point in $B(M)$.
Therefore, for any $n \in \ZZ_{>0}$, the connected component $C$ of $M\setminus M^n$ such that
$C^*$ contains $p$ contains infinitely many pair of pants.
By taking a sufficiently large $n$,
we may assume $C$ does not contain a genus from the compact part.
Also, by our construction of $\psi$ as we referred in the Lemma \ref{lem:control_of_asc,dec_mflds}, it suffices to assume that $C$ does not contain a genus.
This is because $p$ is a limit point in $B(M)$ and attaching a cylinder with a genus cannot be repeated infinitely many times. Removing such finitely many cylinders with genus does not change
the equivalence classes of gradient trajectories.
See Figure \ref{fig:pairsofpants_in_C}.
Descending manifold of a critical point contained in $M^{\leq n+1}\setminus M^{\leq n}$ intersects $S_n$ at two points and divide $S_n$ into two connected arcs.
Gradient trajectory cannot intersect descending manifold of critical point and gradient trajectories starting from points in one of arcs cannot intersect $S_{n+1}$
Also, descending manifold of a critical point contained in $M^{\leq n+2}\setminus M^{\leq n+1}$  intersects $S_n$ at two points and divide $S_n$ into two connected arcs.
Again, gradient trajectory cannot intersect descending manifold of critical point and gradient trajectories starting from points in one of arcs cannot intersect $S_{n+2}$.
We need to take intersection of two open arcs to find gradient trajectories which intersects all
of $S_n,\,S_{n+1},\,S_{n+2}$.
We need to do this for all $n+3,\,n+4,\,\cdots$.
Since descending manifolds cannot intersect each other except for the base point, intersection of such connected arcs on $S_n$ has at most two connected components.
\end{proof}

\begin{cor}
For any $p \in B(M)$, there are at most two equivalence classes up to a Liouville isotopy for the ends of an open gradient-sectorial Lagrangian whose ideal boundary is $p \in B(M)$.
We denote them by $p_{L,+}$ and $p_{L,-}$ depending on their orientations.
\end{cor}

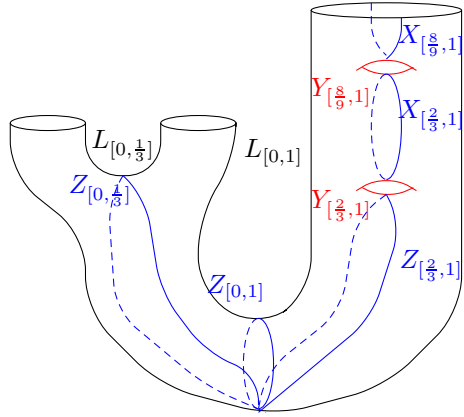
\begin{figure}

\begin{tikzpicture}

\draw[] (1,0.9) arc (0:360:0.5 and 0.1);
\draw[] (6,2.4) arc (0:360:1 and 0.1);
\draw[] (3,0.9) arc (0:360:0.5 and 0.1);

\draw[rounded corners=13pt] (2,0.9)-- ++(0,-0.7)--++(-1,0)--++(0,0.7);

\draw[rounded corners=13pt] (3,0.85)-- ++(0,-0.4)--++(-0.5,-0.8)--++(0,-1)--++(0.8,-0.5)--++(0.7,0.5)--++(0,3.75);

\draw[rounded corners=13pt] (0,0.9)-- ++(0,-0.7)--++(1,-0.8)--++(0,-1)--++(1,-1)--++(1.5,-0.4)--++(1.5,0.4)--++(1,1)--++(0,4);

\draw[blue,densely dashed] (3.3,-1.7) arc (90:270:0.2 and 0.6);
\draw[blue] (3.3,-2.9) arc (270:450:0.2 and 0.6);

\begin{scope}[scale=0.8]
\draw[red,rounded corners=10pt] (5.75,0)-- +(0.5,0.3)-- +(1,0);
\draw[red,rounded corners=8pt] (5.85,0.05)-- +(0.4,-.2)-- +(0.85,0);
\end{scope}

\begin{scope}[scale=0.8]
\draw[red,rounded corners=10pt] (5.75,2)-- +(0.5,0.3)-- +(1,0);
\draw[red,rounded corners=8pt] (5.85,2.05)-- +(0.4,-.2)-- +(0.85,0);
\end{scope}

\draw[blue,rounded corners=10pt] (5,-0.05)-- ++(0.2,-.4)-- ++(-0.4,-1.25)--++(-1.5,-1.25);

\draw[blue,rounded corners=10pt,densely dashed] (5,-0.05)-- ++(-0.4,-.3)-- ++(-0.2,-1.25)--++(-1,-0.8)--++(-0.1,-0.5);

\draw[blue,rounded corners=6pt](5.2,2.3)--++(0,-0.35)--++(-0.2,-0.2);
\draw[blue,rounded corners=6pt,densely dashed](4.8,2.5)--++(0,-0.5)--++(0.2,-0.2);

\draw[blue,rounded corners=10pt] (1.5,0.2)-- ++(0.3,-.3)-- ++(0.2,-1.25)--++(0.6,-0.8)--++(0.7,-0.3)--++(0,-0.4);

\draw[blue,rounded corners=10pt,densely dashed] (1.5,0.2)-- ++(-0.2,-.3)-- ++(0.1,-1.5)--++(1.2,-1)--++(0.8,-0.35);

\draw[blue,densely dashed] (5,1.55) arc (90:270:0.2 and 0.7);
\draw[blue] (5,0.15) arc (270:450:0.2 and 0.7);

\node at (1.5,0.6) {$L_{[0,\frac{1}{3}]}$};
\node at (3.5,0.5) {$L_{[0,1]}$};

\node[blue] at (5.6,-1) {$Z_{[\frac{2}{3},1]}$};
\node[blue] at (5.6,1) {$X_{[\frac{2}{3},1]}$};
\node[red] at (4.4,-0.2) {$Y_{[\frac{2}{3},1]}$};
\node[red] at (4.4,1.3) {$Y_{[\frac{8}{9},1]}$};
\node[blue] at (3,-1.3) {$Z_{[0,1]}$};
\node[blue] at (1.2,0) {$Z_{[0,\frac{1}{3}]}$};

\node[blue] at (5.6,2) {$X_{[\frac{8}{9},1]}$};

\end{tikzpicture}

\caption{Example of labelling generators}
\label{fig:labelling_generators}
\end{figure}

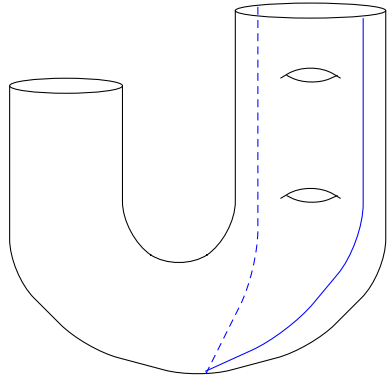
\begin{figure}

\centering

\begin{tikzpicture}

\draw[] (2.5,1.5) arc (0:360:0.75 and 0.1);

\draw[] (6,2.5) arc (0:360:1 and 0.1);

\draw[rounded corners=13pt] (2.5,1.5)--++(0,-2)--++(0.75,-0.5)--++(0.75,0.5)--++(0,3);

\draw[rounded corners=13pt] (1,1.5)--++(0,-2.5)--++(1,-1)--++(1.5,-0.4)--++(1.5,0.4)--++(1,1)--++(0,3.5);

\draw[blue,rounded corners=15pt] (5.7,2.4)-- ++(0,-3)-- ++(-1,-1.2)--++(-1.1,-0.5);

\draw[blue,rounded corners=15pt,densely dashed] (4.3,2.55)-- ++(0,-3.5)-- ++(-0.7,-1.4);

\begin{scope}[scale=0.8]
\draw[rounded corners=10pt] (5.75,0)-- +(0.5,0.3)-- +(1,0);
\draw[rounded corners=8pt] (5.85,0.05)-- +(0.4,-.2)-- +(0.85,0);
\end{scope}

\begin{scope}[scale=0.8]
\draw[rounded corners=10pt] (5.75,2)-- +(0.5,0.3)-- +(1,0);
\draw[rounded corners=8pt] (5.85,2.05)-- +(0.4,-.2)-- +(0.85,0);
\end{scope}

\end{tikzpicture}

\caption{$L_{1^1,1^2}$ on nonplanar end $1$}

\end{figure}

Now we list two collections of Lagrangians on $M$ which together we anticipate will generate
our Fukaya category. The first class consists of ascending and descending manifolds of saddle critical points and the other consists of open Lagrangians which is a gluing of two gradient trajectories issued from the base point. Both classes are determined by end structure and we will show that they split-generate the Fukaya category of $M$ that we introduce in the present paper.

\begin{defn}[Saddle connection Lagrangians]\label{defn:finite_generators}
Let $M$ be a standard surface.
For each $I\in S$, we consider a collection of Lagrangians labelled as follows:
\begin{enumerate}
\item $\{L_I\}$: If $\xi_\cS(I)=1$, $L_I$ represents a Lagrangian which is the ascending manifold of the saddle point of the pair of pants attached on $I$.
\item $\{X_I,\,Y_I\}$: If $\chi_\cS(I)=2$, $X_I$ represents the ascending manifold of one of two critical
points of the genus with bigger $\psi$ value and $Y_I$ represents the descending manifold.
\item $\{Z_I\}$: If $B_I$ contains a critical value whose descending manifold contains the base point,
$Z_I$ represents that descending manifold.
\end{enumerate}
\end{defn}

It is already known that if $M$ is of finite type, Lagrangians labelled in Definition
\ref{defn:finite_generators} split-generate $Fuk(M)$.

\begin{notation}[$\langle S \rangle$]\label{notation:<S>}
Given a subset $S$ of an objects of a category, we denote $\langle S \rangle$ as
the smallest subcategory split-generated by $S$.
\end{notation}

For the split-generation of Fukaya category of infinite type surfaces, we define
open Lagrangians with desired end behaviours.
By our construction of $\psi$, if a point in $M$ is not contained in a descending or
an ascending manifold of a saddle critical point, it should be contained in an ascending
manifold of the base point. This leads us to the following definition.

\begin{defn}[Open Lagrangian generators]\label{defn:open_lag_btwn_ends}
For every pair $p,\, q \in b(M)$, let $L_{p,q}$ be a proper open (oriented) Lagrangian
such that $\ev^+_L(+\infty) = q$ and $\ev^-_L(-\infty) = p$ and it
does not pass any saddle critical points.
We may assume that it passes through the base point $p_{\text{\rm rt}}$.
We assign the orientation and other brane data which
make it an open gradient-sectorial Lagrangian brane. $L_{p,q}$ constructed in this way
is unique up to a Liouville isotopy.
\end{defn}

\subsection{The case of ideal boundaries with countably many limit points}

In this subsection, we consider Riemann surface $(M,\CT)$ whose ideal boundary $B(M)$
has countably many limit points $\{p_i\}_{i\in \ZZ_{>0}}$. We note that $[0,1]\setminus \{p_i\}_{i\in \ZZ_{>0}}$ has at most countably many connected components.

\begin{defn}\label{defn:generators}
Let $M$ be a standard surface and suppose $B(M)$ has
at most countably many limit points $\{p_i\}_{i\in \ZZ_{>0}}$.
We list the following open Lagrangians denoted by as follows:
\begin{enumerate}

\item $L_{p,p}$: For $p \in B(M)$ if there are two equivalence classes of the end whose ideal boundary is $p$.
\item
Take an intersection of connected components of $[0,1]\setminus \{p_i\}_{i\in \ZZ_{>0}}$ with $B(M)$ and label the nonempty sets as $C_j$, $j\in \ZZ_{>0}$.
Pick $q_j\in C_j$ for every $j$.
After that,
\begin{itemize}
\item $L_{{q_1},{q_j}}$ for all $j>1$
\item $L_{{q_1},{p_i}}$ for all $i$
\end{itemize}
\end{enumerate}
\end{defn}

For three distinct elements $p ,\, q ,\, r \in  B(M)$,
$L_{p ,r }\in \langle L_{p ,q },L_{q ,q }, L_{q ,r } \rangle$.
Also, if ideal boundary of $L_I$ is $\{p ,q\} $ for $I\in \cS$, we have the unique maximal chain
$I=I_1\subsetneq I_2 \subsetneq \cdots \subsetneq I_n$ so that

\begin{equation}\label{eqn:generation_L_p,q}
L_{p ,q }\in \langle L_{I_1},\,X_{I_1},\,Y_{I_1},\,X_{I_2},\,Y_{I_2},\,\dots,\,X_{I_n},\,Y_{I_n} ,\,Z_{I_n}\rangle,
\end{equation}

\begin{equation}\label{eqn:generation_L_I}
L_{I_1}\in \langle L_{p ,q },\,X_{I_1},\,Y_{I_1},\,X_{I_2},\,Y_{I_2},\,\dots,\,X_{I_n},\,Y_{I_n} ,\,Z_{I_n}\rangle.
\end{equation}

We may regard it as a quiver, where vertices are elements of $b(M)$ and
map from $p $ to $q $ is $L_{p ,q }$.
Note that for any pair with $p\neq q \in C_j$ of Definition \ref{defn:generators}, we can find a finite
sequence $I_1,\dots,I_n$ of elements of $\cS$ so that ideal boundary points of $L_{I_t}$ are
$p_t,p_{t+1}$ and $p=p_1,\,q=p_{n+1}$.
Therefore, with reverse orientation and split-generation, we can get any map between two vertices
in this quiver from Lagrangians in Definition \ref{defn:finite_generators} and \ref{defn:generators}.

We are ready to state our theorem on split generation.

\begin{theorem}\label{thm:main}
Let $M$ be a standard surface.
If its end structure has at most countably many limit points,
$Fuk(M)$ is split-generated by the set of Lagrangians described in Definition
\ref{defn:finite_generators} and \ref{defn:generators}.
\end{theorem}

\begin{proof}
We call the set of generators in Definition \ref{defn:finite_generators} and \ref{defn:generators}
as $G$ and want to show that $Fuk(M,\CT)= \langle G \rangle$. For this purpose, it is enough
to show that any object of $Fuk(M,\CT)$ is generated by those in $G$. We denote $Fuk(M,\CT) = Fuk(M)$
below for the simplicity of notation since we do not change $\CT$ here.

First, recall that every noncompact element $L\in Ob(Fuk(M))$ is invariant under
gradient flow of $\psi$ near infinity of $M$, i.e., on $\partial M^{\leq n}$ for every
sufficiently large $n\in \ZZ_{>0}$. Recalling that any $n\in \ZZ_{>0}$, $M^{\leq n}$ is a Liouville subdomain of $M$, we have a
Viterbo restriction functor
$$
\rho_n:Fuk(M)\rightarrow Fuk(\widehat{M^{\leq n}})
$$
where $\rightarrow Fuk(\widehat{M^{\leq n}})$ is the Liouville completion of $M^{\leq n}$
obtained by attaching a cylindrical end to each boundary component thereof.
We alert readers that \emph{we do not have a left inverse of this functor unlike the case of
finite type Liouville manifolds, because $\rho_n$ forgets information of
Lagrangian submanifold on $M\setminus M^{\leq n}$.}

Our proof will be done in the following steps:
\begin{enumerate}
\item Construct a functor $\iota^{\vec{r}}_n$ which satisfies $\iota^{\vec{r}}_n \circ \rho_n(L)=L$ for some objects $L\in Fuk(M)$.
\item Split-generate $L\in Fuk(M)$ which satisfies $\iota^{\vec{r}}_n \circ \rho_n(L)=L$ using generators and images of $\iota^{\vec{r}}_n$.
\item Split-generate images of $\iota^{\vec{r}}_n$ using Lagrangians in (1) and generators.
\item Split-generate any object of $Fuk(M)$.
\end{enumerate}

\medskip

\noindent{\bf Step (1):} We construct a functor $\iota^{\vec{r}}_n$ by manually designating
behaviour of Lagrangians in each connected component of $M\setminus M^{\leq n}$.
Since $\partial M^{\leq n}$ has finitely many connected components, we name them as
$\sigma_1,\,\ldots,\,\sigma_m$.
For every $1\leq i\leq m$, pick a point $x_i\in \sigma_i$ which is not contained in
a descending manifold of a critical point of $\psi$.
Then gradient flow starting from $x_i$ is in some equivalence class $r_i\in b(M)$ and
we write $\vec r: = \{r_1,\dots,r_m\}$.
If $L$ is closed, $L\subset \text{\rm int}(M^{\leq n})$ and we send $L$ into $M$ under the inclusion map.
If $L$ is open, compare connected components of $M\setminus  M^{\leq n}$ and
$\widehat{M^{\leq n}}\setminus M^{\leq n}$.
If a connected component of $M\setminus  M^{\leq n}$ attached to $\sigma_i$ is an infinite cylinder, it is symplectomorphic to corresponding component of $\widehat{M^{\leq n}}\setminus M^{\leq n}$
and we set $\iota^{\vec{r}}_n$ as an inclusion map on that component.
Otherwise, we may find a representative of $L$ in the Liouville isomorphism class whose
intersection with $\sigma_i$ is $x_i$ for all $i$ and invariant under gradient flow of $\psi$
near $\sigma_i$.
This can be done by twisting $L$ near each $\sigma_i$.
After that, we send this into $M$ under the inclusion map and extend from $x_i$ along the gradient
flow of $\psi$.
This is an $A_\infty$-functor from $Fuk(\widehat{M^{\leq n}})$ to $Fuk(M)$.

\medskip
\noindent{\bf Step (2):} Pick $L\in Ob(Fuk(M))$. We may assume $L$ is connected.
We may also regard $X_I,Y_I,Z_I$ for ${I\in \bigcup_{k=1}^{n-1} \CI^{(k)}}$ as objects of $Fuk(\widehat{M^{\leq n}})$ and we still denote their images $\iota^{\vec{r}}_n(X_I),\,\iota^{\vec{r}}_n(Y_I),\,\iota^{\vec{r}}_n(Z_I)$
by $X_I,Y_I,Z_I$ and vice versa.

If $L$ is closed, there exists $n\in \ZZ_{>0}$ so that $L\subset M^{\leq n}$ and
$L=\iota^{\vec{r}}_n \circ \rho_n(L)$ and
$\rho_n(L)\in Ob(Fuk(\widehat{M^{\leq n}}))$ and $\rho_n(L)\in \langle X_I,\,Y_I,\,Z_I,\,\widehat{L_I\cap M^{\leq n} }\rangle_{I\in \bigcup_{k=1}^{n-1} \CI^{(k)}}$.
Therefore,
$$
L=\iota^{\vec{r}}_n \circ \rho_n(L)\in \left\langle \iota^{\vec{r}}_n(X_I),\,\iota^{\vec{r}}_n(Y_I),\,\iota^{\vec{r}}_n(Z_I),\,
\iota^{\vec{r}}_n(\widehat{L_I\cap M^{\leq n} }) \right\rangle_{I\in \bigcup_{k=1}^{n-1} \CI^{(k)}}.
$$
Since $X_I,\,Y_I,\,Z_I$ for $I\in \bigcup_{k=1}^n \CI^k$ are also closed Lagrangians contained in
$M^{\leq n}$, we indeed have
$$
L\in \left\langle X_I,\,Y_I,\,Z_I,\,\iota^{\vec{r}}_n(\widehat{L_I\cap M^{\leq n} })
\right\rangle_{I\in \bigcup_{k=1}^{n-1} \CI^{(k)}}.
$$

Now suppose $L$ is open. Since $L$ is connected, $L$ has two ends, say $p,\,q$.
If $p,\,q$ are isolated points in $B(M)$ and cylindrical,
we can find a sufficiently large $n\in \ZZ_{>0}$ such that $\iota^{\vec{r}}_n$ is an inclusion map on those
cylindrical ends and $\iota^{\vec{r}}_n \circ \rho_n(L)=L$ holds.
Again by construction, we have
$\rho_n(L) \in Ob(Fuk(\widehat{M^{\leq n}}))$ and
$\rho_n(L)\in \langle X_I,\,Y_I,\,Z_I,\,\widehat{L_I\cap M^{\leq n} } \rangle_{I\in \bigcup_{k=1}^{n-1} \CI^{(k)}}$.
Then
$$
L=\iota^{\vec{r}}_n \circ \rho_n(L)\in \left \langle X_I,\,Y_I,\,Z_I,\,\iota^{\vec{r}}_n(\widehat{L_I\cap M^{\leq n} }) \right\rangle_{I\in \bigcup_{k=1}^{n-1} \CI^{(k)}}.
$$

\medskip

\noindent{\bf Step (3):} For ${I\in \bigcup_{k=1}^{n-1} \CI^{(k)}}$, $\iota^{\vec{r}}_n(\widehat{L_I\cap M^{\leq n} })=\iota^{\vec{r}}_n\circ \rho_n(L_I)$.
Also, by (\ref{eqn:generation_L_I}), $L_I$ is split-generated by a Lagrangian in Definition  \ref{defn:open_lag_btwn_ends} and closed Lagrangians in Definition \ref{defn:finite_generators}.

For $L_{p,q}$, there exists $s,t\in \{1,\dots,m\}$ determined by $p,\, q$ such that
$\iota^{\vec{r}}_n\circ \rho_n(L_{p,q})=L_{r_s,r_t}$.
Therefore, $\iota^{\vec{r}}_n\circ \rho_n(L_I)$ can be split-generated by Lagrangians in $G$.

\medskip

\noindent{\bf Step (4):} Suppose $L\in Ob(Fuk(M))$ is open  with $\del_\infty L = \{p,q\}$,
where exactly one of the elements, say $q$, is cylindrical.
If $p$ is an isolated point in $B(M)$ and $p\in B'(M)$,
there are two equivalence classes of ends whose ideal boundary is $p$.
For sufficiently large $n$, exactly one of those equivalence classes is in $\vec{r}$.
If an end of $L$ is in  $\vec{r}$, $\iota^{\vec{r}}_n \circ \rho_n(L)=L$ holds.
Otherwise, we have a Lagrangian $L'\in \langle L,\,L_{p,p}\rangle$ so that
 $L\in \langle L',\,L_{p,p}\rangle$ and an end of $L'$ is in $\vec{r}$.
Therefore, $\iota^{\vec{r}}_n \circ \rho_n(L')=L'$ holds and $L\in \langle L',\,L_{p,p}\rangle$

If $p$ is a limit point in $B(M)$, we have
$L'\in \langle L,\,L_{{q_1},p}\rangle$ so that ends of $L'$ are isolated.

Repeat this process to cover the case when both ends are not cylindrical.
Therefore, $L\in \langle G \rangle$ for any $L\in Ob(Fuk(M))$ and $Fuk(M,\CT)= \langle G \rangle$ holds.
\end{proof}

\section{Morphisms of Fukaya category of a surface}\label{sec:calculation}

\subsection{Morphisms in noncylindrical ends}
As we mentioned in the introduction, consider Hamiltonians $H$ of the type $H = \kappa \circ \psi$ for a one-variable function function $\kappa: \R \to \R$.

Recall the pair-of-pants decomposition of $M$ induced by $\psi$.
From now on, we rescale $\psi$ so that critical values are positive integers and
boundaries of finite cylinders are level set of positive integer value of $\psi$.
(We still use the notation $M^{\leq n}$, to denote the same Liouville subdomain as before.)

we can set joining parts to be disjoint and their complement consists of mutually disjoint smaller cylinders.
On these smaller cylinders, $\psi$ is set to be a height function, which is a linear function
with respect to $s$ and the flow of $\psi$ has the form
$$
\phi^1_H((s,t))=(s,t+\frac{\partial H(s,t)}{\partial s})
$$
holds. On the other hand, the Hamiltonian trajectory of a point cannot pass a critical point of $\psi$.
Let $L_1,\,L_2$ be any two gradient trajectories which pass a finite cylinder not contained in an ascending manifold or the descending manifold of the adjacent joining part.
Each gradient trajectories would have the constant $t$ value on the complement of gluing part, say
$t_1$ and $t_2$ respectively.
Then we can set $\kappa$ so that $\phi^1_H(L_1),\,L_2$ transversally intersect at two points in that cylinder:
\begin{enumerate}
\item In the complement of the joining parts, $\psi$ is linear with respect to $s$ and
$\frac {\partial \psi} {\partial t}=0$.
We also set $\kappa(x)=kx+c$ for some $k,\,c\in \RR_{>0}$. In this case, $\frac{\partial H(s,t)}{\partial s}$ is constant.
If $L_1\neq L_2$, set $k=1/\frac{\partial \psi}{\partial s}$ and otherwise set
$k=1/\frac{\partial \psi}{\partial s}+\epsilon$ for sufficiently small to
make Hamiltonian flow $\phi^1_H(L_1)$ pass $L_2$ once.
\item We use representative of $L_1$ which is sufficiently close to critical point of $\psi$ in each critical values of $\psi$ and may assume that Hamiltonian flow $\phi^1_H(L_1)$ in the boundary of finite cylinder, which is a level set of
$\psi$ of a critical value, does not pass $L_2$.
\item By our settings above, $\phi^1_H(L_1)$ and $L_2$ has an intersection point in each collar
neighbourhood of boundary of finite cylinder contained in joining parts.
\end{enumerate}

This leads to the following proposition, since each building block except for an infinite
cylinder contains a joining part and transversal intersection between $\phi^1_H(L_1)$ and $L_2$
in an infinite cylinder is trivial.

\begin{prop}\label{prop:intersection_in_finite_cylinder}
We can find a $\kappa: \R \to \R$ so that a Hamiltonian $H=\kappa \circ \psi$
satisfies the following:
Let $L_1,\,L_2$ be gradient trajectories in $B_I$ which does not intersect an ascending or descending manifold of critical point of $\psi$ in $B_I$.
Then $\phi^t_H(L_1)$ and $L_2$ has a transversal intersection point in $B_I$ for any $t\geq 1$.
\end{prop}

This is also useful when we deal with higher products, since image of $\mu^k$ is determined by
holomorphic disk which is a solution of a Cauchy Riemann boundary equation. Recall that $\rho_n:Fuk(M)\to Fuk(\widehat{M^{\leq n}})$ is the Viterbo restriction functor.

\begin{lemma}
Let $L_0,\,\ldots,\,L_k$ be gradient sectorial Lagrangians.
Then there exists $n\in \ZZ_{>0}$ so that
if $x_i\in \phi^i_H(L_{i-1})\cap \phi^{i-1}_H(L_{i})$ is contained in $M^{\leq m}$ for all $1\leq i \leq k$
for some $m\geq n$ with $\mu^k(x_1,\ldots,x_{k}) \in  M^{\leq m+1}$
$$
\mu^k(x_1,\ldots,x_{k})
=  \mu^k_{Fuk(\widehat{M^{\leq m+1}})}\left(\rho_{m+1}(x_1),\ldots,\rho_{m+1}(x_{k})\right).
$$
\end{lemma}

\begin{proof}
By Proposition \ref{prop:grad-sect_lag_ends}, we may take sufficiently large $n$ and assume that
$L_i\setminus M^{\leq n}$ consists of two gradient trajectories which does not intersect
any ascending or descending manifold of a critical point in $M\setminus  M^{\leq n}$ for every $i$.
Here, we want to find a holomorphic disk whose boundaries are $L_0,\phi^1_H(L_1),\ldots,\phi^{k}_H(L_k)$ and corners are $x_1,\ldots,x_k$ and a point $\mu^k(x_1,\ldots,x_{k})$ in $L_0 \cap \phi^{k}_H(L_k)$.
We denote $\mu^k(x_1,\ldots,x_{k})$ as $x_0$.
Suppose $x_0$ is not contained in $M^{\leq m}$.
Then $x_0$ is contained in a connected component $C$ of $M\setminus M^{\leq m}$.
Since $L_0,\,\phi^{k}_H(L_k)$ are connected and intersect $C$, they should intersect $\partial C$.
Then there is $I\in \cI^{(m+1)}$ so that $B_I$ is attached to $\partial C$.
Hamiltonian flow of a point in each building block $B_I$ cannot escape $B_I$ and
$L_k$ also should intersect $B_I$.
Therefore, by Proposition \ref{prop:intersection_in_finite_cylinder}, $L_0,\,\phi^{k}_H(L_k)$ should have an intersection point in $B_I$.
Since $u$ should be holomorphic on every interior point and $\psi$ value of $L_1$ and $\phi^{k-1}_H(L_k)$ are monotone increasing on $M\setminus M^{\leq m}$, the first intersection point in $B_I$ should be $x_0$ and $u$ should be contained in the interior of $M^{\leq {m+1}}$ by maximal principle.
Therefore, any holomorphic disk is contained in the interior of $M^{\leq {m+1}}$ and we can get the same result under the Viterbo restriction functor $\rho_{m+1}$.
\end{proof}

By this Lemma, we can regard the algebraic structure of Fukaya category of infinite type surface
is locally equivalent to that of its Liouville subdomain.

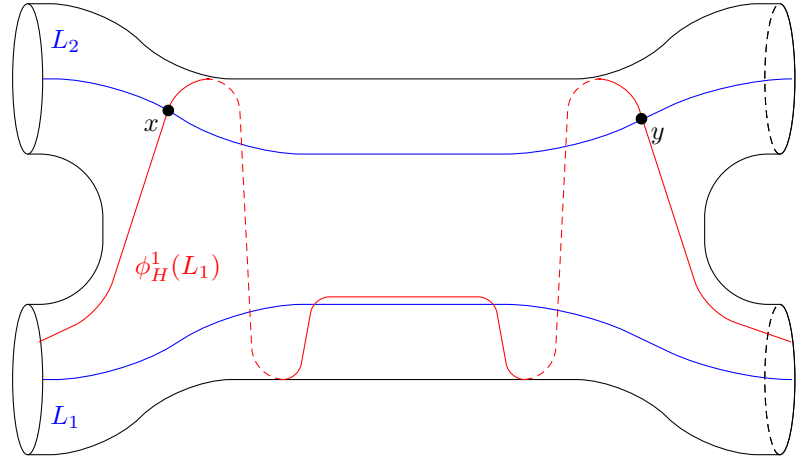
\begin{figure} \centering

\begin{tikzpicture}

\draw[] (-5,1) arc (-90:270:.2 and 1);
\draw[] (-5,-3) arc (-90:270:.2 and 1);

\draw[densely dashed] (5,1) arc (270:90:.2 and 1);
\draw[] (5,1) arc (-90:90:.2 and 1);
\draw[densely dashed] (5,-3) arc (270:90:.2 and 1);
\draw[] (5,-3) arc (-90:90:.2 and 1);

\begin{scope}
\draw[rounded corners=20pt] (-5,3) -- (-4,3)--(-3,2)--(3,2) --(4,3) -- (5,3);
\end{scope}

\begin{scope}
\draw[rounded corners=20pt] (-5,-3) -- (-4,-3)--(-3,-2)--(3,-2) --(4,-3) -- (5,-3);
\end{scope}

\draw[rounded corners=24pt] (-5,1) -- (-4,1) --(-4,-1) -- (-5,-1);
\draw[rounded corners=24pt] (5,1) -- (4,1) --(4,-1) -- (5,-1);

\draw[densely dashed] (5,1) arc (270:90:.2 and 1);
\draw[] (5,1) arc (-90:90:.2 and 1);
\draw[densely dashed] (5,-3) arc (270:90:.2 and 1);
\draw[] (5,-3) arc (-90:90:.2 and 1);

\draw[blue,rounded corners=24pt] (-4.8,-2)--(-3.8,-2) -- (-2.2,-1)--(2.2,-1) --(4.3,-2) -- (5.16,-2);
\draw[blue,rounded corners=24pt] (-4.8,2)--(-3.8,2) -- (-2.2,1)--(2.2,1) --(4.3,2) -- (5.16,2);

\draw[red,rounded corners=12pt] (-4.85,-1.5) -- (-4,-1.1) --(-3,2) -- (-2.6,2);
\draw[red,densely dashed, rounded corners=12pt] (-2.6,2) -- (-2.2,2) --(-2,-2) -- (-1.6,-2);
\draw[red,rounded corners=6pt] (-1.6,-2) -- (-1.4,-2) --(-1.2,-0.9) -- (1.2,-0.9)--(1.4,-2)--(1.6,-2);

\draw[red,densely dashed, rounded corners=12pt] (2.6,2) -- (2.2,2) --(2,-2) -- (1.6,-2);

\draw[red,rounded corners=12pt] (5.15,-1.5) -- (4,-1.1) --(3,2) -- (2.6,2);

\draw[fill=black] (-3.13,1.58) circle (2pt) node[below left]{$x$};
\draw[fill=black] (3.16,1.47) circle (2pt) node[below right]{$y$};
\node[blue] at (-4.5,2.5) {$L_2$};
\node[blue] at (-4.5,-2.5) {$L_1$};
\node[red] at (-3,-0.5) {$\phi^1_H(L_1)$};
\end{tikzpicture}

\caption{Hamiltonian flow of a gradient trajectory near joining parts}
\label{fig:ham_flow_fin_cylinder}
\end{figure}

Now we find a subcomplex of Floer cochain complex between two objects which is generated by intersections in some compact set and quasi-isomorphic to the original complex.

\begin{figure} \centering

\begin{tikzpicture}


\draw[] (-4,3) arc (0:360:0.5 and 0.1);
\draw[] (5,3) arc (0:360:.5 and .1);

\draw[] (-6,-1) arc (-90:270:.2 and 1.5);

\draw[densely dashed] (6,-1) arc (270:90:.2 and 1.5);
\draw[] (6,-1) arc (-90:90:.2 and 1.5);


\draw[densely dashed] (0.5,-2) arc (0:180:0.5 and 0.1);

\draw[] (-0.5,-2) arc (180:360:0.5 and 0.1);

\draw[rounded corners=10pt] (-6,2) -- (-5.5,2)--(-5,2.5)--(-5,3);
\draw[rounded corners=10pt] (-4,3) -- (-4,2.5)--(-3.5,2)--(3.5,2)--(4,2.5)--(4,3);
\draw[rounded corners=10pt] (6,2) -- (5.5,2)--(5,2.5)--(5,3);

\draw[rounded corners=10pt] (-6,-1) -- (-1,-1)--(-0.5,-1.5)--(-0.5,-2);
\draw[rounded corners=10pt] (6,-1) -- (1,-1)--(0.5,-1.5)--(0.5,-2);

\draw[blue] (-5.8,0)--(6.16,0);
\draw[blue] (-5.8,1)--(6.16,1);

\draw[red,rounded corners=12pt] (-5.85,1.5)--(-5.4,1.5) -- (-4.5,-0.1) --(-3.5,1.5)
--(-1,1.5)--(0,-0.1)--(1,1.5)--(3.5,1.5)--(4.5,-0.1)--(5.5,1.5)--(6.1,1.5);

\draw[fill=black] (-5.15,1) circle (2pt);
\draw[fill=black] (-3.8,1) circle (2pt) node[below right]{$x_1$};
\draw[fill=black] (-0.7,1) circle (2pt) node[below left]{$y_1$};

\draw[fill=black] (3.8,1) circle (2pt) node[below left]{$y_2$};
\draw[fill=black] (0.7,1) circle (2pt) node[below right]{$x_2$};

\draw[fill=black] (5.15,1) circle (2pt) node[below right]{$x_{3}$};

\node[blue] at (-1.5,.5) {$L_2$};
\node[blue] at (-1.5,-.5) {$L_1$};
\node[red] at (0,1.5) {$\phi^1_H(L_1)$};
\end{tikzpicture}

\caption{Near limit point of $B(M)$}
\label{fig:near_twoside_lim_pt}
\end{figure}
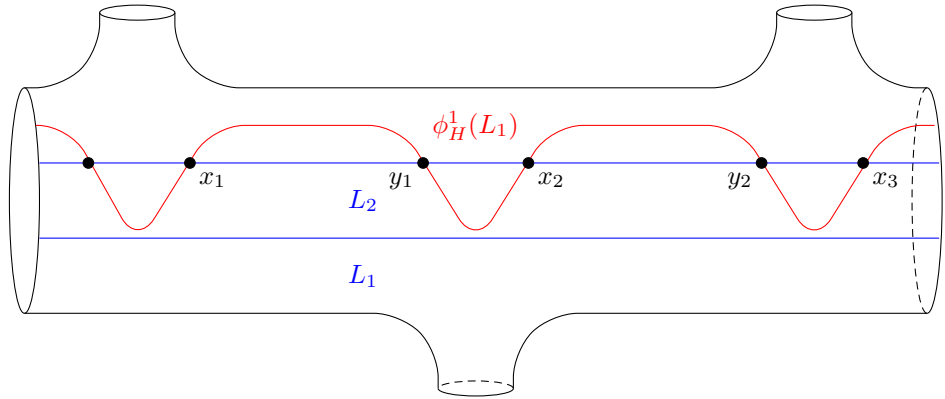

\begin{figure} \centering

\begin{tikzpicture}

\draw[] (-6,-2) arc (-90:270:.2 and 2);

\draw[densely dashed] (6,-2) arc (270:90:.2 and 2);
\draw[] (6,-2) arc (-90:90:.2 and 2);

\draw[rounded corners=10pt] (-6,2) --(6,2);
\draw[rounded corners=10pt] (-6,-2) -- (6,-2);

\draw[blue] (-5.8,-0.7)--(6.16,-0.7);
\draw[blue] (-5.8,-1.5)--(6.16,-1.5);
\draw[blue] (-5.8,1)--(6.16,1);

\draw[rounded corners=12pt] (-4.5,0)--(-4,-0.3)--(-3,-0.4)--(-2,-0.3)--(-1.5,0);

\draw[rounded corners=10pt] (-4.3,-0.1)--(-4,0.2)--(-3,0.3)--(-2,0.2)--(-1.7,-0.1);

\draw[rounded corners=12pt] (4.5,0)--(4,-0.3)--(3,-0.4)--(2,-0.3)--(1.5,0);
\draw[rounded corners=10pt] (4.3,-0.1)--(4,0.2)--(3,0.3)--(2,0.2)--(1.7,-0.1);

\draw[red,rounded corners=12pt] (-5.85,1.5)--(-5.4,1.5) -- (-4.5,-1.6) --(-4,-0.5)
--(-2,-.5)--(-1.5,-1.6)--(-0.5,1.5)--(0.5,1.5)--(1.5,-1.6)--(2,-0.5)--(4,-0.5)--(4.5,-1.6)--(5.4,1.5)--(6.15,1.5);

\draw[fill=black] (-5.25,1) circle (2pt) node[below left]{$y_{0}$};
\draw[fill=black] (-.65,1) circle (2pt) node[below right]{$x_1$};

\draw[fill=black] (0.65,1) circle (2pt) node[below left]{$y_1$};

\draw[fill=black] (5.25,1) circle (2pt) node[below right]{$x_2$};

\draw[fill=black] (4.75,-0.7) circle (2pt);
\draw[fill=black] (4.05,-0.7) circle (2pt);
\draw[fill=black] (1.95,-0.7) circle (2pt);
\draw[fill=black] (1.2,-0.7) circle (2pt);
\draw[fill=black] (-1.95,-0.7) circle (2pt);
\draw[fill=black] (-1.2,-0.7) circle (2pt);
\draw[fill=black] (-4.05,-0.7) circle (2pt);
\draw[fill=black] (-4.75,-0.7) circle (2pt);

\node[blue] at (-1.5,.7) {$L_2$};
\node[blue] at (0,-.4) {$L_2'$};
\node[blue] at (0,-1.2) {$L_1$};
\node[red] at (1.2,1.5) {$\phi^1_H(L_1)$};
\end{tikzpicture}

\caption{Near nonplanar isolated point of $B(M)$}
\label{fig:near_isolated_nonplanar_end}
\end{figure}
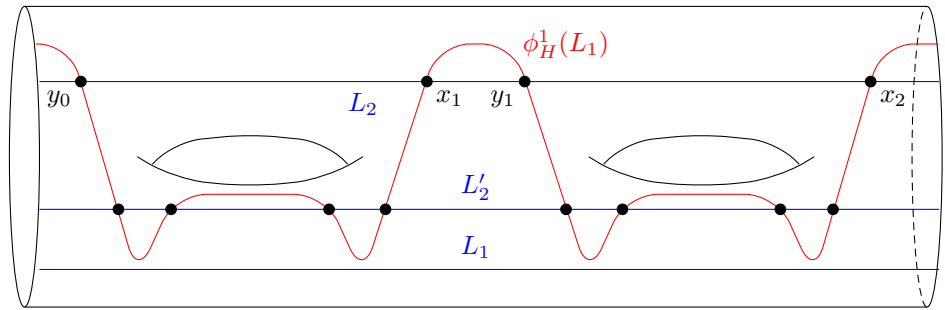

\begin{prop}\label{prop:mor_infinite_surfaces}
Let $L_1,\,L_2$ be gradient sectorial Lagrangians.
Then there exists $n\in \ZZ_{>0}$ such that
\begin{enumerate}
\item $Mor(L_1,L_2)\cong Mor_{Fuk(\widehat{M^{\leq n}})}(\rho_n(L_1),\rho_n(L_2))$, or
\item $Mor(L_1,L_2)$ is quasi-isomorphic to the subcomplex generated by intersections in $M^{\leq n}$.
\end{enumerate}
\end{prop}

\begin{proof}
Let gradient sectorial Lagrangians $L_1,L_2$ of $M$ be given.

Suppose $L_1,L_2$ are open and there exist at least one $p\in B(M)$ such that $p$ is an ideal boundary of both $L_1$ and $L_2$ and one of the following hold:
\begin{enumerate}
\item $p$ is a limit point in $B(M)$.
\item $p$ is an isolated point in $B(M)$ and $p\in B'(M)$, gradient trajectories corresponding to $L_1$ and $L_2$ are in the same Liouville isomorphism class.
\item $p$ is an isolated point in $B(M)$ and $p\in B'(M)$, gradient trajectories corresponding to $L_1$ and $L_2$ are not in the same Liouville isomorphism class.
\item None of the above holds.
\end{enumerate}

Generators of Floer cochain complex between two gradient trajectories would be the following.
For case (1), see Figure \ref{fig:near_twoside_lim_pt}.
In each finite cylinder, intersection points in $\phi^1_H(L_1)\cap L_2$ comes in pair so that
we may label them as $x_i,y_i$, $i\in \ZZ_{>0}$.
Then we can check that $\delta x_i=0$ and $\delta y_i=x_i+x_{i+1}$ for all $i$, and
$$
\langle x_i \rangle_{i\in \ZZ_{>0}}/\langle x_i+x_{i+1} \rangle_{i\in \ZZ_{>0}}=\langle x_1,x_1+x_2,\ldots,x_{k-1}+x_k\rangle
$$
holds for any $k$.
We can find a $n$ so that $M^{\leq n}$ contains $x_1,y_1,\ldots,y_{k-1},x_k$ for some $k$ and intersection points contained in $M^{\leq n}$ gives the same cohomology as intersection points
in whole surface.

For Case (2), see $L_1$ and $L_2'$ in Figure \ref{fig:near_isolated_nonplanar_end}.
It is same as Case (1).

For Case (3), see $L_1$ and $L_2$ in Figure \ref{fig:near_isolated_nonplanar_end}.
In each finite cylinder where $L_1$ and $L_2$ both intersect, intersection points in
$\phi^1_H(L_1)\cap L_2$ comes in pair so that we may label them as $x_i,y_i$, $i\in \ZZ_{>0}$.
Then we can check that $\delta x_i=0$ and $\delta y_i=x_i$ for all $i$.
Then boundary is same as cycle and these $x_i$ and $y_i$ do not contribute to cohomology.
Therefore, we can find a $n$ so that $M^{\leq n}$ contains $x_1,y_1,\ldots,y_{k-1},x_k$ for some $k$ and intersection points contained in $M^{\leq n}$ gives the same cohomology as intersection points
in whole surface.

Case (4) holds when $L_1$ and $L_2$ do not have the same ideal boundary point so that
every intersection point in $\phi^1_H(L_1)\cap L_2$ is contained in some compact set or
they meet at cylindrical ideal boundary points. In both cases, we have
$$
Mor(L_1,L_2)=Mor_{Fuk(\widehat{M^{\leq n}})}(\rho_n(L_1),\rho_n(L_2)).
$$
\end{proof}

\subsection{Restriction functor diagram and inverse limit}

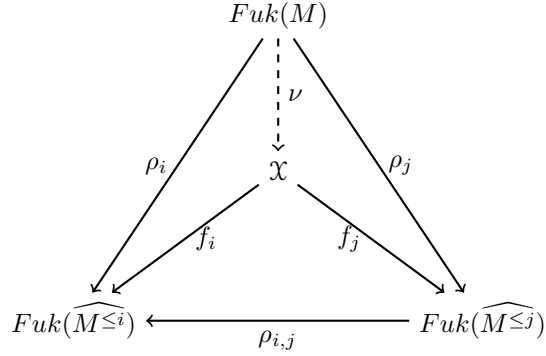
\begin{figure}
\begin{tikzpicture}
    [
        node distance=1.5cm, thick,
    ]
     \node (Fx)                                                  {$Fuk(\widehat{M^{\leq i}})$};
     \node (Fy)[xshift=2cm,right=of Fx]                          {$Fuk(\widehat{M^{\leq j}})$};
     \node (L) [yshift=2cm] at (barycentric cs:Fx=0.1,Fy=0.1)    {$\mathcal{X}$};
     \node (N) [above=of L]                                      {$Fuk(M)$};

    \draw[->,dashed] (N) -- (L)  node [right,midway] {$\nu$};
    \draw[->] (L)  -- (Fx) node [right,midway]  {$f_i$};
    \draw[->] (L)  -- (Fy) node [left,midway] {$f_j$};
    \draw[->] (Fy) -- (Fx) node [below,midway] {$\rho_{i,j}$};
    \draw[->] (N) -- node [left,midway] {$\rho_i$} (Fx) ;
    \draw[->] (N) -- node [right,midway] {$\rho_j$} (Fy) ;
\end{tikzpicture}
\caption{Restriction functor diagram}
\label{fig:restriction_diagram}
\end{figure}

Let $M$ be a standard surface.
Throughout this section, we may assume that every gradient sectorial Lagrangian submanifold
$L\in Ob(Fuk(M))$
is invariant under gradient flow of $\psi$ near $\partial M^{\leq n}$ for every $n\in \ZZ_{>0}$.
Then we can define Viterbo restriction functor $\rho_n:Fuk(M)\to Fuk(\widehat{M^{\leq n}})$ for
every $n$.
For $0<i\leq j$, $M^{\leq i}$ is a Liouville subdomain of $M^{\leq j}$ and we can also define
the Viterbo restriction functor $\rho_{i,j}:Fuk(\widehat{M^{\leq j}})\to Fuk(\widehat{M^{\leq i}})$.
Then the following hold:

\begin{enumerate}
\item $\rho_{i,i}$ is an identity for any $i\geq 1$.
\item $\rho_{i,j}\circ \rho_{j,k}=\rho_{i,k}$ for any $i\leq j\leq k$.
\item $\rho_{i,j}\circ \rho_{j}=\rho_{i}$ for any $i\leq j$.
\end{enumerate}

By items (1) and (2), the pair $(Fuk(\widehat{M}^i)_{i\geq 1},(\rho_{i,j})_{1\leq i\leq j})$ is an inverse diagram and has the inverse limit $\mathcal{X}$ with projections $f_n:\mathcal{X}\to Fuk(\widehat{M}^n)$.
By item (3), we have an $A_\infty$-functor $\nu:Fuk(M) \to \mathcal{X}$ and the diagram in
Figure \ref{fig:restriction_diagram} intertwines.

\begin{theorem}
The $A_\infty$-functor $\nu:Fuk(M) \to \mathcal{X}$ is a quasi-equivalence if and only if $M$ is of finite type.
\end{theorem}
\begin{proof}
If $M$ is of finite type, there exist some $n\in \ZZ_{>0}$ so that $\rho_n:M \to M^{\leq n}$ is
a quasi-equivalence and we have a finite inverse diagram
$$
\left(Fuk(\widehat{M}^i)_{1\leq i\leq n},(\rho_{i,j})_{1\leq i\leq j\leq n}\right).
$$
Inverse limit of this diagram is $Fuk(\widehat{M^{\leq n}})$ and we are done.

For the infinite type case, we will find an object $X$ of $\mathcal{X}$ so that there is no
element $L$ of $Fuk(M)$ with $\nu(L)=X$.
Suppose $M$ is of infinite type. Then $M$ has at least one noncylindrical end $p$ and we can
pick any $q\in B(M)$ so that there is an open gradient-sectorial Lagrangian $L_{p,q}$.
By Proposition \ref{prop:grad-sect_lag_ends}, there exists $n\in \ZZ_{>0}$ such that
connected components of $L_{p,q}\setminus M^{\leq n}$ are gradient trajectories of $\psi$ which does not
intersect any ascending or descending manifold of a critical point in $M\setminus M^{\leq n}$.
We have a unique maximal chain $I_1\supset I_2 \supset \ldots$ so that
$I_1\in \cI^{(n)}$ and $\bigcap^\infty_{i=1}I_i=\{p\}$.
Each $B_{I_i}$ has the boundary component where $B_{I_i}$ is attached to.
Also, to define Viterbo restriction functor, we assumed that every boundary component of
$B_{I_i}$ should have a neighbourhood where every gradient sectorial Lagrangian submanifold
is invariant under the gradient flow of $\psi$.
We can find a simple closed curve $S_i$ contained in $B_{I_i}$ which satisfies the following:
\begin{enumerate}
\item $S_i$ is a connected component of regular level set of $\psi$.
\item $S_i$ is homeomorphic to the boundary component of $B_{i_i}$ where it is attached to.
\item $S_i$ is disjoint from the neighbourhood of boundary components of $B_{I_i}$ where every gradient sectorial Lagrangian submanifold is invariant under the gradient flow of $\psi$.
\end{enumerate}

$L_{p,q}$ should intersect every $B_{I_i}$ and $L_{p,q}\cap B_{I_i}$ is a gradient trajectory which is disjoint from any critical point in $B_I$, so $L$ should intersect $S_i$ by item (2).
Also, by item (1), $S_i$ transversally intersect $L_{p,q}$.
Let $L_1$ be an open Lagrangian which is obtained by applying Dehn twist about
$S_1$ to $L_{p,q}$.
For $i>1$, let $L_i$ be an open Lagrangian which is obtained by applying Dehn
twist about $S_i$ to $L_{i-1}$.
Then each $L_i$ is a gradient-sectorial open Lagrangian and by item (3) we can apply $\rho_{n+i}$ on each $L_i$ and get
\begin{equation}
\rho_{n+i,n+j}(\rho_{n+j}(L_j))=\rho_{n+i}(L_i)
\end{equation}
for all $i\leq j$. Therefore, there exists $X\in \mathcal{X}$ so that
$f_{n+i}(X)=\rho_{n+i}(L_i)$ for all $i$.

Suppose there exists an object $L$ in $Fuk(M)$ so that $\nu(L)=X$.
Then 
\begin{equation}
\rho_{n+i}(L)=f_{n+i}\circ \nu(L)= f_{n+i}(X)=\rho_{n+i}(L_i)
\end{equation}
for all $i$.
Note that $L$ is a gradient-sectorial Lagrangian and there exists $m\in \ZZ_{>0}$ so that 
$L$ is invariant under $\psi$-gradient flow on $M\setminus M^{\leq m}$.
However, $\rho_{n+i}(L)=\rho_{n+i}(L_i)$ for all $i$ and item (1) above implies $L$ is not invariant under $\psi$-gradient flow on a neighbourhood of each $S_i$.
$S_i$ is contained in $B_{I_i}$ and disjoint from $M^{\leq{n+i-1}}$. Therefore, for sufficiently large $i$, $S_i$ is contained in $M\setminus M^{\leq m}$ and it contradicts to gradient-sectoriality of $L$.
This argument holds for any finite connected sum of objects of $Fuk(M)$ 
since a finite connected sum of objects of $Fuk(M)$ also should be 
invariant under $\psi$-gradient flow out of some compact set and
we can always find $S_i$ disjoint from the given compact set.

This finishes the proof.
\end{proof}

Roughly speaking, $Fuk(M)$ is an $A_\infty$-category whose objects consist of
the elements which become stabilized
in sufficiently large Liouville subdomain.

\section{Future directions}

In this section, we describe some natural questions arising from our construction.

The first obvious problem is to understand the action of big mapping class group on
our Fukaya category $Fuk(M,\CT)$ and to find a new construction of
quasimorphisms on $\Symp_{QC}(M,\omega_\CT)$ as mentioned in the introduction of the present paper.
By construction, it is clear that $Fuk(M,\CT)$ depends on the structure of ideal boundary pair $(B(M),B'(M))$.
In particular, it will be another interesting problem to describe the quotients category
$$
Fuk(M,\CT)/ Fuk^{cpt}(M,\CT)
$$
in terms of a combinatorial category of the ideal boundary 
pair $(B(M), B'(M))$
such as the cluster category illustrated by
the case of plumbings of cotangent bundles \cite{bae-jeong-kim}.
This is a subject of future study.

Although we only cover the Fukaya category of Riemann surface without boundary, classification of Riemann surface has been extended to the case with boundary by Prishlyak and Mischenko in \cite{prish-misch}. Using this, we believe that our construction can be generalized to 
the case of  Liouville sectors as  Ganatra, Pardon and Shende
did in \cite{gps}. We also refer readers to \cite{auroux-smith}
for some description of Fukaya category of infinite type surface
in terms of  the filtration of finite-type submanifolds thereof.

Another interesting direction of research would be to relate our 
construction to surfaces of
infinite type that naturally arise in the mirror symmetry of elliptic curves and
the modular forms via the study of Fukaya category of the divisor complements.
(See \cite{nagano-ueda}, \cite{hashimoto-ueda} for example.)

Obviously we would like to extend our study to higher dimensions. In higher dimensions,
there is no classification results like the case of surfaces, but we can construct a
Fukaya category for any Weinstein manifolds of infinite type in the same way as
in the surface case adopting the collection of sectorial Lagrangians as its objects.
One might try to exploit the higher dimensional version of
pair of pants decomposition by G. Mikhalkin in \cite{Mikhalkin}
in the description of morphism spaces of the Fukaya category.

\bibliographystyle{amsalpha}

\bibliography{biblio-oh-choi}

\end{document}